\documentclass{llncs}
\usepackage{natbib} 
\usepackage{pstricks}
\usepackage{pstricks-add}
\usepackage{graphicx}
\usepackage{amsmath}
\usepackage{hyperref}
\usepackage{amssymb}
\usepackage{multirow}
\usepackage{booktabs}
\usepackage{tabularx}
\usepackage{bm}
\usepackage{dsfont}
\begin{document}
	\renewcommand{\thepage}{\arabic{page}}
	\pagestyle{plain} 
	\setcounter{page}{1}
	\title{Optimal Investment - Consumption - Insurance with Partial Information and Correlation Between Assets Price and Factor Process}
	\titlerunning{Hamiltonian Mechanics}  
	%
	\author{WOUNDJIAGUE Apollinaire\inst{1,4} \and Rodwell Kufakunesu \inst{3} \and Julius Esunge\inst{2}}
	\authorrunning{WOUNDJIAGU\'{E} Apollinaire et al.} 
	%
	\tocauthor{WOUNDJIAGU\'{E} Apollinaire, Rodwell Kufakunesu,  Julius Esunge}
	\institute{National Advanced School of Engineering of Maroua P.O. Box 58 Maroua, Cameroon,\\
		\email{appoappolinaire@yahoo.fr}
		\and
		University of Mary Washington, Fredericksburg, Virginia, USA
		\and
		Department of Mathematics and Applied Mathematics, University of Pretoria, 0002, South Africa
		\and
		National Advanced School of Engineering of Yaounde P.O. Box 8390 Yaounde, Cameroon}
	\maketitle   
	\begin{abstract}
		In this research, we present an analysis of the optimal investment, consumption, and life insurance acquisition problem for a wage earner with partial information. Our study considers the non-linear filter case where risky asset prices are correlated to the factor processes under constant relative risk aversion (CRRA) preferences. We introduce a more general framework with an incomplete market, random parameters adapted to the Brownian motion filtration, and a general factor process with a non-linear state estimation and a correlation between the state process (risky asset prices) and the factor process. To address the wage earner's problem, we formulate it as a stochastic control problem with partial information where the risky assets prices are correlated to the factor processes. Our framework is extensive since the non-linear filter applied to the linear case gives a more robust result than the Kalman filter. We obtain the non-linear filter through the Zakai equation and derive a system of the Hamilton-Jacobi-Bellman (HJB) equation and two backward stochastic differential equations (BSDE). We establish the existence and uniqueness of the solution, prove the verification theorem, and construct the optimal strategy. 
		
		\keywords{Assets, Factor processes, non linear filter, Zakai equation, HJB equation, verification theorem}
	\end{abstract}
	\section{Introduction}
	\par In the field of financial and actuarial mathematics, the optimization of investment portfolios is a topic of significant importance. The seminal work of \cite{merton1975optimum} introduced the optimal investment-consumption problem and derived the Hamilton-Jacobi-Bellman (HJB) equation using the dynamic programming approach.
	
	While the stochastic optimal control problem can be solved explicitly through dynamic programming, the stochastic maximum principle, or the convex duality martingale method, the literature reveals that \cite{merton1975optimum} did not consider uncertainty in investment horizon. This limitation was overcome by \cite{blanchet2008optimal}.
	
	\cite{pliska2007optimal} and \cite{ye2007optimal} have further explored the problem solved in \cite{blanchet2008optimal}, incorporating a life insurance purchase problem. Additionally, \cite{duarte2014optimal} extended the problem to multiple risky assets, while\cite{duarte2014optimal} and \cite{shen2016optimal} considered the problem with random parameters and solved it through a combination of a HJB equation and a backward stochastic differential equation (BSDE).
	\cite{hata2020optimal} extends the counterpart of \cite{shen2016optimal} with partial information.
	\par
	The study by \cite{hata2020optimal} investigated investment-consumption-insurance strategies based solely on past information of risky assets, without utilizing information from the factor process. However, this study assumed deterministic model parameters, whereas the consideration of random parameters such as income, interest rate, lifetime horizon, etc. is more realistic. Furthermore, we propose a more generalized setting where \cite{hata2020optimal}'s study is a special case. Specifically, we suggest a non-linear filtering setting where the prices of risky assets are correlated with the factor process.
	\par
	The present paper investigates an optimal investment-consumption-life insurance problem faced by a wage earner under partial information, in which the Kalman filter is non-linear and the prices of risky assets are correlated to the factor process. The concept of partial information, as introduced in \cite{hata2020optimal}, implies that the investment-consumption-insurance strategies are formulated based solely on past information pertaining to risky assets, without utilizing information related to factor processes. This restriction is imposed since the wage earner may not always have complete access to all the factor processes. To enhance the practicality of this model, we expand upon this framework by considering:
	\begin{itemize}
		\item The use of random processes to accurately describe interest rates, volatility, force of mortality, premium-insurance ratios, income, and discount rates.
		\item The incorporation of a general factor process that allows for non-linear state estimation.
		\item The establishment of a correlation between risky asset prices processes and factor processes that can provide valuable insights into the dynamics.
	\end{itemize}
	Thus, the model described in \cite{hata2020optimal} can be seen as a special case of our model.
	\par
	We assume that:
	\begin{itemize}
		\item The wage earner is faced with decision-making regarding consumption, investment, and life insurance during the time interval $[0, T\wedge \xi ]$, where $T$ represents the deterministic retirement time of the wage earner and $\xi$ represents the positive random variable of their time of death.
		\item The wage earner's instantaneous income is a random process ${R(t) \mid t\in [0, T\wedge \xi ]}$ that is adapted to the Brownian motion filtration.
		\item The wage earner purchases a life insurance policy at a premium rate, which is a random process ${\beta(t) \mid t\in [0, T\wedge \xi ]}$ that is adapted to the Brownian filtration. In the event of the wage earner's death, the insurance company pays the beneficiary an amount equal to $\dfrac{\beta (t)}{a(t)}$, where $a(t)$ is the insurance premium-payout that is predetermined by the insurer.
		\item The wage earner seeks to maximize their satisfaction from a consumption process ${C(t) \mid t \in [0, T\wedge \xi] }$, that is, adapted to the Brownian filtration.
		\item The wage earner's preference is described by the utility function of the constant relative risk aversion (CRRA) type.
		\item The wage earner can invest their savings in a financial market described by a non-negative risk-free interest rate $r(t)$, which is a stochastic process adapted to the Brownian filtration, a $k-$dimensional standard Brownian motion, a risk-free stock $S_{0t}$, and a price process of the risky stocks $(S_{1t}, \ldots , S_{kt})^{T}$.
	\end{itemize}
	The rest of the paper is structured as follows: Section two begins with a description of the financial market, followed by an overview of the life insurance market, and concludes with the presentation of the stochastic control problem. In section three, the aforementioned problem is solved, and the verification theorem is subsequently formulated and proved in the following section. Finally, the paper is concluded in the last section.
	\section{The Problem Statement}
	We consider an optimal investment, consumption and life insurance purchase problem for a wage earner with a partial information in the case of non linear filter where the risky assets prices are correlated to the factor process. Indeed we extend \citep{hata2020optimal} by considering parameters as random processes and considering general factor process models, that is the returns and volatility of the assets are random and are affected by some economic factors.The factor process noise correlates to that of the assets price. Then the market is in general incomplete. \par
	Let us first describe the financial market, the insurance market and the wealth process.
	\subsection{Description of the Financial Market}
	Let $\big (\Omega ,  \mathfrak{F}, P_{r}\big)$ be a probability space equipped with the continuous time filtration $\mathbb{F} : = \big (\mathfrak{F}_{t}^{W} \big)_{t \in [0, T]}$ .\par
	The financial market is described by the risk-free bond and $k$ risky assets such that:
	\begin{equation*}
		\left\{
		\begin{array}{ll}
			d S_{0t} = r(Z_{t}, t) S_{0t} dt \\
			S_{00} = s_{00}
		\end{array}
		\right.
	\end{equation*}
	\begin{equation*}
		\left\{
		\begin{array}{ll}
			d S_{it} = S_{it} \bigg \{\eta_{i} (Z_{t},  t)dt + \sum\limits_{j=1}^{N} \Sigma_{ij}(Z_t , t) d W_{jt} + \sum\limits_{j=1}^{N} \alpha_{ij}(t) dB_{jt} \bigg \}  \\
			S_{i0} = s_{i0}, \quad i = 1, \ldots , k
		\end{array}
		\right.
	\end{equation*}
	
	\begin{equation}\label{1}
		\left\{
		\begin{array}{ll}
			d Z_{t} = h(Z_{t} , t) dt + \gamma (Z_{t} , t) d B_{t}  \\
			Z_{0} = z \sim \mathcal{N}(z_{0} , P_{0}), \quad z_{0} \in \mathbb{R}^{m},
		\end{array}
		\right.
	\end{equation}
	where:
	\begin{itemize}
		\item[$\bullet$]
		$B_{t} = (B_{it})_{i = 1, \ldots , N}$ is an $N-$dimensional standard Brownian motion defined on $\big (\Omega ,  \mathfrak{F}, P_{r}\big)$ with a covariance matrix $\mathbb{C}_{B}$;
		\item[$\bullet$]
		$W_{t} = (W_{jt})_{j = 1, \ldots , N}$ is an $N-$dimensional standard Brownian motion defined on $\big (\Omega ,  \mathfrak{F}, P_{r}\big)$ with a covariance matrix $\mathbb{C}_{W}$;
		\item[$\bullet$]
		$B_{t}$ and $W_{t}$ are independent; 
		\item[$\bullet$]
		$Z_{t}$ is the $m-$dimensional stochastic factor process which affects the risk-free interest rate $r(Z_{t}, t)$, the drift $\eta(Z_{t}, t)$ of $S_{t}$, and the volatility matrix $\Sigma (Z_{t}, t)$ of $S_{t}$;
		\item[$\bullet$]
		$\Sigma , \gamma $ are $k \times N$ and $m \times N$ matrix - valued functions respectively;
		\item[$\bullet$]
		$\eta$ and $h$ are $\mathbb{R}^{k} -$valued and $\mathbb{R}^{m}-$valued functions respectively.
	\end{itemize}
	\subsubsection*{Assumptions:}
	\begin{enumerate}
		\item[\textbf{H1)}] $\gamma , h , \Sigma, \eta$ and $r$ are Lipschitz and smooth;
		\item[\textbf{H2)}] $\exists \; \delta_{1}, \delta_{2} > 0$ such that; 
		\begin{equation*}
			\Sigma (t) \Sigma^{T}(t) \geqslant \delta_{1} I_{m\times m}
		\end{equation*} 
		and 
		\begin{equation*}
			\gamma(t) \gamma^{T}(t) \geqslant \delta_{2} I_{m\times m};
		\end{equation*}
		\item[\textbf{H3)}]
		$r$ is smooth, non-negative, and bounded with bounded derivatives; 
		\item[\textbf{H4)}]
		$\big [\alpha_{ij}(t)\big ]^{1 \leqslant i \leqslant k }_{1 \leqslant j \leqslant N}$ is a deterministic matrix in the observation process;
		\item[\textbf{H5)}]
		$(W_{t})_{t \in [0, T]}$ and $(B_{t})_{t \in [0, T]}$ have covariance matrix $\mathbb{C}_{W}$ and $\mathbb{C}_{B}$ respectively that are positive definite.
	\end{enumerate}
	\subsection{Description of the life insurance market}
	We assume that the wage earner is alive at time $t=0$ and has a life time $\xi$, which is a random variable defined on the probability space $\big (\Omega ,  \mathfrak{F}, P_{r}\big)$. The random variable $\xi$ is supposed to be independent of $\mathbb{F}$ and has a distribution function defined by :
	\begin{equation}\label{2}
		G_{\xi}(t) = P_{r}\big (\xi \leqslant t \mid \mathcal{H}_{t} \big) = 1 - e^{-\int_{0}^{t} \mu(s) ds}, \; t\in [0 ,  T],
	\end{equation} 
	where
	\begin{equation*}
		\mathcal{H}_{t} = \sigma (S_{u} ; u \leqslant t).
	\end{equation*}
	$\overline{G}_{\xi}$ is the conditional survival probability of the wage earner alive at $t$ and is defined as follows:
	\begin{equation}\label{3}
		\overline{G}_{\xi}(t) \longmapsto P_{r} \big (\xi \geqslant t \mid \mathcal{H}_{t} \big ) = e^{- \int_{0}^{t}\mu(s) ds}.
	\end{equation}
	The conditional instantaneous death rate for the wage earner surviving to time $t$ is defined as:
	\begin{align*}
		\mu : & [0,T) \longrightarrow \mathbb{R}^{+}\\
		& t \longmapsto \lim_{\tau t \to 0} \dfrac{ P_{r} \big (t \leqslant \xi < t + \tau t  \mid \xi \geqslant t  \big) }{\tau t},
	\end{align*}
	i.e.,
	\begin{align*}
		\mu (t) & = \lim_{\tau t \to 0} \dfrac{P_{r} \big (t \leqslant \xi < t + \tau t \big )}{\tau t P_{r} \big (\xi \geqslant t \big)}\\
		& = \lim_{\tau t \to 0} \dfrac{G_{\xi}(t + \tau t) - G_{\xi}(t) }{\tau t \overline{G}_{\xi}(t)}\\
		& = \dfrac{g_{\xi}(t) }{\overline{G}_{\xi}(t)}\\
		& = - \dfrac{\partial }{\partial t} \big (\ln \big (\overline{G}_{\xi} (t) \big) \big) \Rightarrow \overline{G}_{\xi} (t) = e^{- \int_{0}^{t} \mu (s)ds}.
	\end{align*}
	Here, $\mu$ is a continuous and deterministic function such that 
	\begin{equation*}
		\int_{0}^{+ \infty} \mu(t) dt = + \infty.
	\end{equation*}
	The conditional probability density of the death for the wage earner at $t$ is defined by 
	\begin{equation*}
		g_{\xi}(t) = \mu (t) e^{-\int_{0}^{t} \mu(s)ds}, \quad \forall t \in [0, T].
	\end{equation*}
	Assuming the wage earner purchases a life insurance product in the event of premature death before their retirement time $T$, we further assume that the wage earner pays a premium insurance rate $\beta (t)$. If the claim occurs at a time $\xi < T$, the insurer will pay $\dfrac{\beta (\xi) }{a(\xi)}$ to the beneficiary of the contract. Here, $a(t)$ refers to a continuous and deterministic function on the interval $[0, T]$ known as the insurance premium-payout ratio. The total legacy for a death that occurs at time $t$ is therefore given by:
	\begin{equation*}
		\mathbf{v} (\xi) = \left\{
		\begin{array}{ll}
			X(\xi) + \dfrac{\beta (\xi) }{a(\xi)} & \text{if} \; \xi < T \\
			X(T) & \text{if} \; \xi \geqslant T,
		\end{array}
		\right.
	\end{equation*} 
	where $X(t)$ represents the wealth process of the wage earner at time $t$. \par
	Indeed, if $\xi \geqslant T$, then the wage earner's payment ends at $T$, and he has no need of life insurance. Thus $\beta(T) = 0$.
	\subsection{The wealth process}
	For $j = 0, 1, \ldots, k $ and $t \in [0, \xi \wedge T]$, let $\rho_{j}(t)$ be the amount of the wage earner's wealth allocated to $S_{j}(t)$:
	\begin{equation*}
		\rho (t) = \big (\rho_{1}(t), \rho_{2}(t), \ldots , \rho_{k}(t) \big)^{T}, \; t \in [0, \xi \wedge T], \; \sum\limits_{j=1}^{k}\rho_{j}(t) = 1.
	\end{equation*}
	We assume that 
	\begin{equation*}
		\rho(t), C(t), \beta (t), \quad t \in [0, T]
	\end{equation*}
	are $\mathcal{H}_{t}-$progressively measurable stochastic processes such that :
	\begin{equation*}
		\int_{0}^{T} \bigg (|| \rho (t)||^2 + |C(t)| + |\beta (t)| \bigg) dt < + \infty\quad P_{r} - a.s.
	\end{equation*} 
	By the self-financing condition, the wage's earner wealth process $X_{t}$, starting with initial capital $x$, satisfies the following dynamics:
	\begin{equation*}
		\left\{
		\begin{array}{ll}
			d X_{t} = \big (X_{t} - \rho^{T}(t) \mathbf{1}\big)\dfrac{ d S_{0t} }{S_{0t}} +  \sum\limits_{j=1}^{k}\rho_{j}(t) \dfrac{d S_{jt}}{S_{jt}} - C(t) dt - \beta(t) dt + R(t) dt \\
			X(0) = x,
		\end{array}
		\right.
	\end{equation*} 
	where $ \mathbf{1} = (1, \ldots , 1)^{T}$. \par
	Thus
	\begin{equation*}
		\left\{
		\begin{array}{ll}
			d X_{t} = \big (X_{t} - \rho^{T}(t) \mathbf{1}\big) r (Z_{t} , t) dt +  \sum\limits_{i=1}^{k}\rho_{i}(t) \bigg [\eta_{i}(Z_t , t) dt + \sum\limits_{j=1}^{N} \Sigma_{ij}(Z_t , t) d W_{jt} + \sum\limits_{j=1}^{N} \alpha_{ij}(t) d B_{jt}   \bigg] \\
			- C(t) dt - \beta(t) dt + R(t) dt \\
			X(0) = x.
		\end{array}
		\right.
	\end{equation*} 
	In matrix form, we have:
	\begin{equation}\label{4}
		\left\{
		\begin{array}{ll}
			d X_{t} = \big (X_{t} - \rho^{T}(t) \mathbf{1}\big) r (Z_{t} , t) dt +  \rho^{T}(t) \bigg [\eta (Z_t , t) dt + \Sigma (Z_t , t) d W_{t} + \alpha(t) d B_{t} \bigg] - C(t) dt - \beta(t) dt \\
			+ R(t) dt \\
			X(0) = x, \quad t \in [0, \xi \wedge T],
		\end{array}
		\right.
	\end{equation}
	where:
	\begin{equation*}
		\eta (Z_t , t) = \bigg (\eta_{1}(Z_{t}, t), \ldots , \eta_{k}(Z_t , t)  \bigg)^{T}
	\end{equation*}
	and 
	\begin{equation*}
		\Sigma (Z_t , t) = \begin{pmatrix}
			\Sigma_{11} (Z_t , t) & \Sigma_{12} (Z_t , t) & \ldots & \Sigma_{1N} (Z_t , t)\\
			\Sigma_{21} (Z_t , t) & \Sigma_{22} (Z_t , t) & \ldots & \Sigma_{2N} (Z_t , t) \\
			\vdots & \vdots & \ldots & \vdots \\
			\Sigma_{k1} (Z_t , t) & \Sigma_{k2} (Z_t , t) & \ldots & \Sigma_{kN} (Z_t , t)
		\end{pmatrix}.
	\end{equation*}
	\begin{equation}\label{5}
		\eqref{4} \Rightarrow \left\{
		\begin{array}{ll}
			d X_{t}  & = \bigg \{ X_{t} r (Z_{t}, t ) + \rho^{T}(t) \big (\eta(Z_t , t) - r(Z_t , t) \mathbf{1}\big) - C_{t} - \beta(t) + R(t)\bigg \} dt \\
			& + \rho^{T}(t) \Sigma (Z_t , t) d W_{t} + \rho^{T}(t)\alpha (t) d B_{t}\\
			X(0) & = x, \quad t \in [0, \xi \wedge T]
		\end{array}
		\right.
	\end{equation}
	\subsection{The stochastic control problem}
	The problem is to find the strategies $\rho(t), C(t), \beta (t)$ that maximize the expected utility of the wage earner obtained from his consumption for all $t \in [0, \xi \wedge T]$, his terminal wealth and the value of his legacy.\par
	Let $\mathcal{L}_{T}(x,0)$ be the set of admissible control. Then the wage earner's problem is to find the strategy $\big (\rho_{t}, C_{t}, \beta_{t} \big) \in \mathcal{L}_{T}(x,0)$ that maximize:\\
	$\varphi(0, x, z)$
	\begin{equation}\label{6}
		= \sup_{ (\rho_{t}, C_{t}, \beta_{t} ) \in \mathcal{L}_{T}(x,0)} \mathbb{E}_{x} \bigg [\int_{0}^{\xi \wedge T}  e^{- \int_{0}^{t} \theta (s) ds}V(C_{t}) \; + \; e^{-\int_{0}^{\xi} \theta (t) dt } V(\mathbf{v}(\xi))\mathds{1}_{[\xi \leqslant T]} \; + \; e^{- \int_{0}^{T} \theta (t) dt} V(X(T)) \mathds{1}_{[\xi > T]} \bigg],
	\end{equation}
	where:
	\begin{itemize}
		\item[$\bullet$]
		$\theta(t)$ is the discount rate process which is an $\mathbb{R}^{+}-$valued and $\mathbb{F}-$adapted process;
		\item[$\bullet$]
		$V$ is a power utility function defined by
		\begin{equation}\label{7}
			V(y) =  \left\{
			\begin{array}{ll}
				\dfrac{y^{\delta}}{\delta} & \text{if} \; y \geqslant 0\\
				- \infty & \text{if} \; y < 0, \quad \forall \delta \in (-\infty , 0) \cup (0, 1) .
			\end{array}
			\right.
		\end{equation}
		\item[$\bullet$]
		$\mathbb{E}_{x}$ is the expectation operator conditional on $X(0) = x$.
	\end{itemize}
	Equation \eqref{6} can be seen as a stochastic control problem with random parameters for general factor models, where the factor process $Z_{t}$ is the state process and the price process $S_{t}$ of the risky stock is the observation process and are supposed to be correlated.
	\section{Solution of Problem \eqref{6}}
	Let us consider the following transformation:
	\begin{equation*}
		Q_{it} := \ln S_{it}, \quad i = 1, \ldots , k.
	\end{equation*}
	By Itô's lemma, $Q_{it}$ is the solution of 
	\begin{align*}
		&d Q_{it}\\ & = d \ln S_{it} \\
		& = \dfrac{1}{S_{it}} d S_{it} - \dfrac{1}{2 S^{2}_{it} } \big (d S_{it} \big)^2 \\
		& = \bigg \{\eta_{i}(Z_t , t) - \dfrac{1}{2} \bigg [\sum\limits_{j =1 }^{N} \Sigma_{ij}(Z_t , t) \Sigma_{ji} (Z_t , t)  + \sum\limits_{j=1}^{N} \alpha_{ij}(t) \alpha_{ji}(t) \bigg]   \bigg \} dt + \sum\limits_{j=1}^{N} \Sigma_{ij}(Z_t , t) d W_{jt} + \sum\limits_{j=1}^{N} \alpha_{ij}(t) d B_{jt},\\
		&\; i = 1, \ldots , k.
	\end{align*}
	In a matrix form, we have
	\begin{equation}\label{8}
		\left\{
		\begin{array}{ll}
			d Q_{t} & =[\eta (Z_t , t) - \dfrac{1}{2}\bigg (\Sigma(Z_t ,  t) \Sigma^{T}(Z_t , t)  + \alpha (t) \alpha^{T}(t) \bigg) ] dt + \Sigma (Z_t , t) d W_{t}  + \alpha(t) d B_{t}\\
			Q_{0} & = \ln s_{0}.
		\end{array}
		\right.
	\end{equation}
	Let us consider the non-linear filter defined by the conditional probability 
	\begin{equation}\label{9}
		\hat{Z}^{\psi}_{t} = \mathbb{E} \bigg [\psi (Z_{t}) \mid \mathcal{H}_{t} \bigg], \quad \forall t \in [0, T],
	\end{equation}
	where $\psi$ is a Borel bounded function on $\mathbb{R}^{m}$ \cite{bensoussan1982lectures}. Moreover, the non linear filter can be written as:
	\begin{equation}\label{10}
		\hat{Z}^{\psi}_{t} = \dfrac{ p (t, \psi) }{p (t, 1) }, \forall t \in [0, T] ,
	\end{equation}
	where $p(t, \psi)$ is called the unnormalized conditional probability which is characterized as the solution PDE known as the Kushner - Stratonovitch equation. We need to transform the process $(Q_{t})_{t \in [0, T]}$ into a Weiner process. For that, let us introduce the process $(\Gamma (t))_{t \in [0, T]}$ defined by 
	\begin{equation*}
		\left\{
		\begin{array}{ll}
			d \Gamma (t) & = - \Gamma (t) \ell^{T}(Z_t) \big (\alpha_{t} \mathbb{C}_{B} \alpha^{T}_{t} + \Sigma_{t} \mathbb{C}_{W} \Sigma^{T}_{t} \big)^{-1} \big (\alpha_{t} d W_{t}  + \Sigma_{t} d B_{t}\big) \\
			\Gamma (0) & = 1,
		\end{array}
		\right.
	\end{equation*}  
	where 
	\begin{eqnarray*}
		\ell(Z_t)& = &\eta (Z_t , t) - \dfrac{1}{2} \bigg (\Sigma (Z_t , t) \Sigma^{T}(Z_t , t) + \alpha_{t}\alpha^{T}_{t}\bigg).
	\end{eqnarray*}
	Explicitly,
	\begin{eqnarray*}
		\Gamma (t)& =& \exp \bigg [- \int_{0}^{t} \ell^{T}(Z_{s}) \big (\alpha_{s} \mathbb{C}_{B} \alpha^{T}_{s} + \Sigma_{s} \mathbb{C}_{W} \Sigma^{T}_{s}\big)^{-1} \big (\alpha_{s} d W_{s} + \Sigma_{s} d B_{s} \big)\\
		&&- \dfrac{1}{2} \int_{0}^{t} \ell^{T}(Z_s) \big (\alpha_{s} \mathbb{C}_{B}\alpha^{T}_{s} + \Sigma_{s} \mathbb{C}_{W}\Sigma^{T}_{s}  \big)^{-1} \ell(Z_s) ds \bigg].
	\end{eqnarray*}
	Consider the change of probability
	\begin{align*}
		\dfrac{d \tilde{P}_{r} }{d P_{r}} \bigg |_{\mathcal{F}^{t}} & = \Gamma (t) \\
		\dfrac{d P_{r} }{d \tilde{P}_{r}} \bigg |_{\mathcal{F}^{t}} & = \Lambda (t) = \dfrac{1}{\Gamma (t)},
	\end{align*}
	where
	\begin{equation*}
		\mathcal{F}^{t} = \sigma \big (W_s , B_s , s \leqslant t \big).
	\end{equation*}
	Let us then consider the process
	\begin{equation*}
		\tilde{W}_{t} = W_{t} + \int_{0}^{t} \mathbb{C}_{B} \alpha^{T}_{s} \big (\alpha_{s}\mathbb{C}_{B}\alpha^{T}_{s} + \Sigma_{s} \mathbb{C}_{W}\Sigma_{s}^{T} \big )^{-1} \ell (Z_{s}) ds. 
	\end{equation*}
	We have the following lemmas:
	\begin{lemma}\label{lem1}
		For the filtered probability space $(\Omega, \mathfrak{F}, \tilde{P}_{r}, \mathcal{F}^{t})$, the process $\big (\tilde{W}_{t} \big)_{t \in [0, T]}$ and $\big (Q_{t} \big)_{t \in [0, T]}$ are $\mathcal{F}^{t} - $Weiner process with covariance matrices $\mathbb{C}_{B}$ and $\alpha_{t} \mathbb{C}_{B}\alpha^{T}_{t} + \Sigma_{t} \mathbb{C}_{W} \Sigma_{t}^{T} $ respectively.
	\end{lemma}
	\begin{proof}
		See \cite{bonsoussan1992stochastic}
		$\square$
	\end{proof}
	The so called unnormalized conditional probability and the non linear filter can be written as:
	\begin{equation*}
		\hat{Z}_{t}^{\psi} = \dfrac{\tilde{\mathbb{E}} \bigg [\psi (Z_t) \Lambda (t) \big | \mathcal{H}_{t} \bigg]}{\tilde{ \mathbb{E}} \bigg [\Lambda (t) \big | \mathcal{H}_{t} \bigg]} = \dfrac{p (t, \psi)}{ p (t, 1)}
	\end{equation*}
	and 
	\begin{equation*}
		p(t, \psi) = \tilde{\mathbb{E}} \bigg [\psi (Z_t) \Lambda (t) \big | \mathcal{H}_{t} \bigg]
	\end{equation*}
	respectively, where $\tilde{\mathbb{E}} \bigg [\cdot\big | \mathcal{H}_{t} \bigg]$ is the conditional operator expectation under $\tilde{P}_{r}$. This leads us to the following proposition:
	\begin{proposition}
		We assume that:
		\begin{itemize}
			\item[$\bullet$]
			$h (x, t)$ and $\gamma (x, t)$ are Borel functions such that:
			\begin{align*}
				& \big | h(x_1 , t) - h(x_2 , t) \big | \leqslant k |x_1 - x_2|\\
				& \big | \big | \gamma (x_1 , t) - \gamma(x_2 , t) \big | \big | \leqslant k |x_1 - x_2|, \quad k\in \mathbb{R}, x_1 , x_2 \in \mathbb{R}^{m}
			\end{align*}
			$h (0, t)$ and $\gamma (0, t)$ are bounded and take $\ell (x, t)$ such that $|\ell (x, t) | \leqslant k (1 + |x|)$.
			\item[$\bullet$]
			$\big (\alpha_{ij}(t) \big)^{1 \leqslant i \leqslant K}_{1 \leqslant j \leqslant N}$ is a deterministic matrix
			\item[$\bullet$]
			$\psi (x, t) \in C^{2,1} \big (\mathbb{R}^{m} \times [0, + \infty) \big)$ is a Borel bounded function 
		\end{itemize}
		Let us set 
		\begin{equation*}
			a(x, t) = \dfrac{1}{2} \gamma (x, t) \mathbb{C}_{B} \gamma^{T}(x, t)
		\end{equation*}
		and 
		\begin{align*}
			\mathcal{A}(t) & = - f^{T} \triangledown - tr \big (a \triangledown^2 \big) \\
			& = - \sum\limits_{i} f_{i} \dfrac{\partial }{\partial x_i} - \sum\limits_{i,j} a_{ij} \dfrac{\partial^2}{\partial x_{i} \partial x_{j}}
		\end{align*}
		Then\\
		$p(t, \psi)$\\
		\begin{equation}\label{11}
			= \hat{Z}^{\psi}_{0} + \int_{0}^{t} p \big (s, \dfrac{\partial \psi}{\partial s} - \mathcal{A}(s) \psi (s) \big) ds + \int_{0}^{t} p \big (s, \ell^{T}(s,\psi (s)) + \triangledown \psi^{T}(s) \gamma (s) \mathbb{C}_{W}\alpha^{T}_{s}  \big) \big (\alpha \mathbb{C}_{B} \alpha^{T} + \Sigma \mathbb{C}_{W} \Sigma^{T} \big)^{-1} d Q_{s}
		\end{equation}
		a.s.
	\end{proposition}
	\begin{proof}
		See \cite{bonsoussan1992stochastic}
		$\square$
	\end{proof}
	Now, we look for the  explicit solution of the Zakai equation in the linear case through theorem \ref{th1}:
	\begin{theorem}\label{th1}
		We consider the linear case where 
		\begin{align*}
			h(x, t) & = D(t) x + d (t) \\
			\gamma (x, t) & = \gamma (t) \\
			\eta (x, t) & = U(t) x + u(t)\\
			\Sigma (x, t) & = \Sigma(t).
		\end{align*}
		$D(t),  d(t), U(t), u(t)$ are deterministic and bounded. Then 
		\begin{enumerate}
			\item[\textbf{1)}]
			\eqref{1} et \eqref{8} become:
			\begin{equation*}
				\left\{
				\begin{array}{ll}
					d Z_{t} & = \big (D(t) Z_{t} + d(t) \big) dt + \gamma d B_{t} \\
					Z_{0} & = z \in \mathbb{R}^{m}
				\end{array}
				\right.
			\end{equation*}
			and 
			\begin{equation*}
				\left\{
				\begin{array}{ll}
					d Q_{t} & = \big (U(t) Z_{t} + u^{1}(t) \big) dt + \alpha(t) d B_{t} + \Sigma_{t} d W_{t} \\
					Q_{0} & = \ln s_{0} \in \mathbb{R}^{k}
				\end{array}
				\right.
			\end{equation*}
			respectively, where 
			\begin{equation*}
				u^{1}(t) = u(t) - \dfrac{1}{2} \bigg (\Sigma(t) \Sigma^{T}(t) + \alpha(t)\alpha^{T}(t) \bigg).
			\end{equation*}
			\item[\textbf{2)}]
			The unique solution of the Zakai equation \eqref{11} is given by:
			\begin{equation}\label{12}
				\hat{p}\big (t, \psi \big) = \bigg [\int \psi \big (\hat{z}^{\psi}_{t} + P^{1/2}_{t} x \big ) \dfrac{e^{-(1/2) x^2} }{(2N)^{N/2}} dx \bigg] s_{t},
			\end{equation}
			where $P(t) = \mathbb{E} \bigg [(Z_t - \hat{Z}^{\psi}_{t} ) (Z_t - \hat{Z}^{\psi}_{t} )^{T} \bigg]$ is the solution of the Riccati equation
			\begin{equation}\label{13}
				\left\{
				\begin{array}{ll}
					& \dot{P} - \big (PU^{T} + \mathbb{C}_{B}\alpha^{T} \big) \big (\alpha \mathbb{C}_{B} \alpha + \Sigma \mathbb{C}_{W} \Sigma^{T} \big)^{-1} \big (UP + \alpha \mathbb{C}_{B} \big) + \big (PU^{T} + \mathbb{C}_{B} \alpha^{T} \big) \big (\alpha \mathbb{C}_{B} \alpha^{T} + \Sigma \mathbb{C}_{W} \Sigma^{T} \big )^{-1}\\
					&  \alpha \mathbb{C}_{B}\gamma^{T} + \gamma \mathbb{C}_{B}\alpha^{T} \big (\alpha \mathbb{C}_{B}\alpha^{T} + \Sigma \mathbb{C}_{W} \Sigma^{T}  \big )^{-1} \big (UP + \alpha \mathbb{C}_{B} \big) - \gamma \mathbb{C}_{B}\gamma^{T} - DP - PD^{T} = 0\\
					& P(0) = P_{0}.  
				\end{array}
				\right.
			\end{equation}
			The Kalman filter $\hat{Z}^{\psi}_{t}$ is given by the equation 
			\begin{equation}\label{14}
				\left\{
				\begin{array}{ll}
					d \hat{Z}^{\psi}_{t} & = \big (D  \hat{Z}^{\psi}_{t} + d \big) dt + \big (PU^{T} + \mathbb{C}_{B} \alpha^{T} \big) \big (\alpha \mathbb{C}_{B} \alpha^{T} + \Sigma \mathbb{C}_{W}\Sigma^{T} \big)^{-1} d \nu_{t}\\
					\hat{Z}^{\psi}(0) & = z_{0} 
				\end{array}
				\right.
			\end{equation}
			with 
			\begin{equation}\label{141}
				d \nu_{t} = d Q_{t} - \big (U \hat{Z}^{\psi}_{t} + u^{1} \big)
			\end{equation}
			and the process $s_{t}$ is given by:
			\begin{eqnarray*}
				s_{t} &= &\exp \bigg [\int_{0}^{t}\big ( \hat{Z}^{\psi^T}_{s} U^{T} + u^{1^T} \big)\big (\alpha \mathbb{C}_{B} \alpha^{T} + \Sigma \mathbb{C}_{W} \Sigma^{T} \big)^{-1}d Q\\
				&&- \dfrac{1}{2}  \int_{0}^{t} \big ( \hat{Z}^{\psi^T}_{s} U^{T} + u^{1^T} \big)\big (\alpha \mathbb{C}_{B} \alpha^{T} + \Sigma \mathbb{C}_{W} \Sigma^{T} \big)^{-1} \big (U \hat{Z}^{\psi}_{s}  + u^{1} \big)ds \bigg]
			\end{eqnarray*}
		\end{enumerate}
	\end{theorem}
	To prove theorem \ref{th1}, we need to state and prove the lemma \ref{lm2}.
	\begin{lemma}\label{lm2}
		Consider the filtered probability space $\big (\Omega , \mathcal{F} , P_{r} , \mathcal{F}^{t} \big)$ and the processes defined in the linear case as in theorem \ref{th1}.\par
		Let $\big (s^{1}_{t} \big)_{t \in [0, T]}$ be the process defined by:
		\begin{align*}
			s^{1}_{t} & = \exp \bigg [- \int_{0}^{t}\big ( \hat{Z}^{\psi^T}_{s} U^{T} + u^{1^T} \big)\big (\alpha \mathbb{C}_{B} \alpha^{T} + \Sigma \mathbb{C}_{W} \Sigma^{T} \big)^{-1}\big (\alpha_{s} d B_{s} + \Sigma_{s} d W_{s} \big) \\
			& - \dfrac{1}{2}  \int_{0}^{t} \big ( \hat{Z}^{\psi^T}_{s} U^{T} + u^{1^T} \big)\big (\alpha \mathbb{C}_{B} \alpha^{T} + \Sigma \mathbb{C}_{W} \Sigma^{T} \big)^{-1} \big (U \hat{Z}^{\psi}_{s}  + u^{1} \big)ds \bigg].
		\end{align*}
		Then one has 
		\begin{equation*}
			\mathbb{E} \big [s^{1}_{t} \big] = 1.
		\end{equation*}
	\end{lemma}
	\begin{proof}
		Let us first check that
		\begin{equation*}
			\mathbb{E} \bigg [s^{1}_{t} | \hat{Z}^{\psi}_{t}|^2 \bigg] < C.
		\end{equation*}
		By Itô's formula, we have 
		\begin{align*}
			&d \big (|\hat{Z}^{\psi}_{t} |^2 \big)\\  & = 2  \hat{Z}^{\psi^T}_{t}\cdot d  \hat{Z}^{\psi}_{t} + tr \bigg (d  \hat{Z}^{\psi}_{t} \cdot d \hat{Z}^{\psi^T}_{t} \bigg) \\
			& = 2  \hat{Z}^{\psi^T}_{t} \bigg [\big (D  \hat{Z}^{\psi}_{t} + d \big)dt + \big (PU^{T} + \mathbb{C}_{B} \alpha^T \big)\big (\alpha \mathbb{C}_{B} \alpha^T + \Sigma \mathbb{C}_{W} \Sigma^{T} \big)^{-1} \big (\alpha d B_{t} + \Sigma d W_{t}\big)   \bigg]\\
			& + tr \bigg \{ \bigg [  \big (D  \hat{Z}^{\psi}_{t} + d \big)dt + \big (PU^{T} + \mathbb{C}_{B} \alpha^T \big)\big (\alpha \mathbb{C}_{B} \alpha^T + \Sigma \mathbb{C}_{W} \Sigma^{T} \big)^{-1} \big (\alpha d B_{t} + \Sigma d W_{t}\big)   \bigg] \bigg [ \\
			&  \big (D^{T}\hat{Z}^{\psi^T}_{t} + d^T \big )dt +  \big (d B^{T}_{t}\cdot \alpha^T + d W^{T}_{t} \cdot \Sigma^T \big) \big (\alpha \mathbb{C}_{B}\alpha^T + \Sigma \mathbb{C}_{W} \Sigma^T \big )^{-1} \big ( UP^T + \alpha \mathbb{C}_{B} \big) \bigg ]\bigg \}\\
			& = 2 \hat{Z}^{\psi}_{t} \bigg [\big (D \hat{Z}^{\psi}_{t} + d \big) dt +  \big (PU^T + \mathbb{C}_{B} \alpha^T \big)\big (\alpha \mathbb{C}_{B} \alpha^T + \Sigma \mathbb{C}_{W} \Sigma^T \big)^{-1} \big (\alpha d B_{t} + \Sigma d W_{t} \big) \bigg] + \\
			& tr \bigg [\big (PU^T + \mathbb{C}_{B} \alpha^T \big) \big (\alpha \mathbb{C}_{B}\alpha^T + \Sigma \mathbb{C}_{W} \Sigma^T \big)^{-1} \big (\alpha \mathbb{C}_{B}\alpha^T + \Sigma \mathbb{C}_{W} \Sigma^T \big)\big (\alpha \mathbb{C}_{B}\alpha^T + \Sigma \mathbb{C}_{W} \Sigma^T \big)^{-1} \\
			&\big (UP^T + \alpha \mathbb{C}_{B} \big) \bigg] dt \\
			& = 2 \hat{Z}^{\psi^T}_{t} \bigg [\big (D \hat{Z}^{\psi^T}_{t} + d \big)dt + \sigma \big (\alpha d B_{t} + \Sigma d W_{t} \big) \bigg] + tr \big (\sigma M \sigma^T \big) dt ,
		\end{align*}
		where 
		\begin{align*}
			\sigma & = \big (PU^T + \mathbb{C}_{B} \alpha^T \big) \big (\alpha \mathbb{C}_{B} \alpha^T + \Sigma \mathbb{C}_{W}\Sigma^T \big)^{-1}\\
			M & = \alpha \mathbb{C}_{B}\alpha^{T} + \Sigma \mathbb{C}_{W} \Sigma^{T}.
		\end{align*}
		But
		\begin{equation*}
			d s^{1}_{t} = - \big (\hat{Z}^{\psi^T}_{t} U^{T} + u^{1^T} \big) \big (\alpha \mathbb{C}_{B} \alpha^T + \Sigma \mathbb{C}_{B} \Sigma^{T} \big) \big (\alpha s^{1}_{t} d B_{t} + \Sigma s^{1}_{t} d W_{t}\big).
		\end{equation*}
		Hence
		\begin{align*}
			&d \big (s^{1}_{t} | \hat{Z}^{\psi}_{t}|^2 \big)\\  = & \big (d s^{1}_{t}\big) | \hat{Z}^{\psi}_{t}|^2 + s^{1}_{t} \cdot d (| \hat{Z}^{\psi}_{t}|^2) + \big (d s^{1}_{t} \big)\cdot \big (d \big(| \hat{Z}^{\psi}_{t}|^2 \big) \big)\\
			= & -  s^{1}_{t} | \hat{Z}^{\psi}_{t}|^2 \big (\hat{Z}^{\psi^T}_{t} U^T + u^{1^T} \big)\big (\alpha \mathbb{C}_{B} \alpha^T + \Sigma \mathbb{C}_{W} \Sigma^T \big)^{-1} \big (\alpha d B_{t} + \Sigma d W_{t} \big)\\
			& + 2 s^{1}_{t} \hat{Z}^{\psi^T}_{t} \bigg [\big (D \hat{Z}^{\psi}_{t} + d \big)dt + \sigma \big (\alpha d B_{t} + \Sigma d W_{t} \big) \bigg] +\\
			& s^{1}_{t} tr \big (\sigma M \sigma^T \big)dt - 2 s^{1}_{t} \big (\hat{Z}^{\psi^T}_{t} U^T + u^{1^T} \big)\big (\alpha \mathbb{C}_{B}\alpha^T + \Sigma \mathbb{C}_{W} \Sigma^T \big)^{-1} \hat{Z}^{\psi^T}_{t}\sigma \big (|\alpha|^2 + |\Sigma|^2 \big) dt\\
			& = - s^{1}_{t} | \hat{Z}^{\psi}_{t}|^{2} \big (\hat{Z}^{\psi^T}_{t} U^{T} + u^{1^T} \big)\big (\alpha \mathbb{C}_{B} \alpha^T + \Sigma \mathbb{C}_{W} \Sigma^T \big)^{-1} \big (\alpha d B_{t} + \Sigma d W_{t} \big) + \\
			& s^{1}_{t} \bigg [2 \hat{Z}^{\psi^T}_{t} \big (D \hat{Z}^{\psi}_{t}  +  d\big) + tr (\sigma M \sigma^T) \bigg] dt + 2 s^{1}_{t} \hat{Z}^{\psi^T}_{t} \sigma \big (\alpha d B_t + \Sigma d W_t \big)\\
			& -2 s^{1}_{t} \big (\hat{Z}^{\psi^T}_{t} U^{T} + u^{1^T} \big) \big (\alpha \mathbb{C}_{B} \alpha^{-1} + \Sigma \mathbb{C}_{W} \Sigma^{T} \big)^{-1} \hat{Z}^{\psi^T}_{t} \sigma \big (| \alpha|^2 + |\Sigma|^2 \big) dt \\
			& = s^{1}_{t} \bigg [- |\hat{Z}^{\psi^T}_{t}|^2 \big (\hat{Z}^{\psi^T}_{t} U^T + u^{1^T} \big) \big (\alpha \mathbb{C}_{B} \alpha^T + \Sigma \mathbb{C}_{W}\Sigma^{T} \big)^{-1} + 2 \hat{Z}^{\psi^T}_{t} \sigma\bigg] \big (\alpha d B_{t} + \Sigma d W_{t} \big) + \\
			& s^{1}_{t} \bigg [2 \hat{Z}^{\psi^T}_{t} \big (D \hat{Z}^{\psi}_{t}  + d\big) + tr (\sigma M \sigma^T) - 2 \big (\hat{Z}^{\psi^T}_{t} U^{T} + u^{1^T}\big)\big (\alpha \mathbb{C}_{B} \alpha^T + \Sigma \mathbb{C}_{B}\Sigma^{T} \big)^{-1} \hat{Z}^{\psi^T}_{t} \sigma \big (|\alpha|^2 + |\Sigma|^2 \big) \bigg] dt.
		\end{align*}
		Now, $\forall \epsilon > 0$, we have
		\begin{align*}
			&d \bigg (\dfrac{s^{1}_{t} | \hat{Z}^{\psi}_{t}|^2 }{1 + \epsilon s^{1}_{t} |\hat{Z}^{\psi}_{t}|^2 } \bigg)\\ & =\dfrac{1}{\bigg (1 + \epsilon s^{1}_{t}|\hat{Z}^{\psi}_{t}|^2 \bigg)^2} \bigg \{  \bigg \{   s^{1}_{t} \bigg [ -|\hat{Z}^{\psi}_{t}|^{2} \big (\hat{Z}^{\psi^T}_{t} U^T + u^{1^T} \big) \big (\alpha \mathbb{C}_{B}\alpha^T + \Sigma \mathbb{C}_{W} \Sigma^{T} \big)^{-1} + 2 \hat{Z}^{\psi^T}_{t} \sigma  \bigg]\big (\alpha d B_{t} + \Sigma d W_{t} \big) \\
			& + s^{1}_{t} \bigg [2 \hat{Z}^{\psi^T}_{t} ( D \hat{Z}^{\psi}_{t} + d) + tr (\sigma M \sigma^T) - 2 \big (\hat{Z}^{\psi^T}_{t} U^T + u^{1^T} \big) \big (\alpha \mathbb{C}_{B} \alpha^T + \Sigma \mathbb{C}_{B} \Sigma^T \big)^{-1} \hat{Z}^{\psi^T}_{t} \sigma \big (|\alpha|^2 + |\Sigma|^2 \big) \bigg] dt \bigg \} \\
			&  \big (1 + \epsilon s^{1}_{t} |\hat{Z}^{\psi}_{t}|^2 \big) - \epsilon \bigg \{ s^{1}_{t} \bigg [- |\hat{Z}^{\psi}_{t}|^2 \big (\hat{Z}^{\psi^T}_{t} U^T + u^{1^T} \big) \big (\alpha \mathbb{C}_{B} \alpha^T + \Sigma \mathbb{C}_{B} \Sigma^T \big)^{-1}  + 2 \hat{Z}^{\psi^T}_{t} \sigma \bigg] \big (\alpha d B_{t} + \Sigma d W_{t}\big) \\
			& + s^{1}_{t} \bigg [2 \hat{Z}^{\psi^T}_{t} \big (D \hat{Z}^{\psi}_{t}  + d\big) + tr (\sigma M \sigma^T) - 2 \big (\hat{Z}^{\psi^T}_{t} U^{T} + u^{1^T} \big) \big (\alpha \mathbb{C}_{B} \alpha^T + \Sigma \mathbb{C}_{W} \Sigma^T \big)^{-1} \hat{Z}^{\psi^T}_{t} \sigma \big (|\alpha|^2 +\\
			& |\Sigma|^2 \big)  \bigg] dt \bigg \}  s^{1}_{t} |\hat{Z}^{\psi}_{t}|^2 \bigg \}.
		\end{align*}
		$\Rightarrow$ Integrating between $0$ and $t$, we take expectation, which yields 
		\begin{align*}
			& \mathbb{E} \bigg [\int_{0}^{t} d \bigg (\dfrac{s^{1}_{t} | \hat{Z}^{\psi}_{s}|^2 }{1 + \epsilon s^{1}_{s} |\hat{Z}^{\psi}_{s}|^2 } \bigg)  \bigg] \leqslant \mathbb{E} \bigg [ \int_{0}^{t}  \dfrac{   s^{1}_{s}  \big (2\hat{Z}^{\psi^T}_{s} ( D \hat{Z}^{\psi}_{s}  + d) +tr (\sigma M \sigma^T)   \big) }{ 1 + \epsilon s^{1}_{s} |\hat{Z}^{\psi}_{s}|^2 } ds\bigg]\\
			& \mathbb{E} \bigg (\dfrac{ s^{1}_{t} |\hat{Z}^{\psi}_{t}|^2}{1 + \epsilon s^{1}_{t} |\hat{Z}^{\psi}_{t}|^2}  \bigg) - \dfrac{ |z^{\psi}_{0}|^2}{1 + \epsilon |z^{\psi}_{0}|^2} \leqslant  \mathbb{E} \bigg [ \int_{0}^{t}  \dfrac{   s^{1}_{s}  \big (2\hat{Z}^{\psi^T}_{s} ( D \hat{Z}^{\psi}_{s}  + d) +tr (\sigma M \sigma^T)   \big) }{ 1 + \epsilon s^{1}_{s} |\hat{Z}^{\psi}_{s}|^2 } ds\bigg]\\
			& \dfrac{d}{dt} \mathbb{E} \bigg (\dfrac{ s^{1}_{t} |\hat{Z}^{\psi}_{t}|^2}{1 + \epsilon s^{1}_{t} |\hat{Z}^{\psi}_{t}|^2}  \bigg) \leqslant \mathbb{E} \bigg [ \dfrac{ s^{1}_{t}  \big (2\hat{Z}^{\psi^T}_{t} ( D \hat{Z}^{\psi}_{t}  + d) +tr (\sigma M \sigma^T)   \big) }{ 1 + \epsilon s^{1}_{t} |\hat{Z}^{\psi}_{t}|^2 }  \bigg].
		\end{align*}
		Since $\mathbb{E} (s^{1}_{t}) \leqslant 1$ \citep{bonsoussan1992stochastic}, It follows that
		\begin{align*}
			& \dfrac{d}{dt} \mathbb{E} \bigg (\dfrac{s^{1}_{t} |\hat{Z}^{\psi}_{t}|^2  }{1 + \epsilon s^{1}_{t} |\hat{Z}^{\psi}_{t}|^2 }\bigg) \leqslant K_{1} \bigg (\mathbb{E} \bigg (\dfrac{s^{1}_{t} |\hat{Z}^{\psi}_{t}|^2  }{1 + \epsilon s^{1}_{t} |\hat{Z}^{\psi}_{t}|^2 }\bigg) + 1\bigg) \\
			\Rightarrow & \mathbb{E} \bigg (\dfrac{s^{1}_{t} |\hat{Z}^{\psi}_{t}|^2  }{1 + \epsilon s^{1}_{t} |\hat{Z}^{\psi}_{t}|^2 }\bigg) \leqslant C, \quad \forall t \in [0, T].
		\end{align*}
		Hence 
		\begin{align*}
			& \mathbb{E} \bigg [\lim_{\epsilon \to 0} \dfrac{s^{1}_{t} |\hat{Z}^{\psi}_{t}|^2  }{1 + \epsilon s^{1}_{t} |\hat{Z}^{\psi}_{t}|^2 } \bigg] \leqslant \lim_{\epsilon \to 0} \mathbb{E} \bigg (\dfrac{s^{1}_{t} |\hat{Z}^{\psi}_{t}|^2  }{1 + \epsilon s^{1}_{t} |\hat{Z}^{\psi}_{t}|^2 }\bigg) \leqslant C \\
			\Rightarrow & \mathbb{E} \big (s^{1}_{t} | \hat{Z}^{\psi}_{t} |^2 \big) \leqslant C, \quad \forall t \in [0, T].
		\end{align*}
		Next, we have 
		\begin{align*}
			d \bigg (\dfrac{s^{1}_{t} }{1 + \epsilon s^{1}_{t}} \bigg) & = \dfrac{ (d s^{1}_{t} ) \big (1 + \epsilon s^{1}_{t} \big) - s^{1}_{t} d \big (1 + \epsilon s^{1}_{t} \big)}{\big (1 + \epsilon s^{1}_{t} \big)^2}\\
			& = \dfrac{ d s^{1}_{t}  + \epsilon s^{1}_{t} d s^{1}_{t} - \epsilon s^{1}_{t} d s^{1}_{t} }{ \big (1 + \epsilon s^{1}_{t} \big)^2} \\
			& = \dfrac{ d s^{1}_{t} }{\big (1 + \epsilon s^{1}_{t} \big)^2}
		\end{align*}
		\begin{align*}
			& \int_{0}^{t} d \bigg (\dfrac{s^{1}_{u} }{1 + \epsilon s^{1}_{u}} \bigg)  = \int_{0}^{t} \dfrac{d s^{1}_{u} }{\big (1 + \epsilon s^{1}_{u} \big)^2 } \\
			& \dfrac{s^{1}_{t}}{ 1 + \epsilon s^{1}_{t} } - \dfrac{1}{1 + \epsilon} = - \int_{0}^{t} \dfrac{ \big (\hat{Z}^{\psi^T}_{u} U^T + u^{1^T} \big)\big (\alpha \mathbb{C}_{B} \alpha^{T} + \Sigma \mathbb{C}_{W} \Sigma^T \big)^{-1} \big ( \alpha d B_u + \Sigma d W_u\big)}{ \big (1 + \epsilon s^{1}_{u} \big)^2}\\
			\Rightarrow & \mathbb{E} \bigg (\dfrac{s^{1}_{t}}{1 + \epsilon s^{1}_{t}} \bigg) = \dfrac{1}{1 + \epsilon}. 
		\end{align*}
		Since $\mathbb{E} (s^{1}_{t}) < 1, $ then 
		\begin{equation*}
			\lim_{\epsilon \to 0} \mathbb{E} \bigg (\dfrac{s^{1}_{t}}{1 + \epsilon s^{1}_{t}} \bigg)= \mathbb{E}(s^{1}_{t}) = 1.
		\end{equation*}
		Then necessary,  we have $\mathbb{E} (s^{1}_{t}) = 1 \quad \forall t \in [0, 1]$. 
		$\square$
	\end{proof}
	\begin{proof} \textit{of theorem \ref{th1}}\\
		Let us first check that $\tilde{ \mathbb{E}} (s_{t}) = 1.$\par
		For that, we define a new probability measure $\tilde{P}_{r}$, by setting 
		\begin{equation*}
			\dfrac{d \tilde{P}_r}{d P_r} \bigg |_{\mathcal{F}^t} = s^{1}_{t}.
		\end{equation*}
		For the filtered probability space $\big (\Omega, \mathcal{F}, \tilde{P}_{r}, \mathcal{F}^{t} \big)$, the process $( \tilde{W}_{t})_{t \in [0, T]}$ and $(Q_{t})_{t \in [0, T]}$ are $\mathcal{F}^{t}-$Wiener process with $\mathbb{C}_{B}$ and $\alpha \mathbb{C}_{B} \alpha^{T} + \Sigma \mathbb{C}_{W} \Sigma^T$ as covariance matrices respectively (lemma \ref{lem1}). \par
		Let us set
		\begin{equation*}
			s_{t} = \dfrac{1}{s^{1}_{t}}.
		\end{equation*}
		Then 
		\begin{equation*}
			\tilde{\mathbb{E}} (s_{t}) = \mathbb{E} \big (s^{1}_{t} s_{t} \big) = 1
		\end{equation*}
		and 
		\begin{equation*}
			\dfrac{ d P_r}{d \tilde{P}_r} \bigg |_{\mathcal{F}^t} = s_{t}.
		\end{equation*}
		Next, let us set
		\begin{equation*}
			\Pi_{t} = P^{1/2}_{t}
		\end{equation*}
		and assume that $P_{t}$ is invertible and that $\Pi_{t}$ is differentiable and $\dot{\Pi}_{t}$ is the solution of the Lyapunov equation
		\begin{equation*}
			\dot{P}_{t} = \Pi_{t} \dot{\Pi}_{t} + \dot{\Pi}_{t}.
		\end{equation*} 
		It follows that 
		\begin{align*}
			&d \bigg [\psi \big (\hat{Z}^{\psi}_{t} + P^{1/2}_{t} x , t   \big)s_{t} \bigg] \\ = & s_{t} \bigg [\dfrac{\partial \psi}{\partial t} + \nabla \psi^{T} \big (D \hat{Z}^{\psi}_{t} + d \big) + \nabla \psi^T \dot{\Pi}_{t} x + \dfrac{1}{2} tr \nabla^{2} \psi \big (PU^T + \mathbb{C}_{B} \alpha^T \big) \big (\alpha \mathbb{C}_{B}\alpha^T  + \Sigma \mathbb{C}_{W} \Sigma^T  \big)^{-1} \\ 
			&\big ( UP + \alpha \mathbb{C}_{B} \big)  \bigg] dt  + s_{t} \bigg [\nabla \psi^{T}\big (PU^T + \mathbb{C}_{B} \alpha^T \big) + \psi \big (\hat{Z}^{\psi^T}_{t} U^T  + u^{1^T} \big) \bigg] \big (\alpha \mathbb{C}_{B} \alpha^T + \Sigma \mathbb{C}_{W} \Sigma^T \big)^{-1} d Q_{t},
		\end{align*}
		where the space argument of $\dfrac{\partial \psi}{\partial t} , \nabla \psi , \nabla^{2} \psi$ on the right hand side is evaluated at $\hat{Z}^{\psi}_{t} + P^{1/2}_{t} x$. \par
		We notice that:
		\begin{equation*}
			\int \nabla \psi^{T} \big (\hat{Z}^{\psi}_{t} + P^{1/2}_{t} x , t \big) \dot{\Pi}_{t} x e^{-(1/2) x^2} dx = \dfrac{1}{2} \int tr \nabla^{2} \psi \big (\hat{Z}^{\psi}_{t} + P^{1/2}_{t} x , t \big) \dot{P}_{t} e^{-(1/2)x^2 } dx.
		\end{equation*}
		Hence
		\begin{eqnarray*}
			d p(t, \psi) & = &s_{t} \int \bigg [\dfrac{\partial \psi}{\partial t} + \nabla \psi^{T} \big (D\hat{Z}^{\psi}_{t} + d  \big) + \dfrac{1}{2} tr \nabla^{2} \psi\\
			&&\bigg( \dot{P}_{t} +  \big(PU^T   + \mathbb{C}_{B} \alpha^2 \big) \big (\alpha \mathbb{C}_{B} \alpha^T + \Sigma \mathbb{C}_{W} \Sigma^T  \big)^{-1} \big (UP + \alpha \mathbb{C}_{B}  \big) \bigg )\bigg] \\
			&&  \dfrac{e^{-(1/2) x^2} }{(2 N)^{N/2}} dx  + s_{t} \bigg [\int \bigg (\nabla \psi^T \big (PU^T  + \mathbb{C}_{B} \alpha^T \big) + \psi \big (\hat{Z}^{\psi}_{t} U^T + u^{1^T} \big)  \bigg) \dfrac{e^{ -(1/2)x^2 }}{(2N)^{N/2}} dx  \bigg] \\
			&&
			\big (\alpha \mathbb{C}_{B} \alpha^T + \Sigma \mathbb{C}_{W} \Sigma^T \big) d Q_{t}.
		\end{eqnarray*}
		From \eqref{13}, we notice that 
		\begin{align}\nonumber\label{15}
			& \dfrac{1}{2} \int tr \nabla^{2} \psi \bigg (\dot{P}_{t} + \big (P U^T + \mathbb{C}_{B} \alpha^T \big) \big (\alpha \mathbb{C}_{B} \alpha^T + \Sigma \mathbb{C}_{W} \Sigma^T \big)^{-1} \big (UP + \alpha \mathbb{C}_{B} \big) \bigg) e^{-(1/2) x^2} dx \\\nonumber
			&  = \dfrac{1}{2} \int tr \nabla^{2} \psi \big (\mathbb{C}_{B} + DP + PD^T \big) e^{-(1/2)x^2 } dx \\
			& = \int \bigg (\dfrac{1}{2} tr \nabla^2 \psi \mathbb{C}_{B} + \nabla \psi^{T}D P^{1/2} x \bigg)e^{-(1/2)x^2} dx,
		\end{align}
		Since 
		\begin{equation}\label{16}
			\int \nabla \psi^{T} \big (PU^{T} + \mathbb{C}_{B} \alpha^T \big) e^{-(1/2)x^2} dx = \int \psi x^T P^{1/2}U^{T} e^{-(1/2)x^2} dx
		\end{equation}
		From \eqref{15} and \eqref{16}, we can evaluate the right hand side of \eqref{14} and obtain 
		\begin{eqnarray*}
			d p (t, \psi)& = &p(t, \psi) \bigg [\dfrac{\partial \psi}{\partial t} + \nabla \psi^{T} \big (D Z_{t} + d \big) + \dfrac{1}{2} tr \nabla^2 \psi \mathbb{C}_{B}\bigg] dt\\
			&&+ p(t, \psi) \psi \big (x^{T}U^{T} + u^{1^T} \big) \big (\alpha \mathbb{C}_{B} \alpha^T + \Sigma \mathbb{C}_{W}\Sigma^T \big)^{-1} d Q_{t}.
		\end{eqnarray*}
		$\square$
	\end{proof}
	Using equation \eqref{12}, equation \eqref{10} becomes
	\begin{equation}\label{17}
		\hat{Z}^{\psi}_{t} = \dfrac{\hat{p}(t, \psi)}{\hat{p}(t, 1)}, \quad t \in [0,T].
	\end{equation} 
	Moreover by equation \eqref{141}, we have the dynamic of the risky asset given by
	\begin{equation}\label{18}
		\left\{
		\begin{array}{ll}
			d Q_{t} & = \bigg (U(t) \hat{Z}^{\psi}_{t} + u^{1}(t) \bigg) dt + d \nu\\
			Q_{0} & = \ln s_{0} \in \mathbb{R}^{k}.
		\end{array}
		\right.
	\end{equation}
	From equation \eqref{17}, we can write equation \eqref{5} as 
	\begin{equation}\label{19}
		\left\{
		\begin{array}{ll}
			d X_{t} & = \bigg [X_{t} r_{t} + \rho^{T}(t) \big (U \hat{Z}^{\psi}_{t} + u - r_{t} \mathbf{1} \big) - C(t) - \beta(t) + R(t) \bigg] dt + \rho^{T}(t) d \nu_{t},  \quad t\in [0, T\wedge \xi]\\
			X(0) & = x.
		\end{array}
		\right.
	\end{equation}
	Let $\overline{\mathcal{L}}_{T}(x, 0)$ the set defined by
	\begin{align*}
		\overline{\mathcal{L}}_{T}(x, 0) & = \bigg \{\big (\rho(t), C(t), \beta(t) \big)_{t \in [0,T]} \bigg | \int_{0}^{T} |\rho(t)|^2 < +\infty,\;  \int_{0}^{T} |C(t)| < +\infty, \; \int_{0}^{T} |\beta(t)| < +\infty,  P_{r}-a.s \\
		& \text{and equation}\; \eqref{19}\; \text{has a unique strong solution such that}\; X(t) + \hat{\omega}_{1}(t) > 0\; P_{r}-a.s.\; t\in [0, T] \bigg \}.
	\end{align*}
	Next, the problem \eqref{6} can be written as:
	\begin{equation*}
		\varphi \big (0, x, z \big) = \sup_{(\rho_t, C_t, \beta_t) \in \overline{\mathcal{L}}_{T}(x,0) } \mathcal{V}\big (0,x, z;\rho_t, C_t, \beta_t \big)
	\end{equation*}
	with
	\begin{align*}
		&\mathcal{V}\big (x, z, \rho_t, C_t, \beta_t \big)\\ & = \mathbb{E} \bigg [\int_{0}^{\xi \wedge T} e^{-\int_{0}^{t}\theta (s)ds } V(C_{t}) dt + e^{-\int_{0}^{\xi} \theta (t) dt} V(\mathbf{v} (\xi)) \mathds{1}_{[\xi \leqslant T]} + e^{-\int_{0}^{T} \theta (t) dt} V(X(T)) \mathds{1}_{[\xi > T]}  \bigg]\\
		& = \mathbb{E} \bigg \{ \mathbb{E} \bigg [ \int_{0}^{\xi \wedge T} e^{-\int_{0}^{t}\theta(s)ds } V(C_{t}) dt \bigg | \mathcal{H}_{t} + e^{-\int_{0}^{\xi} \theta(t) dt} V(\mathbf{v} (\xi)) \mathds{1}_{[\xi \leqslant T]} \bigg | \mathcal{H}_{t} + e^{-\int_{0}^{T}\theta (t) dt } V(X(T)) \mathds{1}_{[\xi > T]} \bigg | \mathcal{H}_{t} \bigg ] \bigg \} \\
		& = \mathbb{E} \bigg [ \int_{0}^{T} e^{-\int_{0}^{t}\theta (s)ds } V(C_t) P_{r} \big (\xi \wedge T > t \big | \mathcal{H}_{t} \big) dt + \mathbb{E} \big [ e^{- \int_{0}^{\xi} \theta (t)dt } V(\mathbf{v} (\xi)) \mathds{1}_{[ \xi \leqslant T]} \big | \mathcal{H}_{t} \big ] \\ 
		& + e^{-\int_{0}^{T} \theta (t) dt } V(X(T)) \mathbb{E} \big [\mathds{1}_{[\xi > T]} \big | \mathcal{H}_{t}  \big] \bigg ]\\
		& = \mathbb{E} \bigg [\int_{0}^{T} e^{-\int_{0}^{t}\theta (s)ds  } V(C_t) \overline{G}_{\xi}(t)dt + \int_{0}^{T} e^{-\int_{0}^{t}\theta (s)ds } V(\mathbf{v} (t)) \mathds{1}_{[t \leqslant T]} g_{\xi}(t)dt + e^{- \int_{0}^{T}\theta(t)dt } V(X(T)) \overline{G}_{\xi} (T)  \bigg] \\
		& = \mathbb{E} \bigg [\int_{0}^{T} e^{-\int_{0}^{t}\theta (s)ds} \bigg (V(C_t) \overline{G}_{\xi}(t) + V(\mathbf{v} (t))g_{\xi}(t) \bigg) dt + e^{-\int_{0}^{T} \theta(t)dt } V(X(T)) \overline{G}_{\xi}(T)  \bigg].
	\end{align*}
	But
	\begin{equation*}
		\overline{G}_{\xi}(t) = e^{- \int_{0}^{t} \mu(s)ds } \quad \text{and}\quad g_{\xi}(t) = \mu (t) e^{-\int_{0}^{t}\mu (s)ds }.
	\end{equation*}
	Thus
	\begin{equation*}
		\mathcal{V}\big (x, z, \rho_t, C_t, \beta_t \big) = \mathbb{E} \bigg [\int_{0}^{T} e^{- \int_{0}^{t} (\theta(s) + \mu(s))ds } \big (V(C_t)  + \mu (t) V(\mathbf{v} (t))\big) dt + e^{- \int_{0}^{T} (\theta(t) + \mu(t))dt} V(X(T)) \bigg].
	\end{equation*}
	The value function is an $\mathcal{F}^{t}-$measurable random variable, since all model parameters are random. So the value function can not be determined from the partial differential equation as usual. This leads us to  theorem \ref{th2} in order to allow to solve the control problem through the combination of a HJB equation with BSDE associated to \eqref{6}. \par
	\begin{theorem}\label{th2}
		Let $\overline{\mathcal{O}}$ be the closure of the solvency region $\mathcal{O}$. Suppose that there exists a function $F \in C^{1,2,2}(\mathcal{O})$ and an admissible control $ \overline{\mathcal{L}}_{T}(x,0)$ such that
		\begin{enumerate}
			\item[\textbf{(a)}]
			$\mathcal{D}^{\rho, C, \beta} \big [ F \big (t, x, \omega_{1}(t), \omega_{2}(t) \big) \big ] + V(C) + \mu (t) V \bigg (x + \dfrac{\beta}{a(t)} \bigg) \leqslant 0, \quad P_{r} - a.s., \; \forall (\rho, C, \beta) \in  \overline{\mathcal{L}}_{T}(x,0)$ and $(t, x) \in \mathbb{R}^{m} \times [0, T \wedge \xi]$, where $\mathcal{D}^{\rho, C, \beta} $ is a partial differential generator.
			\item[\textbf{(b)}]
			$\mathcal{D}^{\hat{\rho}, \hat{C}, \hat{\beta}} \big [ F \big (t, x, \omega_{1}(t), \omega_{2}(t) \big) \big ] + V(\hat{C}) + \mu (t) V \bigg (x + \dfrac{\hat{\beta}}{a(t)} \bigg) = 0, \quad P_{r} - a.s., \; (t, x) \in \mathbb{R}^{m} \times [0, T \wedge \xi]$
			\item[\textbf{(c)}]
			$ \forall (\rho, C, \beta) \in  \overline{\mathcal{L}}_{T}(x,0)$,
			\begin{equation*}
				\lim_{t \to T^{-}} F \big (t, x, \omega_{1}(t), \omega_{2}(t) \big) = V(x), \quad P_{r} - a.s.
			\end{equation*}
			\item[\textbf{(d)}]
			Let $\Upsilon$ be the set of stopping times $\zeta \leqslant T$.  Then
			\begin{equation}\label{191}
				F \big (t, x, \omega_{1}(t), \omega_{2}(t) \big)= \varphi (t, x,z) = \sup_{(\rho, C, \beta) \in \mathcal{L}_{T}(x,t)} \mathcal{V}\big (t,x,z;\rho_t, C_t, \beta_t \big)= \mathcal{V}\big (t,x,z; \hat{\rho}_t,\hat{ C}_t, \hat{\beta}_t \big),
			\end{equation}
			where
			\begin{eqnarray*}
				&&\mathcal{V}\big (t,x,z;\rho_t, C_t, \beta_t \big)\\
				&&= \mathbb{E}_{t, x} \bigg [\int_{t}^{T} e^{-\int_{0}^{s}(\theta(u) + \mu(u))du } \big (V(C_s) + \mu(s) V(\mathbf{v} (s)) \big)ds + e^{-\int_{t}^{T}(\theta(s) + \mu(s) )ds} V(X(T)) \bigg]
			\end{eqnarray*}
			with $\mathbb{E}_{t,x}(\cdot) = \mathbb{E} \big (\cdot \big | X_{t} = x, \mathcal{H}_{t} \big)$ and $\mathcal{L}_{T}(x, t)$ is the restriction of $\mathcal{L}_{T}(x,0)$ on $[t, T]$. 
		\end{enumerate}
	\end{theorem}
	\begin{proof}
		See \cite{shen2016optimal}.
	\end{proof}
\begin{proposition}\label{prop1}
Let $\overline{\mathcal{O}}$ be the closure of the solvency region $\mathcal{O}$. Suppose that there exists a function $F \in C^{1,2,2}(\mathcal{O})$ and an admissible control $ \overline{\mathcal{L}}_{T}(x,0)$ such that $F$ is the solution of \eqref{20}-\eqref{21}:
	\begin{equation}\label{20}
		\left\{
		\begin{array}{ll}
			& -\dfrac{\partial F}{\partial t} + \sup_{(\rho, C, \beta) \in \mathbb{R}^{k} \times \mathbb{R}^{+} \times \mathbb{R}} \bigg \{\mathcal{D}^{\rho, C, \beta} F(t, x, z,\omega_{1}(t), \omega_{2}(t)) + V(C) + \mu (t) V\big (x + \dfrac{\beta}{a_t} \big) \bigg \}  = 0\\
			& F(T, x, 0, 1)  = V(x) 
		\end{array}
		\right.
	\end{equation}
	and 
	\begin{equation}\label{21}
		\left\{
		\begin{array}{ll}
			& \omega_{1}(t)  = \int_{t}^{T} f_{1}\big (s, \hat{Z}^{\psi}_{s}, \omega_{1}(s), \lambda_{1}(s) \big) ds - \int_{t}^{T}\lambda^{T}_{1}(s)d\nu_{s}\\
			& \omega_{2}(t)  = 1 + \int_{t}^{T} f_{2}\big (s, \hat{Z}^{\psi}_{s}, \omega_{2}(s), \lambda_{2}(s) \big) ds - \int_{t}^{T}\lambda^{T}_{2}(s)d\nu_{s}\\
			&  \omega_{1}(T)  = 0\\
			&  \omega_{2}(T)  = 1,
		\end{array}
		\right.
	\end{equation}
	where:
	\begin{itemize}
		\item[$\bullet$]
		$f_{1}, f_{2}$ are $\mathcal{F}^{t}-$measurable functions;
		\item[$\bullet$]
		$\mathcal{D}^{\rho, C, \beta}$ is a partial differential generator acting on a function $F$ as:
		\begin{align*}
			&\mathcal{D}^{\rho, C, \beta} F\\ & = \dfrac{\partial F}{\partial x} d X_{t} + \dfrac{\partial F}{\partial \omega_1} d \omega_{1}(t) + \dfrac{\partial F}{\partial \omega_2} d \omega_{2}(t) + \dfrac{1}{2} \dfrac{\partial^2 F}{\partial x^2} (d X_t)^2  + \dfrac{\partial^2 F}{\partial x \partial \omega_1} d X_{t} d \omega_{1}(t)  \\
			& + \dfrac{\partial^2 F}{\partial x \partial \omega_2} d X_{t} d \omega_{2}(t)+ \dfrac{1}{2} \dfrac{\partial^2 F}{\partial \omega^{2}_{1}}(d \omega_{1}(t))^{2} + \dfrac{1}{2}\dfrac{\partial^2 F}{\partial \omega^{2}_{2}}(d \omega_{2}(t))^{2} + \dfrac{\partial^2 F}{(d \omega_1) (d \omega_2)} (d \omega_1 (t)) (d \omega_2 (t)).
		\end{align*}
	\end{itemize}
 If we assume that $F$ is of the form
	\begin{eqnarray*}
		F\big (t, x, \omega_{1}(t), \omega_{2}(t) \big) = \dfrac{1}{\delta} \big (x + \omega_{1}(t) \big)^{\delta} \big (\omega_{2}(t) \big)^{1-\delta}, \delta \in (-\infty , 0) \cup (0,1),
	\end{eqnarray*}
	Then
		\begin{equation*}
		f_{1} \big (t, z, \omega_{1}(t), \lambda_{1}(t) \big) = - \big (r_t + a(t) \big) \omega_{1}(t) + R(t) - \lambda^{T}_{1} \big (Uz + u - r_t \mathbf{1} \big)
	\end{equation*}
	and 
	\begin{equation*} 
		f_{2} \big (t, z, \omega_{2}(t), \lambda_{2}(t) \big) =  \bigg (\dfrac{\delta}{1-\delta} \bigg) \lambda^T_2 \big (Uz + u - r_t \mathbf{1} \big) + H(t) + K(t,z) \omega_{2}(t),
	\end{equation*}
	where 
\begin{align*}
		H(t) & = 1 + \dfrac{(a(t))^{-\frac{\delta}{1-\delta}} }{(\mu (t))^{-\frac{1}{1-\delta}} }, \\
		K(t,z) & =  \dfrac{\delta}{2 \big (1 - \delta \big)^2}  \big (Uz + u - r_{t} \mathbf{1} \big)^{T} \big (\alpha \mathbb{C}_{B} \alpha^{T} +  \Sigma \mathbb{C}_{W} \Sigma^{T} \big)^{-1}  \big ( Uz +  u - r_{t} \mathbf{1} \big) +  \dfrac{\big (\theta(t) + \mu(t)\big)}{1-\delta} \\
		& + \bigg (\dfrac{\delta}{1-\delta} \bigg) \bigg (r_t + a(t) \bigg).
	\end{align*}
\end{proposition}	
\begin{proof}
See Appendix in section \ref{A}.\\
$\square$
\end{proof}
	\begin{lemma}\label{lm3}
		Similarly to \cite{hata2020optimal}, let us assume \textbf{H1)-H2)} and $f_{1}$ defined as in equation \eqref{24}. Then the BSDE
		\begin{equation}\label{26}
			\left\{
			\begin{array}{ll}
				\omega_{1}(t) & = \int_{t}^{T}f_{1}\big (s, \hat{Z}^{\psi}_{s}, \omega_{1}(s), \lambda_{1}(s) \big)\mathrm{d}s - \int_{t}^{T}\lambda^{T}_{1}(s)\mathrm{d}\nu_{s} \\
				\omega_{1}(T) & = 0
			\end{array}
			\right.
		\end{equation} 
		has a unique solution:
		\begin{equation}\label{27}
	\left\{
			\begin{array}{ll}
				\hat{\omega}_{1}(t) & =  \int_{t}^{T} e^{-\int_{t}^{s} (r_u + a(u))du } R(s) \mathrm{d}s \\
			\hat{\lambda}_{1}(t) & = 0.
			\end{array}
			\right.		
		\end{equation}
	\end{lemma}
	\begin{proof}
		Let us define the probability measure $P^{1}_{r}$ by
		\begin{align*}
			\dfrac{d P^{1}_{r}}{d P_r} \bigg |_{\mathcal{F}^{t}} & : = \exp \bigg [\int_{0}^{t} \big (U \hat{Z}^{\psi}_{s} + u - r_s \mathbf{1} \big)^{T}\big (\alpha \mathbb{C}_{B}\alpha^T + \Sigma \mathbb{C}_{W}\Sigma^T \big)^{-1} \mathrm{d}\nu_{s} \\
			& - \dfrac{1}{2} \int_{0}^{t} \big (U \hat{Z}^{\psi}_{s} + u - r_s \mathbf{1} \big)^{T}\big (\alpha \mathbb{C}_{B}\alpha^T + \Sigma \mathbb{C}_{W}\Sigma^T \big)^{-1} \big (U \hat{Z}^{\psi}_{s} + u - r_s \mathbf{1} \big)\mathrm{d}s \bigg].
		\end{align*}
		Thus under $P^{1}_{r}, \nu^{1}_{t}$ defined by
		\begin{equation*}
			\nu^{1}_{t} : = \nu_{t} + \int_{0}^{t} \big (U \hat{Z}^{\psi}_{s} + u - r_s \mathbf{1} \big) \mathrm{d}s
		\end{equation*}
		is a $\mathcal{F}^{t}-$Weiner process with covariance matrix $\big (\alpha \mathbb{C}_{B}\alpha^T + \Sigma \mathbb{C}_{W}\Sigma^T \big)$. Then equation \eqref{26} can be written as:
		\begin{align}\nonumber\label{28}
			\hat{\omega}_{1}(t) & =  \int_{t}^{T}\bigg [-\omega_{1}(s) \big (r_s + a(s) \big) + R(s) - \lambda^{T}_{1}(s) \big (U \hat{Z}^{\psi}_{s} + u - r_{s} \mathbf{1}\big) \bigg] \mathrm{d}s \\\nonumber
			& - \int_{t}^{T}\lambda^{T}_{1}(s) \bigg [d \nu^{1}_{s} - \big (U \hat{Z}^{\psi}_{s}  + u - r_{s} \mathbf{1}\big) \mathrm{d}s \bigg] \\
			& = \int^{T}_{t}  \bigg [-\omega_{1}(s) \big (r_s + a(s) \big) + R(s)\bigg] ds - \int_{t}^{T}\lambda^{T}_{1}(s) \mathrm{d} \nu^{1}_{s}.
		\end{align}
		Hence from proposition 4.1.1 of \cite{zhang2017backward}, we conclude that the BSDE \eqref{28} has a unique solution \eqref{27}.\par
		$\hat{\omega}_{1}(t)$ can be interpreted as an actuarial value process of future income and $\lambda_{1}(t)$ its volatility process.
		$\square$
	\end{proof}
	\begin{proposition}
		Consider the assumptions \textbf{H1),  H2)} and the linear case as specified in the previous theorem. Then 
		\begin{equation*}
			P = \mathbb{E}\bigg [\big (Z_{t} - \hat{Z}^{\psi}_{t} \big)\big (Z_{t} - \hat{Z}^{\psi}_{t} \big)^{T} \bigg | \mathcal{F}^{t} \bigg]
		\end{equation*}
		solves the following Riccati equation
		\begin{equation*}
			\left\{
			\begin{array}{ll}
				& \dot{P}(t) + P(t) S_{2}(t)P(t) + S^{T}_{1}(t) P(t) + P(t) S_{1}(t) + S_{0}(t) = 0 \\
				& P(s) = 0, \quad s \in [t, T],
			\end{array}
			\right.
		\end{equation*}
		where:
		\begin{align*}
			S_{2}(t) & = - U^{T} \big (\alpha \mathbb{C}_{B}\alpha^T + \Sigma \mathbb{C}_{W}\Sigma^T \big)^{-1} U\\
			S_{1}(t) & = - U^{T} \big (\alpha \mathbb{C}_{B}\alpha^T + \Sigma \mathbb{C}_{W}\Sigma^T \big)^{-1}\big (\alpha \mathbb{C}_{B} - \alpha \mathbb{C}_{B}\gamma^T \big) - D^T\\
			S_{0}(t) & = - \mathbb{C}_{B} \alpha^T \big (\alpha \mathbb{C}_{B}\alpha^T + \Sigma \mathbb{C}_{W}\Sigma^T \big)^{-1}\big (\alpha \mathbb{C}_{B} \gamma^T - \alpha \mathbb{C}_{B} \big) + \gamma \mathbb{C}_{B}\alpha^{T} \big (\alpha \mathbb{C}_{B}\alpha^T + \Sigma \mathbb{C}_{W}\Sigma^T \big)^{-1}\alpha \mathbb{C}_{B}\\
			& - \gamma \mathbb{C}_{B} \gamma^T.
		\end{align*}
	\end{proposition}
	\begin{proof}
		We have:
		\begin{align*}
			&d \big (Z_{t} - \hat{Z}^{\psi}_{t} \big)\\ & = d Z_{t} - d \hat{Z}^{\psi}_{t} \\
			& = \big (D Z_{t} + d \big) dt + \gamma d B_{t} - \bigg [\big (D\hat{Z}^{\psi}_{t} + d \big)dt + \big (PU^{T} + \mathbb{C}_{B} \alpha^T \big)\big (\alpha \mathbb{C}_{B}\alpha^T + \Sigma \mathbb{C}_{W}\Sigma^T \big)^{-1}d \nu_t  \bigg] \\
			& = D \big (Z_t - \hat{Z}^{\psi}_{t} \big) dt + \gamma d B_{t} - \big (PU^{T} + \mathbb{C}_{B} \alpha^T \big) \big (\alpha \mathbb{C}_{B}\alpha^T + \Sigma \mathbb{C}_{W}\Sigma^T \big)^{-1}\big (\alpha d B_{t} + \Sigma d W_{t} \big)                                             
		\end{align*}
		\begin{align*}
			d \big (Z_{t} - \hat{Z}^{\psi}_{t} \big)^T & =\big (Z_{t} - \hat{Z}^{\psi}_{t} \big)^T D^{T} dt + d B^{T}_{t} \cdot \gamma^T - \big (d B^{T}_{t} \cdot \alpha^T + d W^{T}_{t} \cdot \Sigma^T \big)  \big (\alpha \mathbb{C}_{B}\alpha^T + \Sigma \mathbb{C}_{W}\Sigma^T \big)^{-1}\big (UP \\
			& + \alpha \mathbb{C}_{B} \big)
		\end{align*}
		\begin{align*}
			d \bigg [ \big (Z_{t} - \hat{Z}^{\psi}_{t} \big) \big (Z_{t} - \hat{Z}^{\psi}_{t} \big)^{T} \bigg] & = \bigg \{ D  \big (Z_{t} - \hat{Z}^{\psi}_{t} \big) dt + \gamma d B_{t} - \big ( PU^T + \mathbb{C}_{B} \alpha^T \big)  \big (\alpha \mathbb{C}_{B}\alpha^T  + \Sigma \mathbb{C}_{W}\Sigma^T \big)^{-1}\big (\\
			& \alpha d B_{t} + \Sigma d W_{t} \big) \bigg \} \big (Z_{t} - \hat{Z}^{\psi}_{t} \big)^{T}  + \big (Z_{t} - \hat{Z}^{\psi}_{t} \big) \bigg \{ \big (Z_{t} - \hat{Z}^{\psi}_{t} \big)^{T} D^{T}dt + d B^{T}_{t} \cdot \gamma^{T} \\
			& -  \big (dB^{T}_{t}\cdot \alpha^T + d W^{T}\cdot \Sigma^{T} \big ) \big (\alpha \mathbb{C}_{B}\alpha^T  + \Sigma \mathbb{C}_{W}\Sigma^T \big)^{-1} \big (UP + \alpha \mathbb{C}_{B} \big) \bigg \} +\\
			& \bigg \{ D\big (Z_{t} - \hat{Z}^{\psi}_{t} \big) dt + \gamma dB_{t} - \big (PU^T + \mathbb{C}_{B} \alpha^T \big)\big (\alpha \mathbb{C}_{B}\alpha^T \\
			& + \Sigma \mathbb{C}_{W}\Sigma^T \big)^{-1}\big (\alpha d B_t + \Sigma d W_t \big )\bigg \} \bigg \{  \big (Z_{t} - \hat{Z}^{\psi}_{t} \big)^{T} D^{T} dt  + dB^{T}_{t} \cdot \gamma^{T} - \big (d B^{T}_{t} \cdot \alpha^T\\
			& + d W^{T}_{t} \cdot \Sigma^T \big) \big (\alpha \mathbb{C}_{B}\alpha^T  + \Sigma \mathbb{C}_{W}\Sigma^T \big)^{-1}\big (UP + \alpha \mathbb{C}_{B} \big) \bigg \}  
		\end{align*}
		\begin{align*}
			P(t) & = \int_{0}^{t} \bigg \{ DP + PD^{T} + \gamma \mathbb{C}_{B} \gamma^T - \gamma \mathbb{C}_{B} \alpha^{T}  \big (\alpha \mathbb{C}_{B}\alpha^T  + \Sigma \mathbb{C}_{W}\Sigma^T \big)^{-1} \big (UP + \alpha \mathbb{C}_{B} \big) \\
			&- \big (PU^T + \mathbb{C}_{B} \alpha^T\big)  \big (\alpha \mathbb{C}_{B}\alpha^T  + \Sigma \mathbb{C}_{W}\Sigma^T \big)^{-1}\alpha \mathbb{C}_{B}\gamma^T + \big (PU^T + \mathbb{C}_{B}\alpha^T \big )  \big (\alpha \mathbb{C}_{B}\alpha^T  + \Sigma \mathbb{C}_{W}\Sigma^T \big)^{-1}\big (UP \\
			& +\alpha \mathbb{C}_{B} \big ) \bigg \}ds
		\end{align*}
		\begin{align*}
			\Rightarrow \dot{P}(t) & = \big (PU^T + \mathbb{C}_{B}\alpha^T \big ) \big (\alpha \mathbb{C}_{B}\alpha^T  + \Sigma \mathbb{C}_{W}\Sigma^T \big)^{-1}\big (UP + \alpha \mathbb{C}_{B} \big) - \big (PU^T + \mathbb{C}_{B}\alpha^T \big) \big (\alpha \mathbb{C}_{B}\alpha^T  \\
			& + \Sigma \mathbb{C}_{W}\Sigma^T \big)^{-1}\alpha \mathbb{C}_{B}\gamma^{T} \gamma \mathbb{C}_{B}\alpha^T  \big (\alpha \mathbb{C}_{B}\alpha^T  + \Sigma \mathbb{C}_{W}\Sigma^T \big)^{-1} \big (UP + \alpha\mathbb{C}_{B} \big) + \gamma \mathbb{C}_{B}\gamma^T + DP + PD^T \\
			& = PU^T  \big (\alpha \mathbb{C}_{B}\alpha^T  + \Sigma \mathbb{C}_{W}\Sigma^T \big)^{-1}\big (UP + \alpha \mathbb{C}_{B} \big) + \mathbb{C}_{B}\alpha^{T} \big (\alpha \mathbb{C}_{B}\alpha^T  + \Sigma \mathbb{C}_{W}\Sigma^T \big)^{-1}\big (UP + \alpha \mathbb{C}_{B} \big) \\
			&- PU^{T}  \big (\alpha \mathbb{C}_{B}\alpha^T  + \Sigma \mathbb{C}_{W}\Sigma^T \big)^{-1}\alpha \mathbb{C}_{B}\gamma^T - \mathbb{C}_{B}\alpha^{T} \big (\alpha \mathbb{C}_{B}\alpha^T  + \Sigma \mathbb{C}_{W}\Sigma^T \big)^{-1}\alpha\mathbb{C}_{B}\gamma^{T} - \gamma \mathbb{C}_{B}\alpha^{T} \big (\alpha \mathbb{C}_{B}\alpha^T \\
			& + \Sigma \mathbb{C}_{W}\Sigma^T \big)^{-1} UP - \gamma \mathbb{C}_{B}\alpha^{T} \big (\alpha \mathbb{C}_{B}\alpha^T  + \Sigma \mathbb{C}_{W}\Sigma^T \big)^{-1}\alpha \mathbb{C}_{B} + \gamma \mathbb{C}_{B}\gamma^T + DP + PD^T\\
			& = PU^{T}  \big (\alpha \mathbb{C}_{B}\alpha^T  + \Sigma \mathbb{C}_{W}\Sigma^T \big)^{-1}UP + PU^{T} \big (\alpha \mathbb{C}_{B}\alpha^T  + \Sigma \mathbb{C}_{W}\Sigma^T \big)^{-1}\alpha \mathbb{C}_{B} + \mathbb{C}_{B}\alpha^{T} \big (\alpha \mathbb{C}_{B}\alpha^T  +\\
			& \Sigma \mathbb{C}_{W}\Sigma^T \big)^{-1}UP  + \mathbb{C}_{B}\alpha^{T} \big (\alpha \mathbb{C}_{B}\alpha^T  + \Sigma \mathbb{C}_{W}\Sigma^T \big)^{-1}\alpha \mathbb{C}_{B}- PU^{T} \big (\alpha \mathbb{C}_{B}\alpha^T  + \Sigma \mathbb{C}_{W}\Sigma^T \big)^{-1}\alpha\mathbb{C}_{B}\gamma^T -\\
			&  \mathbb{C}_{B}\alpha^{T} \big (\alpha \mathbb{C}_{B}\alpha^T  + \Sigma \mathbb{C}_{W}\Sigma^T \big)^{-1}\alpha \mathbb{C}_{B}\gamma^{T}-\gamma\mathbb{C}_{B}\alpha^T  \big (\alpha \mathbb{C}_{B}\alpha^T  + \Sigma \mathbb{C}_{W}\Sigma^T \big)^{-1} UP - \gamma \mathbb{C}_{B}\alpha^{T} \big (\alpha \mathbb{C}_{B}\alpha^T  + \\
			& \Sigma \mathbb{C}_{W}\Sigma^T \big)^{-1}\alpha \mathbb{C}_{B} + \gamma \mathbb{C}_{B}\gamma^T + DP + PD^T.
		\end{align*}
		Thus,
		\begin{align*}
			\dot{P} - PU^{T} \big (\alpha \mathbb{C}_{B}\alpha^T  + \Sigma \mathbb{C}_{W}\Sigma^T \big)^{-1}UP & = \big (\mathbb{C}_{B}\alpha^T - \gamma \mathbb{C}_{B}\alpha^T \big) \big (\alpha \mathbb{C}_{B}\alpha^T  + \Sigma \mathbb{C}_{W}\Sigma^T \big)^{-1}UP + PU^{T} \big (\alpha \mathbb{C}_{B}\alpha^T  \\
			&+ \Sigma \mathbb{C}_{W}\Sigma^T \big)^{-1}\big ( \alpha \mathbb{C}_{B} - \alpha \mathbb{C}_{B}\gamma^T \big ) + \mathbb{C}_{B}\alpha^{T} \big (\alpha \mathbb{C}_{B}\alpha^T  + \Sigma \mathbb{C}_{W}\Sigma^T \big)^{-1}\alpha \mathbb{C}_{B} \\
			&- \mathbb{C}_{B}\alpha^{T} \big (\alpha \mathbb{C}_{B}\alpha^T  + \Sigma \mathbb{C}_{W}\Sigma^T \big)^{-1}\alpha \mathbb{C}_{B}\gamma^{T}-\gamma \mathbb{C}_{B}\alpha^{T} \big (\alpha \mathbb{C}_{B}\alpha^T  \\
			&+ \Sigma \mathbb{C}_{W}\Sigma^T \big)^{-1}\alpha \mathbb{C}_{B} + \gamma \mathbb{C}_{B}\gamma^T + DP + PD^T.
		\end{align*}
		\begin{align*}
			& \Rightarrow \dot{P} - PU^{T} \big (\alpha \mathbb{C}_{B}\alpha^T  + \Sigma \mathbb{C}_{W}\Sigma^T \big)^{-1}UP - \big (\mathbb{C}_{B}\alpha^T - \gamma \mathbb{C}_{B}\alpha^T \big) \big (\alpha \mathbb{C}_{B}\alpha^T  + \Sigma \mathbb{C}_{W}\Sigma^T \big)^{-1}UP - \\
			& DP - PU^{T} \big (\alpha \mathbb{C}_{B}\alpha^T  + \Sigma \mathbb{C}_{W}\Sigma^T \big)^{-1}\big (\alpha \mathbb{C}_{B} - \alpha \mathbb{C}_{B}\gamma^T \big) - PD^T - \mathbb{C}_{B}\alpha^{T} \big (\alpha \mathbb{C}_{B}\alpha^T  + \Sigma \mathbb{C}_{W}\Sigma^T \big)^{-1}\alpha \mathbb{C}_{B} \\
			&+ \mathbb{C}_{B}\alpha^{T}  \big (\alpha \mathbb{C}_{B}\alpha^T  + \Sigma \mathbb{C}_{W}\Sigma^T \big)^{-1}\alpha \mathbb{C}_{B}\gamma^{T} +\gamma \mathbb{C}_{B}\alpha^{T} \big (\alpha \mathbb{C}_{B}\alpha^T  + \Sigma \mathbb{C}_{W}\Sigma^T \big)^{-1}\alpha \mathbb{C}_{B} - \gamma \mathbb{C}_{B}\gamma^T = 0.\\
			& \Rightarrow \dot{P}(t) + P(t)S_{2}(t)P(t) + S^{T}_{1}(t)P(t) + P(t)S_{1}(t) + S_{0}(t) = 0,
		\end{align*}
		where
		\begin{align*}
			S_{2}(t) & = - U^{T}  \big (\alpha \mathbb{C}_{B}\alpha^T  + \Sigma \mathbb{C}_{W}\Sigma^T \big)^{-1}U \\
			S_{1}(t) & = - U^{T}  \big (\alpha \mathbb{C}_{B}\alpha^T  + \Sigma \mathbb{C}_{W}\Sigma^T \big)^{-1} \big (\alpha \mathbb{C}_{B} - \alpha \mathbb{C}_{B} \gamma^T \big) - D^T \\
			S_{0}(t) & = -\mathbb{C}_{B}\alpha^{T} \big (\alpha \mathbb{C}_{B}\alpha^T  + \Sigma \mathbb{C}_{W}\Sigma^T \big)^{-1} \big (\alpha \mathbb{C}_{B}\gamma^T - \alpha \mathbb{C}_{B} \big) + \gamma \mathbb{C}_{B}\alpha^{T} \big (\alpha \mathbb{C}_{B}\alpha^T  + \Sigma \mathbb{C}_{W}\Sigma^T \big)^{-1}\alpha \mathbb{C}_{B} \\
			& - \gamma \mathbb{C}_{B}\gamma^{T}.
		\end{align*}
		$\square$
	\end{proof}
	We consider the BSDE
	\begin{equation}\label{29}
		\left\{
		\begin{array}{ll}
			\omega_{2}(t) & = 1 + \int_{t}^{T} f_{2}\big (s, \hat{Z}^{\psi}_{s}, \omega_{2}(s), \lambda_{2}(s) \big) ds - \int_{t}^{T}\lambda^{T}_{2}(s)d \nu_{s}\\
			\omega_{2}(T) & = 1,
		\end{array}
		\right.
	\end{equation}
	where $f_2$ is defined as in equation \eqref{25}.\par
	In addition, we consider the following Riccati equation
	\begin{equation}\label{30}
		\left\{
		\begin{array}{ll}
			& \dot{M}(t) + M(t)A_{2}(t)M(t) + A^{T}_{1}M(t) + M(t)A_{1}(t) + A_{0} = 0\\
			& M(s) = 0, \quad s \in [t, T],
		\end{array}
		\right.
	\end{equation}
	where
	\begin{align*}
		A_{2}(t) & = \dfrac{1}{1-\delta} \big (PU^T + \mathbb{C}_{B}\alpha^T \big) \big (\alpha \mathbb{C}_{B}\alpha^T  + \Sigma \mathbb{C}_{W}\Sigma^T \big)^{-1}\big (PU^T + \mathbb{C}_{B}\alpha^T \big)^{T} \\
		A_{1}(t) & = D + \dfrac{\delta}{1-\delta} \big (PU^T + \mathbb{C}_{B} \alpha^T \big) \big (\alpha \mathbb{C}_{B}\alpha^T  + \Sigma \mathbb{C}_{W}\Sigma^T \big)^{-1}U \\
		A_{0}(t) & = \dfrac{\delta}{1 - \delta} U^T  \big (\alpha \mathbb{C}_{B}\alpha^T  + \Sigma \mathbb{C}_{W}\Sigma^T \big)^{-1} U.
	\end{align*}
	The term $n(t)$ is the solution of the linear differential equation:
	\begin{equation*}
		\left\{
		\begin{array}{ll}
			& \dot{n}(t) + \big [A_{1}(t) + A_{2}(t)M(t)\big]^{T} n(t) + M(t) d + \dfrac{\delta}{1 - \delta} \bigg [U^{T} \big (\alpha \mathbb{C}_{B}\alpha^T  + \Sigma \mathbb{C}_{W}\Sigma^T \big)^{-1} + M(t)\big (PU^T \\
			& + \mathbb{C}_{B}\alpha^T \big) \big (\alpha \mathbb{C}_{B}\alpha^T + \Sigma \mathbb{C}_{W}\Sigma^T \big)^{-1}  \bigg]\big (u - r_{t} \mathbf{1} \big) = 0 \\
			& n(s) = 0, \quad s\in [t, T]
		\end{array}
		\right.
	\end{equation*}
	which is equivalent to
	\begin{equation}\label{31}
		\left\{
		\begin{array}{ll}
			& \dot{n}(t) + \big [A_{1}(t) + A_{2}(t)M(t)\big]^{T} n(t) + M(t) d + \dfrac{\delta}{1 - \delta} \bigg [U^{T}  + M(t)\big (PU^T + \mathbb{C}_{B}\alpha^T \big)\bigg ] \big (\alpha \mathbb{C}_{B}\alpha^T  \\
			&+ \Sigma \mathbb{C}_{W}\Sigma^T \big)^{-1} \big (u - r_{t} \mathbf{1} \big) = 0 \\
			& n(s) = 0, \quad s\in [t, T]
		\end{array}
		\right.
	\end{equation}
	and $q(t)$ is the solution of 
	\begin{equation}\label{32}
		\left\{
		\begin{array}{ll}
			& \dot{q}(t) + \dfrac{1}{2} tr \bigg [\big (PU^T + \mathbb{C}_{B} \alpha^T \big) \big (\alpha \mathbb{C}_{B}\alpha^T  + \Sigma \mathbb{C}_{W}\Sigma^T \big)^{-1} \big (PU^T + \mathbb{C}_{B} \alpha^T \big)^T M(t) \bigg]  + \dfrac{1}{2} n^{T}(t)\big (PU^T \\
			& + \mathbb{C}_{B} \alpha^T \big) \big (\alpha \mathbb{C}_{B}\alpha^T  + \Sigma \mathbb{C}_{W}\Sigma^T \big)^{-1} \big (PU^T + \mathbb{C}_{B} \alpha^T \big)^T n(t)  + d^{T}n(t) + \delta r_{t} + \dfrac{\delta}{2(1 - \delta)} \big (u - r_{t} \mathbf{1}\\
			& + \big (PU^T + \mathbb{C}_{B}\alpha^T \big)^T n(t) \big)^{T}\big (\alpha \mathbb{C}_{B}\alpha^T + \Sigma \mathbb{C}_{W}\Sigma^T \big)^{-1}  \big (u - r_{t} \mathbf{1} + \big (PU^T + \mathbb{C}_{B}\alpha^T \big)^T n(t) \big) = 0 \\
			& q(s) = 0, \quad s \in [t, T].
		\end{array}
		\right.
	\end{equation}
	We state the lemma \ref{lm4}:
	\begin{lemma}\label{lm4}
		Assume \textbf{H1)-H2)} and $f_{2}$ defined as in equation \eqref{25}. If equations \eqref{30}-\eqref{31} and \eqref{32} have as solutions $M(t), n(t)$ and $q(t)$ respectively, then the BSDE
		\begin{equation}\label{33}
			\left\{
			\begin{array}{ll}
				\omega_{2}(t) & = 1 + \int_{t}^{T} f_{2} \big (s, \hat{Z}^{\psi}_{t}, \omega_{2}(s), \lambda_{2}(s) \big)\mathrm{d}s - \int_{t}^{T}\lambda^{T}_{2}(s) \mathrm{d} \nu_{s} \\
				\omega_{2}(T) & = 1
			\end{array}
			\right.
		\end{equation} 
		has as solution the couple:
		\begin{align}\nonumber \label{34}
			\hat{\omega}^{z}_{2}(t) & = \exp \bigg \{\dfrac{1}{(1 - \delta)} \int_{t}^{T} \big (\theta(s) + \mu(s) + \delta a(s) \big ) \mathrm{d}s + \dfrac{1}{(1 - \delta)} \bigg [\dfrac{1}{2} z^T M(T)z + n^{T}(T)z + q(T) \bigg]   \bigg \} \\ 
			& + \int_{t}^{T}e^{\frac{\delta}{ (1 - \delta)} \int_{t}^{s}  a(u)\mathrm{d}u } \exp\bigg \{\dfrac{1}{(1 - \delta)} \bigg [\dfrac{1}{2} z^T M(s) z + n^{T}(s)z + q(s) \bigg]  \bigg \} \bigg (1 + \dfrac{(a(s))^{- \frac{\delta}{1-\delta}} }{(\mu(s))^{-\frac{1}{1-\delta}}} \bigg) \mathrm{d}s 
		\end{align}
		and 
		\begin{equation}\label{35}
			\hat{\lambda}_{2}(t) =  \big (\alpha \mathbb{C}_{B}\alpha^T  + \Sigma \mathbb{C}_{W}\Sigma^T \big)^{-1} \big (PU^T + \mathbb{C}_{B} \alpha^T \big)^T \dfrac{\partial}{\partial t} \hat{\omega}_{2}^{\hat{Z}^{\psi}_t}(t).
		\end{equation}
	\end{lemma}
	\begin{proof}
		Let us define the probability measure $P^{2}_{r}$ by 
		\begin{align*}
			\dfrac{d P^{2}_{r}}{d P_r} \bigg |_{\mathcal{F}^{t}} & : = \exp \bigg [ \dfrac{\delta}{ (1 - \delta)} \int_{0}^{t} \big (U\hat{Z}^{\psi}_{s} + u - r_{s}\mathbf{1} \big)^{T}  \big (\alpha \mathbb{C}_{B}\alpha^T  + \Sigma \mathbb{C}_{W}\Sigma^T \big)^{-1}  d\nu_{s} \\
			& - \dfrac{1}{2} \bigg (\dfrac{\delta}{1-\delta} \bigg)^{2}\int_{0}^{t}   \big (U\hat{Z}^{\psi}_{s} + u - r_{s}\mathbf{1} \big)^{T}  \big (\alpha \mathbb{C}_{B}\alpha^T  + \Sigma \mathbb{C}_{W}\Sigma^T \big)^{-1}\big (U\hat{Z}^{\psi}_{s} + u - r_{s}\mathbf{1} \big) ds \bigg ].
		\end{align*}
		Thus under $P^{2}_{r}$
		\begin{equation*}
			\nu^{2}_{t} : = \nu_{t} - \dfrac{\delta}{ 1 - \delta} \int_{0}^{t} \big (U\hat{Z}^{\psi}_{s} + u - r_{s}\mathbf{1} \big) ds
		\end{equation*}
		is a $\mathcal{F}^t -$Weiner process with covariance matrix $ \big (\alpha \mathbb{C}_{B}\alpha^T  + \Sigma \mathbb{C}_{W}\Sigma^T \big)$.  Then \eqref{33} can be written as follows:
		\begin{align*}
			\hat{\omega}^{\hat{Z}^{\psi}_{t}}_{2}(t) & = 1 + \int_{t}^{T} \bigg [H(s) +  \dfrac{\delta}{1 - \delta}\lambda^{T}_{2}(s) \big (U\hat{Z}^{\psi}_{s} + u - r_s \mathbf{1} \big)  + K(s, \hat{Z}^{\psi}_{s})\hat{\omega}^{\hat{Z}^{\psi}_{s}}_{2}(s) \bigg] ds\\
			&  - \int_{t}^{T} \lambda^{T}_{2}(s) \bigg (d \nu^{2}_{s} + \dfrac{\delta}{1 - \delta} \big (U\hat{Z}^{\psi}_{s} + u - r_s \mathbf{1} \big) ds \bigg).
		\end{align*}
		\begin{align}\nonumber \label{36}
			\Rightarrow \hat{\omega}^{\hat{Z}^{\psi}_{t}}_{2}(t) & = 1 +  \int_{t}^{T} \bigg [H(s)  + K(s, \hat{Z}^{\psi}_{s})\hat{\omega}^{\hat{Z}^{\psi}_{s}}_{2}(s)  \bigg] ds - \int_{t}^{T} \lambda^{T}_{2}(s) d \nu^{2}_{s}.
		\end{align}
		In addition, under $P^{2}_{r}, \hat{Z}^{\psi}_{s}$ solves 
		\begin{equation*}
			\left\{
			\begin{array}{ll}
				d \hat{Z}^{\psi}_{s} & = \big (D \hat{Z}^{\psi}_{s} + d \big) ds + \big (PU^T + \mathbb{C}_{B} \alpha^T \big) \big (\alpha \mathbb{C}_{B}\alpha^T  + \Sigma \mathbb{C}_{W}\Sigma^T \big)^{-1} \big (d \nu^{2}_{s}  + \dfrac{\delta}{ 1 - \delta} \big (U\hat{Z}^{\psi}_{s} + u - r_s \mathbf{1} \big)ds  \big) \\
				\hat{Z}^{\psi}_{t} & = z,\quad s\in [t, T]
			\end{array}
			\right.
		\end{equation*}
		\begin{equation*}
			\Rightarrow \left\{
			\begin{array}{ll}
				d \hat{Z}^{\psi}_{s} & =\bigg [ D \hat{Z}^{\psi}_{s} + d + \dfrac{\delta}{1-\delta}\big (PU^T + \mathbb{C}_{B} \alpha^T \big) \big (\alpha \mathbb{C}_{B}\alpha^T  + \Sigma \mathbb{C}_{W}\Sigma^T \big)^{-1} \big (U\hat{Z}^{\psi}_{s} + u - r_s \mathbf{1} \big)\bigg ] ds + \big (PU^T \\
				& + \mathbb{C}_{B} \alpha^T \big) \big (\alpha \mathbb{C}_{B}\alpha^T  + \Sigma \mathbb{C}_{W}\Sigma^T \big)^{-1}  d \nu^{2}_{s} \\
				\hat{Z}^{\psi}_{t} & = z,\quad s\in [t, T].
			\end{array}
			\right.
		\end{equation*}
		Using lemma A.3 from \cite{shen2016optimal} and proposition 4.1.1 from \cite{zhang2017backward}, we conclude that the BSDE \eqref{33} has a unique solution $\big (\hat{\omega}^{ \hat{Z}^{\psi}_{t}}_{2}(t), \hat{\lambda}_{2}(t) \big)$ defined by:
		\begin{align*}
			\hat{\omega}^{ \hat{Z}^{\psi}_{t}}_{2}(t) & = e^{\frac{1}{(1 - \delta)} \int_{t}^{T}\big (\theta(s) + \mu(s) \delta a(s)\big ) ds} \mathbb{E}_{P^{2}_{r}} \bigg [ \exp \bigg [ \int_{t}^{T} \dfrac{1}{ (1 - \delta)} \bigg ( \dfrac{\delta}{2(1-\delta)} \big ( U\hat{Z}^{\psi}_{s} + u - r_{s} \mathbf{1} \big )^{T}\big ( \alpha \mathbb{C}_{B} \alpha^T\\
			& + \Sigma \mathbb{C}_{W}\Sigma^T \big )^{-1} \big ( U\hat{Z}^{\psi}_{s} +  u - r_s \mathbf{1} \big ) + \delta r_s \bigg ) ds \bigg ] \bigg | \hat{Z}^{\psi}_{t} = z \bigg ]  + \int_{t}^{T} e^{\frac{1}{ (1 - \delta)} \int_{t}^{s}\big ( \theta (u) + \mu(u) + \delta a(u) \big )du} \mathbb{E}_{P^{2}_{r}} \bigg [\\
			&  \exp \bigg [ \int_{t}^{s} \frac{1}{ (1 - \delta)}  \bigg ( \frac{\delta}{2 (1 - \delta)} \big ( U\hat{Z}^{\psi}_{u} + u  - r_u \mathbf{1} \big )^{T} \big ( \alpha \mathbb{C}_{B}\alpha^T  + \Sigma \mathbb{C}_{W}\Sigma^{T} \big )^{-1}\big ( U\hat{Z}^{\psi}_{u} + u - r_u \mathbf{1} \big ) \\
			& + \delta r_u \bigg ) du \bigg ] \bigg |  \hat{Z}^{\psi}_{t}  = z \bigg ] \bigg (1 + \dfrac{(a(s))^{- \frac{\delta}{1-\delta}} }{(\mu(s))^{-\frac{1}{1-\delta}}} \bigg) ds\\
\hat{\lambda}_{2}(t) & =  \big (\alpha \mathbb{C}_{B}\alpha^T  + \Sigma \mathbb{C}_{W}\Sigma^T \big)^{-1}\big (PU^T + \mathbb{C}_{B}\alpha^T \big)\dfrac{\partial}{\partial t} \hat{\omega}^{ \hat{Z}^{\psi}_{t}}_{2}(t)
		\end{align*}
		Denote \\
\begin{eqnarray*}
&& \mathbf{\hbar}(t; s, z)\\
&& 	= \mathbb{E}_{P^{2}_{r}} \bigg [\exp \bigg (\dfrac{1}{2 (1 - \delta)} \int_{t}^{s} \bigg [\dfrac{\delta}{2 (1 - \delta)}\big (U\hat{Z}^{\psi}_{u} + u - r_u \mathbf{1}\big)^T \big (\alpha \mathbb{C}_{B}\alpha^T  + \Sigma \mathbb{C}_{W}\Sigma^T \big)^{-1} \big (U\hat{Z}^{\psi}_{u} + u - r_u \mathbf{1}\big)\\
&& + \delta r_u  \bigg] du \bigg) \bigg |  \hat{Z}^{\psi}_{t} = z \bigg].
\end{eqnarray*}
		Then $\mathbf{\hbar}$ is the solution of 
		\begin{equation}\label{37}
			\left\{
			\begin{array}{ll}
				& \dfrac{\partial \mathbf{\hbar}}{\partial t} + \dfrac{1}{2} tr \bigg [ \big (PU^T + \mathbb{C}_{B}\alpha^T \big)\big (\alpha \mathbb{C}_{B}\alpha^T  + \Sigma \mathbb{C}_{W}\Sigma^T \big)^{-1}\big (PU^T + \mathbb{C}_{B}\alpha^T \big)^{T} \dfrac{\partial^2 \mathbf{\hbar}}{\partial t^2}  \bigg] \\
				&+ \bigg [d + \dfrac{\delta}{1-\delta}  \big (PU^T + \mathbb{C}_{B}\alpha^T \big)\big (\alpha \mathbb{C}_{B}\alpha^T  + \Sigma \mathbb{C}_{W}\Sigma^T \big)^{-1}\big (u - r_{t}\mathbf{1} \big) + \bigg [ D\\
				& + \dfrac{\delta}{1-\delta} \big (PU^T + \mathbb{C}_{B}\alpha^T \big)\big (\alpha \mathbb{C}_{B}\alpha^T  + \Sigma \mathbb{C}_{W}\Sigma^T \big)^{-1}U \bigg ] z \bigg]^{T} \dfrac{\partial \mathbf{\hbar}}{\partial t}  + \dfrac{1}{1 - \delta} \bigg \{ \dfrac{\delta}{2(1 - \delta)}\big (Uz \\
				& + u - r_t \mathbf{1}\big)^T \big (\alpha \mathbb{C}_{B}\alpha^T  + \Sigma \mathbb{C}_{W}\Sigma^T \big)^{-1} \big (Uz + u - r_t \mathbf{1}\big)  + \delta r_{t} \bigg \} \mathbf{\hbar} = 0, \quad t < s\\
				& \mathbf{\hbar}(T; s, z) = 1.
			\end{array}
			\right.
		\end{equation}
		It is easy to show that 
		\begin{equation*}
			\mathbf{\hbar}^{\star}(t; s, z) = \dfrac{1}{(1 - \delta)} \bigg [\dfrac{1}{2} z^{T} M(t) z + n^{T}(t) z + q(t) \bigg]
		\end{equation*}
		solves equation \eqref{37} which complete the proof.
		$\square$
	\end{proof}
	\begin{theorem}
		Assume \textbf{H1)-H2)} and $f_1$ and $f_2$ defined as in \eqref{28} and equation \eqref{25} respectively. If equations \eqref{30}-\eqref{31} and \eqref{32} have as solutions $M(t), n(t)$ and $q(t)$ respectively, then the problem \eqref{20} has a solution:
		\begin{equation}\label{38}
			\hat{F}\big (t, x, \hat{\omega}_{1}(t),  \hat{\omega}^{ \hat{Z}^{\psi}_{t}}_{2}(t) \big) = \dfrac{1}{\delta} \big (x + \hat{\omega}_{1}(t) \big)^{\delta} \big (\hat{\omega}^{ \hat{Z}^{\psi}_{t}}_{2}(t) \big)^{1 - \delta},
		\end{equation}
		where $\hat{\omega}_{1}(t)$ and $\hat{\omega}^{ \hat{Z}^{\psi}_{t}}_{2}(t)$ are given by equation \eqref{27} and equations \eqref{34} respectively.
	\end{theorem}
	\section{Verification theorem}
	We consider the risk sensitive portfolio optimization problem treated in \cite{nagai2001risk} for the proof of the next theorem.	
	Since in the power utility function (equation \eqref{7}), we have $\delta \in (-\infty , 0) \cup (0,1)$, let us first study the case $\delta \in (0,1)$.
	\begin{theorem}Case of $\delta \in (0,1)$.\\
		Let $\mathcal{L}^{+}_{T}(x, 0)$ the space of admissible strategies. Assume \textbf{H1), H2), H3), H4)} and \textbf{H5)}. Furthermore, we assume that equations \eqref{30}-\eqref{31} and \eqref{32} have solutions $M(t), n(t)$ and $q(t)$ respectively. Then the strategy $\hat{\rho} \big (t, x, \hat{\omega}_{1}(t),  \hat{\omega}^{ \hat{Z}^{\psi}_{t}}_{2}(t)\big), \; \hat{C}  \big (t, x, \hat{\omega}_{1}(t),  \hat{\omega}^{ \hat{Z}^{\psi}_{t}}_{2}(t)\big), \; \hat{\beta}  \big (t, x, \hat{\omega}_{1}(t),  \hat{\omega}^{ \hat{Z}^{\psi}_{t}}_{2}(t)\big) \in \mathcal{L}^{+}_{T}(x, 0)$ is an optimal strategy for the problem \eqref{6}, where:
		\begin{align*}
			\hat{\rho} \big (t, x, \hat{\omega}_{1}(t),  \hat{\omega}^{ \hat{Z}^{\psi}_{t}}_{2}(t)\big) & = \dfrac{1}{1-\delta} \big (\alpha \mathbb{C}_{B}\alpha^T  + \Sigma \mathbb{C}_{W}\Sigma^T \big)^{-1} \bigg [U\hat{Z}^{\psi}_{t} + u - r_t \mathbf{1} + (1-\delta) \big (PU^T \\
			&+ \mathbb{C}_{B} \alpha^T \big)^T \dfrac{D \hat{\omega}^{ \hat{Z}^{\psi}_{t}}_{2}(t)}{\hat{\omega}_{2}(t)} \bigg]\big (x + \hat{\omega}_{1}(t) \big), \\
			& \\
			\hat{C}  \big (t, x, \hat{\omega}_{1}(t),  \hat{\omega}^{ \hat{Z}^{\psi}_{t}}_{2}(t)\big) & = \dfrac{x + \hat{\omega}_{1}(t) }{\hat{\omega}^{ \hat{Z}^{\psi}_{t}}_{2}(t)},\\
			& \\
			\hat{\beta}  \big (t, x, \hat{\omega}_{1}(t),  \hat{\omega}^{ \hat{Z}^{\psi}_{t}}_{2}(t)\big) & = a(t) \bigg [\bigg (\dfrac{a(t)}{\mu (t)} \bigg)^{- \frac{1}{1-\delta}}  \bigg ( \dfrac{x + \hat{\omega}_{1}(t) }{\hat{\omega}^{ \hat{Z}^{\psi}_{t}}_{2}(t)} \bigg ) - x \bigg].
		\end{align*}
	\end{theorem}
	\begin{proof}
		For $\big (\rho ,  C, \beta \big) \in \mathcal{L}^{+}_{T} (x, 0)$, we have :
		\begin{eqnarray*}
			&& \mathrm{d} \bigg [e^{-\int_{0}^{t}\big (\theta(s) + \mu(s) \big)ds } \hat{F} \big (t, X_t , \hat{\omega}_{1}(t),  \hat{\omega}^{ \hat{Z}^{\psi}_{t}}_{2}(t) \big) \bigg]\\ 
			 && = - \big (\theta (t) + \mu (t)\big) e^{-\int_{0}^{t}\big (\theta(s) + \mu(s) \big)ds } \hat{F} + e^{-\int_{0}^{t}\big (\theta(s) + \mu(s) \big)ds } \hat{F}_{t} \big (t, X_t , \hat{\omega}_{1}(t),  \hat{\omega}^{ \hat{Z}^{\psi}_{t}}_{2}(t) \big) dt +\\
			 && e^{-\int_{0}^{t}\big (\theta(s) + \mu(s) \big)ds } \hat{F}_{x} \big (t, X_t , \hat{\omega}_{1}(t),   \hat{\omega}^{ \hat{Z}^{\psi}_{t}}_{2}(t) \big) d X_{t} + e^{-\int_{0}^{t}\big (\theta(s) + \mu(s) \big)ds } \hat{F}_{\omega_1} \big (t, X_t , \hat{\omega}_{1}(t),   \hat{\omega}^{ \hat{Z}^{\psi}_{t}}_{2}(t) \big) d \hat{\omega}_{1}(t)\\
			 &&  + e^{-\int_{0}^{t}\big (\theta(s) + \mu(s) \big)ds } \hat{F}_{\omega_2} \big (t, X_t , \hat{\omega}_{1}(t),   \hat{\omega}^{ \hat{Z}^{\psi}_{t}}_{2}(t) \big) d   \hat{\omega}^{ \hat{Z}^{\psi}_{t}}_{2}(t)\\
			 && = e^{-\int_{0}^{t}\big (\theta(s) + \mu(s) \big)ds } \bigg \{ - \big (\theta(t) + \mu (t)\big)  \hat{F} \big (t, X_t , \hat{\omega}_{1}(t),   \hat{\omega}^{ \hat{Z}^{\psi}_{t}}_{2}(t) \big) +   \hat{F}_{t} \big (t, X_t , \hat{\omega}_{1}(t),   \hat{\omega}^{ \hat{Z}^{\psi}_{t}}_{2}(t) \big)dt \\
			&& + \hat{F}_{x} \bigg [\bigg (r_{t} X_{t}  + \hat{\rho}^{T}(t) \big (U\hat{Z}^{\psi}_{t} + u - r_{t}\mathbf{1} \big) - \hat{C}(t) - \hat{\beta} (t) + R(t) \bigg)dt + \hat{\rho}^{T}(t)d \nu_{t}  \bigg] \\
			&& + \hat{F}_{\omega_1} \bigg [- \hat{f}_{1} dt + \hat{\lambda}^{T}_{1} (t) d \nu_{t} \bigg] + \hat{F}_{\omega_2}  \bigg [- \hat{f}_{2} dt + \hat{\lambda}^{T}_{2} (t) d \nu_{t} \bigg]\bigg \} \\
			&& = e^{-\int_{0}^{t}\big (\theta(s) + \mu(s) \big)ds } \bigg \{ - \big (\theta (t) + \mu(t)\big) \hat{F} + \bigg [ \hat{F}_{t} + \hat{F}_{x} \bigg [ r_t X_t + \hat{\rho}^{T}(t) \big ( U  \hat{Z}^{\psi}_{t} + u - r_{t}\mathbf{1} \big ) - \hat{C}(t) - \hat{\beta}(t) + R(t) \bigg ] \\
			&& - \hat{F}_{\omega_1} \hat{f}_{1} -  \hat{F}_{\omega_2} \hat{f}_{2} \bigg ] dt + \bigg [ \hat{F}_{x} \hat{\rho}^{T}(t) + \hat{F}_{\omega_1} \hat{\lambda}^{T}_{1}(t) +  \hat{F}_{\omega_2} \hat{\lambda}^{T}_{2}(t) \bigg ] d \nu_{t} \bigg \} \\
			&& = e^{-\int_{0}^{t}\big (\theta(s) + \mu(s) \big)ds } \bigg \{ - \big (\theta (t) + \mu(t)\big) \hat{F} + \bigg [ \hat{F}_{t} + \hat{F}_{x} \bigg [ r_{t} X_{t} + \hat{\rho}^{T}(t) \big (U\hat{Z}^{\psi}_{t} + u   - r_{t}\mathbf{1} \big) - \hat{C}(t) - \hat{\beta}(t) + R(t) \bigg ]\\
			&& - \hat{F}_{\omega_1} \hat{f}_{1} - \hat{F}_{\omega_2} \hat{f}_{2} \bigg ] dt + \bigg [ \hat{F}_{x} \hat{\rho}^{T}(t) +  \hat{F}_{\omega_1} \hat{\lambda}_1^{T}(t) +  \hat{F}_{\omega_2} \hat{\lambda}_2^{T}(t) \bigg ] d \nu_{t} \bigg \}  \\
		&& = e^{-\int_{0}^{t}\big (\theta(s) + \mu(s) \big)ds } \bigg \{ - \big (\theta (t) + \mu(t)\big) \hat{F} + \mathcal{D}^{\rho,  C, \beta} \hat{F} \big (t, X_t , \hat{\omega}_{1}(t),   \hat{\omega}^{ \hat{Z}^{\psi}_{t}}_{2}(t) \big) dt \\
			&& + \bigg [  \hat{F}_{x} \big (t, X_t , \hat{\omega}_{1}(t),   \hat{\omega}^{ \hat{Z}^{\psi}_{t}}_{2}(t) \big) \hat{\rho}^{T}(t) +  \hat{F}_{\omega_1} \big (t, X_t , \hat{\omega}_{1}(t),   \hat{\omega}^{ \hat{Z}^{\psi}_{t}}_{2}(t) \big) \hat{\lambda}_1^{T}(t) +  \hat{F}_{\omega_2} \big (t, X_t , \hat{\omega}_{1}(t),   \hat{\omega}^{ \hat{Z}^{\psi}_{t}}_{2}(t) \big) \hat{\lambda}_2^{T}(t) \bigg ] d \nu_{t} \bigg \} \\
			&& \leqslant \bigg \{ \mathcal{D}^{\rho, C, \beta}  \hat{F}\big (t, X_t , \hat{\omega}_{1}(t),   \hat{\omega}^{ \hat{Z}^{\psi}_{t}}_{2}(t) \big)dt + \bigg [\hat{F}_{x} \big (t, X_t , \hat{\omega}_{1}(t),   \hat{\omega}^{ \hat{Z}^{\psi}_{t}}_{2}(t) \big) \hat{\rho}^{T}(t)  +  \hat{F}_{\omega_1} \big (t, X_t , \hat{\omega}_{1}(t),   \hat{\omega}^{ \hat{Z}^{\psi}_{t}}_{2}(t) \big) \hat{\lambda}_1^{T}(t)\\
				&& +  \hat{F}_{\omega_2} \big (t, X_t , \hat{\omega}_{1}(t),   \hat{\omega}^{ \hat{Z}^{\psi}_{t}}_{2}(t) \big) \hat{\lambda}_2^{T}(t) \bigg ] d \nu_{t} \bigg \} e^{- \int_{0}^{t}\big (\theta(s) + \mu(s) \big)ds }.
		\end{eqnarray*}
		Hence 
		\begin{align}\nonumber\label{50}
			& \mathrm{d} \bigg [e^{-\int_{0}^{t}\big (\theta(s) + \mu(s) \big)ds } \hat{F} \big (t, X_t , \hat{\omega}_{1}(t),  \hat{\omega}^{ \hat{Z}^{\psi}_{t}}_{2}(t) \big) \bigg] \\\nonumber
			&  \leqslant \bigg \{ - \bigg [ V (\hat{C}_{t}) + \mu(t) V \bigg (X_t + \dfrac{\hat{\beta}(t) }{a(t)} \bigg) \bigg ] dt + \bigg [ \bigg (\dfrac{\hat{\omega}^{ \hat{Z}^{\psi}_{t}}_{2}(t) }{X_{t} +  \hat{\omega}_{1}(t)}\bigg )^{1-\delta} \hat{\rho}^{T}(t) \\\nonumber
			& + \bigg (\dfrac{1-\delta}{\delta} \bigg) \bigg (\dfrac{X_{t} +  \hat{\omega}_{1}(t)}{\hat{\omega}^{ \hat{Z}^{\psi}_{t}}_{2}(t)} \bigg)^{\delta} \big (\alpha \mathbb{C}_{B}\alpha^T  + \Sigma \mathbb{C}_{W}\Sigma^T \big)^{-1} \big (PU^{T} + \mathbb{C}_{B} \alpha^{T} \big) \dfrac{\partial}{\partial t}  \hat{\omega}^{ \hat{Z}^{\psi}_{t}}_{2}(t) \bigg ] d \nu_{t} \bigg \}  e^{- \int_{0}^{t}\big (\theta(s) + \mu(s) \big)ds }.
		\end{align}
		Let us set 
		\begin{equation*}
			\xi^{R} : = \inf \bigg \{ t > 0; X(t) +  \hat{\omega}_{1}(t) > R \bigg \} \wedge \inf \bigg \{ t > 0; \big | \hat{Z}^{\psi}_{t}  \big| > R \bigg \} \wedge \inf\bigg \{ t > 0; \; \int_{0}^{t}\bigg | \dfrac{\rho(s)}{ X(s) + \hat{\omega}_{1}(s)} \bigg|^2 ds > R \bigg \}. 
		\end{equation*}
		Then we have
		\begin{align*}
			\hat{F}\big (0, x, \hat{\omega}_{1}(0), \hat{\omega}_{2}^{z}(0) \big) & \geqslant \mathbb{E} \bigg [ \int_{0}^{T\wedge \xi^R} e^{- \int_{0}^{t} \big (\theta(s) + \mu(s) \big)ds } \bigg [ V(C_t) + \mu(t) V \bigg (X_t + \dfrac{\beta(t)}{a(t)} \bigg) \bigg ] dt \\
			& + e^{- \int_{0}^{T \wedge \xi^R} \big (\theta(t) + \mu(t)\big) dt } \hat{F} \bigg (T\wedge \xi^R, X (T\wedge \xi^R), \hat{\omega}_{1} (T\wedge \xi^R),  \hat{\omega}^{ \hat{Z}^{\psi}_{T\wedge \xi^R }}_{2}(T\wedge \xi^R)  \bigg)\bigg ] 
		\end{align*}
		From the monotone convergence theorem and Fatou's lemma, it follows that:
		\begin{align}\nonumber\label{51}
			\hat{F}\big (0, x, \hat{\omega}_{1}(0), \hat{\omega}_{2}^{z}(0) \big) & \geqslant \mathbb{E} \bigg [ \int_{0}^{T\wedge \xi^R} e^{- \int_{0}^{t} \big (\theta(s) + \mu(s) \big)ds } \bigg [ V(C_t) + \mu(t) V \bigg (X_t + \dfrac{\beta(t)}{a(t)} \bigg) \bigg ] dt \\
			& + e^{- \int_{0}^{T} \big (\theta(t) + \mu(t)\big) dt }V(X(T)) \bigg] = \mathcal{V} \big (0, x, z; \rho , C, \beta \big). 
		\end{align}
		Setting 
		\begin{align}\nonumber \label{52}
			\hat{h} (t) & = \dfrac{\hat{\rho} \big (t, X_t , \hat{\omega}_{1}(t),  \hat{\omega}^{ \hat{Z}^{\psi}_{t}}_{2}(t) \big)}{\hat{X}(t) + \hat{\omega}_{1}(t) }\\\nonumber
			& = \dfrac{1}{\hat{X}(t) + \hat{\omega}_{1}(t)} \bigg \{ \dfrac{1}{1-\delta} \big (\alpha \mathbb{C}_{B}\alpha^T  + \Sigma \mathbb{C}_{W}\Sigma^T \big)^{-1} \bigg [U\hat{Z}^{\psi}_{t} + u - r_{t}\mathbf{1} +  (1-\delta)  \big (PU^{T} \\\nonumber
			& + \mathbb{C}_{B} \alpha^{T} \big)^{T} \dfrac{D \hat{\omega}^{ \hat{Z}^{\psi}_{t}}_{2}(t)}{\hat{\omega}^{\hat{Z}^{\psi}_{t}}_{2}(t)} \bigg ] \big (X_{t} + \hat{\omega}_{1}(t) \big) \bigg \} \\
			& = \dfrac{1}{(1 - \delta)} \big (\alpha \mathbb{C}_{B}\alpha^T  + \Sigma \mathbb{C}_{W}\Sigma^T \big)^{-1} \bigg [U\hat{Z}^{\psi}_{t} + u - r_{t}\mathbf{1} + (1-\delta)\big (PU^{T} + \mathbb{C}_{B} \alpha^{T} \big)^{T} \dfrac{D \hat{\omega}^{ \hat{Z}^{\psi}_{t}}_{2}(t)}{\hat{\omega}^{\hat{Z}^{\psi}_{t}}_{2}(t)} \bigg ] .  
		\end{align}
		It follows from \eqref{19} and \eqref{52} that
		\begin{align*}
			& \mathrm{d} \bigg \{\hat{X}(t) + \hat{\omega}_{1}(t)  \bigg \} \\
			& = d  \hat{X}(t) + d \hat{\omega}_{1}(t) \\
			& = \bigg \{r_{t} \hat{X}_{t} + \hat{\rho}^{T}(t) \big (U \hat{Z}^{\psi}_{t} + u - r_{t} \mathbf{1}  \big) - \hat{C}(t) - \hat{\beta}(t) + R(t) \bigg \} dt + \hat{\rho}^{T}(t) d \nu_{t} + \bigg [\big (r_{t} + a(t) \big) \hat{\omega}_{1}(t) - R(t) \bigg] dt \\
			& = \bigg \{r_{t} \big (\hat{X}_{t}  + \hat{\omega}_{1} (t)\big) + a(t) \hat{\omega}_{1}(t)  +  \hat{\rho}^{T}(t)  \big (U \hat{Z}^{\psi}_{t} + u - r_{t} \mathbf{1}  \big) - H(t) \bigg (\dfrac{ \hat{X}(t) + \hat{\omega}_{1}(t) }{\hat{\omega}^{\hat{Z}^{\psi}_{t}}_{2}(t)} \bigg)  + a(t)\hat{X}(t)  \bigg \} dt   \\
			& + \hat{\rho}^{T}(t) d \nu_{t}\\
			& = \bigg \{ \big ( r_{t}  + a(t) \big ) \big (\hat{X}_{t}  + \hat{\omega}_{1} (t)\big)  +  \hat{\rho}^{T}(t)  \big (U \hat{Z}^{\psi}_{t} + u - r_{t} \mathbf{1}  \big) - H(t) \bigg (\dfrac{ \hat{X}(t) + \hat{\omega}_{1}(t) }{\hat{\omega}^{\hat{Z}^{\psi}_{t}}_{2}(t)} \bigg)  \bigg \} dt   + \hat{\rho}^{T}(t) d \nu_{t}.
		\end{align*}
		\begin{equation*}
			\Rightarrow \dfrac{d \bigg \{ \hat{X}(t) + \hat{\omega}_{1}(t) \bigg \}}{ \hat{X}_{t}  + \hat{\omega}_{1}(t)} = \bigg [r_{t} + a(t) + \hat{h}^{T}(t)\big (U \hat{Z}^{\psi}_{t} + u - r_{t} \mathbf{1}  \big) - \dfrac{H(t)}{\hat{\omega}^{\hat{Z}^{\psi}_{t}}_{2}(t)}  \bigg] dt  + \hat{h}^{T}(t) d \nu_{t}, 
		\end{equation*}
		i.e.,
		\begin{equation}\label{100}
			\hat{X}(t) + \hat{\omega}_{1}(t) = \big (x + \hat{\omega}_{1}(0) \big) \exp \bigg (\int_{0}^{t} \bigg [ r_{s} + a(s) + \hat{h}^{T}(s) \big (U \hat{Z}^{\psi}_{s} + u - r_{s} \mathbf{1}  \big) - \dfrac{H(s)}{\hat{\omega}^{\hat{Z}^{\psi}_{s}}_{2}(s) } \bigg] ds + \int_{0}^{t} \hat{h}^{T}(s) d \nu_{s} \bigg) > 0.
		\end{equation}
		Hence 
		\begin{equation*}
			\big (\hat{\rho}, \hat{C}, \hat{\beta} \big) \in \mathcal{L}^{+}_{T}(x, 0).
		\end{equation*}
		Consider 
		\begin{equation*}
			\hat{\xi}^{R} : = \inf \bigg \{t>0; \; \hat{X}(t) + \hat{\omega}_{1}(t) > R  \bigg \} \wedge \inf \bigg \{t >0; | \hat{Z}^{\psi}_{t} | >0 \bigg \}.
		\end{equation*}
		Then we see that $\exists \mathcal{Q}_{T}$ such that for $t < T \wedge \hat{\xi}^{R}$,
		\begin{equation*}
			\int_{0}^{t} \big |\hat{h}(s) \big|^2 ds = \int_{0}^{t} \bigg |\dfrac{\hat{\rho}(s) }{\hat{X}_{s} + \hat{\omega}_{1}(s)} \bigg|^{2} \leqslant \int_{0}^{T} \big |\hat{\rho}(s)\big|^{2} = \mathcal{Q}_{T}.
		\end{equation*}
		Similarly, we can show that 
		\begin{align}\nonumber \label{53}
			\hat{F}\big (0, x, \hat{\omega}_{1}(0), \hat{\omega}^{z}_{2}(0) \big) & = \mathbb{E} \bigg [\int_{0}^{T \wedge \hat{\xi}^{R}} e^{- \int_{0}^{t} (\theta(s) + \mu(s)) ds } \bigg [V(\hat{C}_{t}) + \mu (t) V \big (\hat{X}_{t} + \dfrac{\hat{\beta}_{t} }{a(t)}\big)\bigg]dt \\
			& + e^{- \int_{0}^{T \wedge \hat{\xi}^{R}} (\theta(t) + \mu(t))dt }  \hat{F} \bigg (T \wedge \hat{\xi}^{R}, \hat{X} (T \wedge \hat{\xi}^{R}), \hat{\omega}_{1}(T \wedge \hat{\xi}^{R}), \hat{\omega}^{\hat{Z}^{\psi}_{T \wedge \hat{\xi}^{R}}  }_{2}(T \wedge \hat{\xi}^{R})  \bigg) \bigg].
		\end{align}
		Note that 
		\begin{align*}
			\hat{X}(t) + \dfrac{\hat{\beta}(t) }{a(t)} & = \hat{X}(t) + \bigg (\dfrac{a(t)}{\mu (t)} \bigg)^{-\frac{1}{1-\delta}} \bigg (\dfrac{\hat{X}(t) + \hat{\omega}_{1}(t)}{\hat{\omega}^{\hat{Z}^{\psi}_{t} }_{2}(t)} \bigg) - \hat{X}(t) \\
			& =  \bigg (\dfrac{a(t)}{\mu (t)} \bigg)^{-\frac{1}{1-\delta}} \bigg (\dfrac{\hat{X}(t) + \hat{\omega}_{1}(t)}{\hat{\omega}^{\hat{Z}^{\psi}_{t} }_{2}(t)} \bigg).
		\end{align*}
		Thus, from the monotone convergence theorem and for all fixed $T>0$, we have
		\begin{align}\nonumber\label{54}
		& \lim_{R \to + \infty}  \mathbb{E} \bigg [\int_{0}^{T \wedge \hat{\xi}^{R}} e^{- \int_{0}^{t} (\theta(s) + \mu(s)) ds } \bigg [V(\hat{C}_{t}) + \mu (t) V \big (\hat{X}_{t} + \dfrac{\hat{\beta}_{t} }{a(t)}\big)\bigg]dt \bigg ]\\
		& =  \mathbb{E} \bigg [\int_{0}^{T} e^{- \int_{0}^{t} (\theta(s) + \mu(s)) ds } \bigg [V(\hat{C}_{t}) + \mu (t) V \big (\hat{X}_{t} + \dfrac{\hat{\beta}_{t} }{a(t)}\big)\bigg]dt \bigg ].
		\end{align}
		Next, we need to show that 
		\begin{align}\nonumber\label{55}
		& \lim_{R \to + \infty} \mathbb{E} \bigg [ e^{- \int_{0}^{T \wedge \hat{\xi}^{R}} (\theta(t) + \mu(t))dt }  \hat{F} \bigg (T \wedge \hat{\xi}^{R}, \hat{X} (T \wedge \hat{\xi}^{R}), \hat{\omega}_{1}(T \wedge \hat{\xi}^{R}), \hat{\omega}^{\hat{Z}^{\psi}_{T \wedge \hat{\xi}^{R}}  }_{2}(T \wedge \hat{\xi}^{R})  \bigg) \bigg]\\
		&  = \mathbb{E} \bigg [ e^{-\int^{T}_{0} (\theta (t) + \mu(t))dt} V (\hat{X}_{T})\bigg ].
		\end{align}
		For that, we just need ton prove the uniform integrability of the random variables
		\begin{equation}\label{56}
			\bigg \{\hat{F} \bigg (T \wedge \hat{\xi}^{R}, \hat{X} (T \wedge \hat{\xi}^{R}), \hat{\omega}_{1}(T \wedge \hat{\xi}^{R}), \hat{\omega}^{\hat{Z}^{\psi}_{T \wedge \hat{\xi}^{R}}  }_{2}(T \wedge \hat{\xi}^{R})  \bigg) \bigg \}.
		\end{equation} 
		We notice that
		\begin{equation*}
			\hat{F} \bigg (T \wedge \hat{\xi}^{R}, \hat{X} (T \wedge \hat{\xi}^{R}), \hat{\omega}_{1}(T \wedge \hat{\xi}^{R}), \hat{\omega}^{\hat{Z}^{\psi}_{T \wedge \hat{\xi}^{R}}  }_{2}(T \wedge \hat{\xi}^{R})  \bigg)  \leqslant  \mathcal{Q} \bigg [ \hat{X} \big(T \wedge \hat{\xi}^{R} \big) + \hat{\omega}_{1}\big(T \wedge \hat{\xi}^{R} \big)\bigg]^{\delta} \bigg [ \hat{\omega}^{\hat{Z}^{\psi}_{T \wedge \hat{\xi}^{R}}  }_{2} (T \wedge \hat{\xi}^{R}) \bigg]^{1-\delta}.
		\end{equation*}
		Therefore, to prove the uniform integrability of \eqref{56}, we need to prove that $\exists \mathcal{Q} > 0$ such that 
		\begin{equation}\label{57}
			\mathbb{E} \bigg [\bigg \{ \hat{X} \big(T \wedge \hat{\xi}^{R} \big) + \hat{\omega}_{1}\big(T \wedge \hat{\xi}^{R} \big) \bigg \}^{\zeta} \bigg \{  \hat{\omega}^{\hat{Z}^{\psi}_{T \wedge \hat{\xi}^{R}}  }_{2} (T \wedge \hat{\xi}^{R}) \bigg  \}^{(1-\delta)( 1 + \kappa)}    \bigg] \leqslant \mathcal{Q},
		\end{equation}
		where $\zeta : = (1 + \kappa) \delta$. The following equations are prove in \citep{hata2020optimal}
		\begin{equation}\label{58}
			\big \{\hat{\omega}_{2}^{z}(t) \big \}^{(1 - \delta)(1 + \kappa)} \leqslant \mathcal{Q}_{1} e^{\tilde{r} (t, z)}
		\end{equation}
		and 
		\begin{equation}\label{59}
			\bigg \{\hat{X} (t) + \hat{\omega}_{1}(t)\bigg \}^{\zeta} e^{\tilde{r} (t , Z^{\psi}_{t} )} \leqslant \big \{x + \hat{\omega}_{1}(0)  \big \}^{\zeta} e^{\zeta  \int_{0}^{T} \mu(t)dt + \tilde{r}(0,z)} \Phi_{t}^{\hat{h}, \zeta}\quad P_{r}-\text{a.s},
		\end{equation}
		where
		\begin{equation*}
			\tilde{r}(t, z) = \dfrac{1}{2} z^{T} M(t) z + n^{T}(t) z + q(t)
		\end{equation*}
		and $\Phi_{t}^{\hat{h}, \zeta}$ is defined by
		\begin{align*}
			\Phi_{t}^{\hat{h}, \zeta} & = \exp \bigg \{  -\dfrac{1}{2} \int_{0}^{t} \bigg [ \zeta \hat{h}(s) + \big (\alpha \mathbb{C}_{B}\alpha^{T}  + \Sigma \mathbb{C}_{W} \Sigma^{T}\big)^{-1} \big (PU^{T} + \mathbb{C}_{B} \alpha^{T} \big)^{T} D \tilde{r} (s, \hat{Z}^{\psi}_{s}) \bigg ]^{T} \big (\alpha \mathbb{C}_{B}\alpha^{T}  \\
			& + \Sigma \mathbb{C}_{W} \Sigma^{T}\big)\bigg [ \zeta \hat{h}(s) + \big (\alpha \mathbb{C}_{B}\alpha^{T}  + \Sigma \mathbb{C}_{W} \Sigma^{T}\big)^{-1} \big (PU^{T} + \mathbb{C}_{B} \alpha^{T} \big)^{T} D \tilde{r} (s, \hat{Z}^{\psi}_{s}) \bigg ] ds   \\
			& + \int_{0}^{t} \bigg [ \zeta \hat{h}(s) + \big (\alpha \mathbb{C}_{B}\alpha^{T}  + \Sigma \mathbb{C}_{W} \Sigma^{T}\big)^{-1} \big (PU^{T} + \mathbb{C}_{B} \alpha^{T} \big)^{T} D \tilde{r} (s, \hat{Z}^{\psi}_{s}) \bigg ]^{T} d \nu_{s} \bigg \} .
		\end{align*}  
		Using equation \eqref{58} and equation \eqref{59}, we have
		\begin{align*}
			& \big \{\hat{X}(t) + \hat{\omega}_{1}(t) \big \}^{\zeta} \big \{ \hat{\omega}^{z}_{2}(t) \big \}^{(1 - \delta) (1 + \kappa)} e^{\tilde{r} (t,  \hat{Z}^{\psi}_{t})}\\
			 & \leqslant \big \{x + \hat{\omega}_{1}(0)  \big \}^{\zeta} \mathcal{Q}_{1} e^{\tilde{r}(t,z) } e^{\zeta \int_{0}^{T} \mu(t)dt + \tilde{r}(0,z)} \Phi^{\hat{h}, \zeta}_{t} \; P_{r}-\text{a.s.}\\
			&\Rightarrow  \bigg \{\hat{X}\big(T \wedge \hat{\xi}^{R} \big) + \hat{\omega}_{1} \big(T \wedge \hat{\xi}^{R} \big) \bigg \}^{\zeta} \bigg \{ \hat{\omega}^{\hat{Z}^{\psi}_{T \wedge \hat{\xi}^{R} }}_{2} \big(T \wedge \hat{\xi}^{R} \big) \bigg \}^{(1 - \delta) (1 + \kappa)} e^{\tilde{r} \big (T \wedge \hat{\xi}^{R} ,  \hat{Z}^{\psi}_{T \wedge \hat{\xi}^{R} } \big)}\\
			& \leqslant \big \{x + \hat{\omega}_{1}(0)  \big \}^{\zeta} \mathcal{Q}_{1} e^{\tilde{r}\big(T \wedge \hat{\xi}^{R} ,\hat{Z}^{\psi}_{T \wedge \hat{\xi}^{R} } \big) }  \exp \bigg \{\zeta \int_{0}^{T} \mu(t)dt + \tilde{r}(0,z) \bigg \} \Phi^{\hat{h}, \zeta}_{T \wedge \hat{\xi}^{R} } \; P_{r}-\text{a.s.}\\
			& \Rightarrow  \mathbb{E} \bigg [  \bigg \{\hat{X}\big(T \wedge \hat{\xi}^{R} \big) + \hat{\omega}_{1} \big(T \wedge \hat{\xi}^{R} \big) \bigg \}^{\zeta} \bigg \{ \hat{\omega}^{\hat{Z}^{\psi}_{T \wedge \hat{\xi}^{R} }}_{2} \big(T \wedge \hat{\xi}^{R} \big) \bigg \}^{(1 - \delta) (1 + \kappa)} \bigg] \\			
& \leqslant \big \{x + \hat{\omega}_{1}(0)  \big \}^{\zeta} \mathcal{Q}_{1} \exp \bigg \{\zeta \int_{0}^{T} \mu(t)dt + \tilde{r}(0,z) \bigg \} ,
		\end{align*}
		since $\mathbb{E} \big [  \Phi^{\hat{h}, \zeta}_{T \wedge \hat{\xi}^{R} } \big] = 1$.  Hence equation \eqref{57} and equation \eqref{55} follow.\par
		Finally, from equations \eqref{53}, \eqref{54} and equation \eqref{55}, we see that  $\forall T > 0$ fixed,
		\begin{equation*}
			\varphi \big (0, x, z \big) = \mathcal{V} \big (0, x, z; \hat{\rho}, \hat{C}, \hat{\beta} \big) = \hat{F} \big (0, x, \hat{\omega}_{1}(0), \hat{\omega}^{z}_{2}(0) \big). 
		\end{equation*}
		$\square$
	\end{proof}
		In the risk averse case, i.e. , $\delta \in (-\infty , 0)$, we consider the space of admissible strategies by:
		\begin{align*}
			\mathcal{L}^{-}_{T} (x, 0) & = \bigg \{ \big (\rho , C, \beta \big) \in \overline{\mathcal{L}}_{T}(0, x); \big |\rho (t) \big| \leqslant \tilde{\upsilon} \big (1 + \big |\hat{Z}^{\psi}_{t} \big| \big) \big [X(t) + \hat{\omega}_{1}(t) \big], \; \forall \epsilon >0\; \text{such that}\\
			& C(t) \leqslant \overline{\upsilon} e^{\epsilon \big (1 + \big |\hat{Z}^{\psi}_{t} \big|^2 \big)} \big [X(t) + \hat{\omega}_{1}(t) \big], \big |\beta (t) - a(t)\hat{\omega}_{1}(t) \big| \leqslant  \overline{\upsilon} e^{\epsilon \big (1 + \big |\hat{Z}^{\psi}_{t} \big|^2 \big)} \big [X(t) + \hat{\omega}_{1}(t) \big]\bigg \}
		\end{align*}
		where $\overline{\upsilon}$ and $\tilde{\upsilon}$ are some constants depending on $\rho, C$ and $\beta$.
	\begin{theorem}Case of $\delta \in (-\infty , 0)$.\\
		Assume \textbf{H1), H2), H3), H4)} and \textbf{H5)}. Furthermore, we assume that equations \eqref{30}-\eqref{31} and \eqref{32} have solutions $M(t), n(t)$ and $q(t)$ respectively.  Then $\big (\hat{\rho}(t), \hat{C}(t), \hat{\beta}(t) \big) \in \mathcal{L}^{-}_{T} (x, 0)$ is an optimal strategy for the problem \eqref{6}, where 
		\begin{align*}
			\hat{\rho} \big (t, x, \hat{\omega}_{1}(t),  \hat{\omega}^{ \hat{Z}^{\psi}_{t}}_{2}(t)\big) & = \dfrac{1}{1-\delta} \big (\alpha \mathbb{C}_{B}\alpha^T  + \Sigma \mathbb{C}_{W}\Sigma^T \big)^{-1} \bigg [U\hat{Z}^{\psi}_{t} + u - r_t \mathbf{1} + (1-\delta) \big (PU^T \\
			& + \mathbb{C}_{B} \alpha^T \big)^T \dfrac{D \hat{\omega}^{ \hat{Z}^{\psi}_{t}}_{2}(t)}{\hat{\omega}_{2}(t)} \bigg]\big (x + \hat{\omega}_{1}(t) \big) \\
			&\\
			\hat{C}  \big (t, x, \hat{\omega}_{1}(t),  \hat{\omega}^{ \hat{Z}^{\psi}_{t}}_{2}(t)\big) & = \dfrac{x + \hat{\omega}_{1}(t) }{\hat{\omega}^{ \hat{Z}^{\psi}_{t}}_{2}(t)}\\
			&\\
			\hat{\beta}  \big (t, x, \hat{\omega}_{1}(t),  \hat{\omega}^{ \hat{Z}^{\psi}_{t}}_{2}(t)\big) & = a(t) \bigg [\bigg (\dfrac{a(t)}{\mu (t)} \bigg)^{- \frac{1}{1-\delta}}  \bigg ( \dfrac{x + \hat{\omega}_{1}(t) }{\hat{\omega}^{ \hat{Z}^{\psi}_{t}}_{2}(t)} \bigg ) - x \bigg].
		\end{align*}
		Moreover, $\varphi \big (0, x, z \big) = \hat{F} \big (0, x, \hat{\omega}_{1}(0), \hat{\omega}^{z}_{2}(0) \big)$. 
	\end{theorem}
	\begin{proof}
		Denote $\hat{F}_{\delta} , \hat{\omega}_{2, \delta}$ and $V_{\delta}$ for dependence of $\hat{F}, \hat{\omega}_{2}$ and $V$ on $\delta$ respectively. We need to prove first that $\forall \big (\rho , C, \beta \big) \in \mathcal{L}^{-}_{T} (x, 0)$,
		\begin{equation}\label{60}
			\mathcal{V} \big (0, x, z; \rho, C, \beta \big) \leqslant \hat{F} \big (0, x, \hat{\omega}_{1}(0), \hat{\omega}^{z}_{2}(0) \big).
		\end{equation}
		Let us recall that $\hat{Z}^{\psi}_{t}$ satisfies equation \eqref{14}.  Taking equation \eqref{50} into account, we set 
		\begin{align}\nonumber\label{61}
			\xi^{\delta}_{n} : = & \inf \bigg \{ t > 0; \big |\hat{Z}^{\psi}_{t}  \big| > 0 \bigg \} \wedge \inf \bigg \{ t > 0; X(t) + \hat{\omega}_{1}(t) < \dfrac{1}{n} \bigg \} \wedge \inf \bigg \{ t > 0; \\
			& \int_{0}^{t} \bigg |\dfrac{\rho(s)}{X(s) + \hat{\omega}_{1}(s) } \bigg|^2 ds > n \bigg \} \wedge \inf \bigg \{ t > 0; \int_{0}^{t} \bigg [\big (C(s) \big)^{\delta} + \bigg(X(s) + \dfrac{\beta (s)}{a(s)} \bigg)^{\delta} \bigg] ds > n \bigg \}
		\end{align}
		As in equation \eqref{51}, we have for, we have for $\big (\rho , C, \beta \big) \in \mathcal{L}^{-}_{T} (x, 0)$ 
		\begin{align*}
			\hat{F}_{\delta} \big (0, x, \hat{\omega}_{1}(0), \hat{\omega}^{z}_{2,\delta}(0) \big) \geqslant & \mathbb{E} \bigg [ \int_{0}^{T \wedge \xi^{\delta}_n } e^{- \int_{0}^{t} \big ( \theta (s) + \mu(s)\big) ds } \bigg [ V_{\delta} \big (C(t) \big) + \mu (t) V_{\delta} \bigg (X(t) + \dfrac{\beta (t)}{a(t)} \bigg ) \bigg ] dt \\
			& + e^{-\int_{0}^{T\wedge \xi^{\delta}_n } \big (\theta(t) + \mu(t)\big)dt } \hat{F}_{\delta} \bigg (T\wedge \xi^{\delta}_n , X \big (T\wedge \xi^{\delta}_n \big), \hat{\omega}_{1} \big (T\wedge \xi^{\delta}_n \big), \hat{\omega}^{ \hat{Z}^{\psi}_{T\wedge \xi^{\delta}_n}  }_{2,\delta} \big (T\wedge \xi^{\delta}_n \big)  \bigg) \bigg ] 
		\end{align*}
		If
		\begin{equation*}
			\mathbb{E} \bigg [\int_{0}^{T} e^{-\int_{0}^{t} \big (\theta(s) + \mu(s)\big)ds } \bigg [V_{\delta} \big (C(t) \big)  + \mu(t) V_{\delta} \bigg (X(t) + \dfrac{\beta (t)}{a(t)} \bigg) \bigg] dt  \bigg] = + \infty,
		\end{equation*}
		then equation \eqref{60} automatically holds.\par
		Next, we assume that 
		\begin{equation}\label{62}
			\mathbb{E} \bigg [ \int_{0}^{T} e^{-\int_{0}^{t} \big (\theta (s) + \mu(s) \big) ds } \bigg [ V_{\delta} \big (C(t) \big) + \mu(t) V_{\delta} \bigg (X(t) + \dfrac{\beta(t)}{a(t)} \bigg) \bigg ] dt \bigg ] > -\infty 
		\end{equation}
		For sufficiently small $\kappa > 0$, we set
		\begin{equation*}
			\overline{ \zeta} = \overline{\zeta}_{\kappa, \delta} := \dfrac{\delta}{ 1 + 2 \kappa}
		\end{equation*}
		Thus we consider the following relation:
		\begin{align}\nonumber\label{63}
			\hat{F}_{\overline{\zeta}} \big (0, x, \hat{\omega}_{1}(0), \hat{\omega}^{z}_{2, \overline{\zeta}}(0) \big) & \geqslant \mathbb{E} \bigg [ \int_{0}^{T \wedge \xi^{\overline{\zeta}}_n }e^{-\int_{0}^{t} \big (\theta (s) + \mu(s) \big) ds } \bigg [ V_{\overline{\zeta}} \big (C(t) \big) + \mu(t) V_{\overline{\zeta}} \bigg (X(t) + \dfrac{\beta(t)}{a(t)} \bigg) \bigg ] dt\\
			& + e^{-\int_{0}^{T\wedge \xi^{ \overline{\zeta}}_n } \big (\theta(t) + \mu(t)\big)dt } \hat{F}_{ \overline{\zeta}} \bigg (T\wedge \xi^{ \overline{\zeta}}_n , X \big (T\wedge \xi^{ \overline{\zeta}}_n \big), \hat{\omega}_{1} \big (T\wedge \xi^{ \overline{\zeta}}_n \big), \hat{\omega}^{ \hat{Z}^{\psi}_{T\wedge \xi^{ \overline{\zeta}}_n}  }_{2, \overline{\zeta}} \big (T\wedge \xi^{ \overline{\zeta}}_n \big)  \bigg) \bigg ] 
		\end{align}
		where $\hat{F}_{\overline{\zeta}} \big (0, x, \hat{\omega}_{1}(0), \hat{\omega}^{z}_{2, \overline{\zeta}}(0) \big)$ and $\xi^{ \overline{\zeta}}_n$ are defined in equations \eqref{38} and \eqref{61} respectively when $\delta$ is replaced by $\overline{\zeta}$. Denote $\xi^{ \overline{\zeta}}_n$ as $\xi_n$ in the following. Then we need to show that $\exists \mathcal{Q}_{0}$ independent of $n$ such that
		\begin{equation*}
			\mathbb{E} \bigg [ \bigg \{ \bigg ( X \big ( T\wedge \xi_{n} \big ) + \hat{\omega}_{1} \big ( T\wedge \xi_{n} \big ) \bigg )^{\overline{\zeta}} \bigg ( \hat{\omega}^{\hat{Z}^{\psi}_{T \wedge \xi_n}}_{2, \overline{\zeta}} \big ( T\wedge \xi_{n} \big ) \bigg )^{1 - \overline{\zeta}} \bigg \}^{(1 + \kappa)} \bigg ] \leqslant \mathcal{Q}_{0}.
		\end{equation*}
		For that, we use the following relation:
		\begin{eqnarray*}
		&& \mathbb{E} \bigg [ \bigg \{ \bigg ( X \big ( T\wedge \xi_{n} \big ) + \hat{\omega}_{1} \big ( T\wedge \xi_{n} \big ) \bigg )^{\overline{\zeta}} \bigg ( \hat{\omega}^{\hat{Z}^{\psi}_{T \wedge \xi_n}}_{2, \overline{\zeta}} \big ( T\wedge \xi_{n} \big ) \bigg )^{1 - \overline{\zeta}} \bigg \}^{(1 + \kappa)} \bigg ]\\
		&& = \mathbb{E} \bigg [  \bigg ( X \big ( T\wedge \xi_{n} \big ) + \hat{\omega}_{1} \big ( T\wedge \xi_{n} \big ) \bigg )^{\tilde{\delta}} \bigg ( \hat{\omega}^{\hat{Z}^{\psi}_{T \wedge \xi_n}}_{2, \overline{\zeta}} \big ( T\wedge \xi_{n} \big ) \bigg )^{\iota + 1}  \bigg ] .
		\end{eqnarray*}
		where $\tilde{\delta} = \tilde{\delta}_{\kappa , \delta}$ and $\iota = \iota_{\kappa, \delta}$ are defined by
		\begin{equation}\label{64}
			\tilde{\delta} : = \dfrac{1 + \kappa}{1 + 2 \kappa}\delta \quad \text{and} \quad \iota : = (1 - \overline{\zeta})(1 + \kappa) - 1
		\end{equation}
		respectively.\par
		Thus, It means to show that 
		\begin{equation}\label{65}
			\mathbb{E} \bigg [  \bigg ( X \big ( T\wedge \xi_{n} \big ) + \hat{\omega}_{1} \big ( T\wedge \xi_{n} \big ) \bigg )^{\tilde{\delta}} \bigg ( \hat{\omega}^{\hat{Z}^{\psi}_{T \wedge \xi_n}}_{2, \overline{\zeta}} \big ( T\wedge \xi_{n} \big ) \bigg )^{\iota + 1}  \bigg ] \leqslant \mathcal{Q}_{0} .
		\end{equation}
		See \cite{hata2020optimal} for the proof of equation \eqref{65}.\par
		From equation \eqref{65}, we have the uniformly integrability of the term
		\begin{equation*}
			e^{-\int_{0}^{T\wedge \xi^{ \overline{\zeta}}_n } \big (\theta(t) + \mu(t)\big)dt } \hat{F}_{ \overline{\zeta}} \bigg (T\wedge \xi^{ \overline{\zeta}}_n , X \big (T\wedge \xi^{ \overline{\zeta}}_n \big), \hat{\omega}_{1} \big (T\wedge \xi^{ \overline{\zeta}}_n \big), \hat{\omega}^{ \hat{Z}^{\psi}_{T\wedge \xi^{ \overline{\zeta}}_n}  }_{2, \overline{\zeta}} \big (T\wedge \xi^{ \overline{\zeta}}_n \big)  \bigg)
		\end{equation*}
		in equation \eqref{63}. Hence, as $n \to + \infty$ in equation \eqref{63}, since 
		\begin{equation}\label{66}
			V_{\overline{\zeta}} (x) > V_{\delta} (x) + \dfrac{2 \kappa}{\delta},
		\end{equation}
		we have
		\begin{align}\nonumber\label{67}
			\hat{F}_{\overline{\zeta}} \big (0, x, \hat{\omega}_{1}(0), \hat{\omega}^{z}_{2, \overline{\zeta}}(0) \big) & \geqslant \mathcal{V} \big (0, x, z;\rho, C, \beta; \overline{\zeta} \big)\\\nonumber
			&  \geqslant \mathcal{V} \big (0, x, z;\rho, C, \beta; \delta \big) \\
			& + \dfrac{2\kappa}{\delta} \bigg [ \int_{0}^{T} e^{-\int_{0}^{t} \big (\theta(s) + \mu(s)\big) ds} \big (1 + \mu(t) \big) dt + e^{-\int_{0}^{T} \big (\theta(t) + \mu(t) \big) dt} \bigg ].
		\end{align}
		where $\mathcal{V} \big (0, x, z;\rho, C, \beta; \delta \big)$ is written for the dependence of $\mathcal{V} \big (0, x, z;\rho, C, \beta \big)$ on $\delta$. Indeed for $x_{1} > 0$ fixed, we can show that $g (\lambda) = \dfrac{1}{\lambda} \big (x^{\lambda}_{1} - 1 \big), \; \lambda < 0$ is increasing in $\lambda$ and equation \eqref{66} follows from it. \par
		Similarly to Lemma 3.1 of \cite{nagai2011asymptotics}, the solution $M(t)$ of \eqref{30} is in $C^{1}$class with respect to $\delta$ and also are solution $n(t)$ and $q(t)$ of \eqref{31} and \eqref{32} respectively. Thus, since $\hat{F}_{\delta}$ is continuous in $\delta$, by taking $\kappa \to + \infty$ in equation \eqref{67}, we get equation \eqref{60}.\par
		Let's now prove that $\big (\hat{\rho}, \hat{C}, \hat{\beta} \big)$ gives an optimal strategy for \eqref{6}. Fist, we show that $\big (\hat{\rho}, \hat{C}, \hat{\beta} \big)$ gives an admissible strategy. Set
		\begin{equation*}
			\hat{h}(t) = \dfrac{\hat{\rho} (t)}{\hat{X} (t) + \hat{\omega}_{1}(t)}.
		\end{equation*}
		From equation \eqref{100}, we have $\hat{X}(t) + \hat{\omega}_{1}(t) > 0.$ We can see that there exists $\epsilon_{1} > 0$ such that $\big |\hat{\rho}(t)\big| \leqslant \epsilon_{1} \big (1 + \big |\hat{Z}^{\psi}_{t} \big| \big) \big (\hat{X} (t) + \hat{\omega}_{1} (t) \big)$. \par
		Next, from equations \eqref{33} and \eqref{35}, we have 
		\begin{equation*}
			\hat{\omega}^{z}_{2}(t) = e^{ \frac{1}{1-\delta} \int_{t}^{T} \big [- \big (\theta(s) + \mu(s)\big) + \delta a(s) \big] ds + \frac{1}{1-\delta} \Psi_{\delta} (t; T, z)} + \int_{t}^{T}  e^{ \frac{1}{1-\delta} \int_{t}^{T} \big [- \big (\theta(u) + \mu(u)\big) + \delta a(u) \big] du }\cdot e^{\frac{1}{1-\delta} \Psi_{\delta}(t; s, z) } H(s) ds , 
		\end{equation*}
		where
		\begin{align*}
			\Psi_{\delta} \big (t; s, z \big) & = \big (1 - \delta \big) \mathbb{E}_{P_{r}^{2}} \bigg [ \exp \bigg [ \int_{t}^{s} \dfrac{1}{(1 - \delta)} \bigg ( \dfrac{\delta}{2 (1 - \delta) } \big (U\hat{Z}^{\psi}_{v} \\
			& + u - r_{v}\mathbf{1} \big)^{T} \big (\alpha \mathbb{C}_{B}\alpha^{T} + \Sigma \mathbb{C}_{W} \Sigma^T \big)^{-1}\big (U\hat{Z}^{\psi}_{v} + u - r_{v}\mathbf{1} \big)  + \delta r_{v} \bigg ) dv \bigg ] \bigg | \hat{Z}^{\psi}_{t} = z \bigg ].
		\end{align*}
		We notice that for $\delta < 0, \; \Psi_{\delta} (t; s, z)$ is monotone decreasing in $s \in [t, T]$. Hence
		\begin{equation*}
			\Psi_{\delta} \big (t; s,z \big) = \dfrac{1}{2} z^{T} M(t) z + n^{T}(t)z + q(t).
		\end{equation*}
		For $\kappa_{0} , \epsilon > 0$ sufficiently small, $\exists \upsilon_{2} > 0$ such that for $t \leqslant T - \kappa_{0}$, we have 
		\begin{align*}
			\bigg (\hat{\omega}^{\hat{Z}^{\psi}_{t}}_{2}(t) \bigg)^{-1} & \leqslant \upsilon_{2} \bigg [\int_{t}^{t + \kappa_0} e^{\frac{1}{1-\delta} \Psi_{\delta} \big (t; s, \hat{Z}^{\psi}_{t}\big)} ds \bigg]^{-1} \leqslant \upsilon_{2} e^{- \frac{1}{1-\delta} \Psi_{\delta}\big (t; t + \kappa_0 , \hat{Z}^{\psi}_t \big) }\\
			& \leqslant \upsilon_{2} e^{\epsilon \big (1 +  |\hat{Z}^{\psi}_{t} |^2 \big)}
		\end{align*}
		If $\kappa_{0}$ is sufficiently small and for $T - \kappa_{0} \leqslant t \leqslant T, \; \exists \upsilon_{3} > 0$ such that
		\begin{equation*}
			\bigg (\hat{\omega}^{\hat{Z}^{\psi}_{t}}_{2}(t) \bigg)^{-1}  \leqslant \upsilon_{3} e^{- \frac{1}{1-\delta} \Psi_{\delta}\big (t; t + \kappa_0 , \hat{Z}^{\psi}_t \big) } \leqslant  \upsilon_{3} e^{\epsilon \big (1 +  |\hat{Z}^{\psi}_{t} |^2 \big)}
		\end{equation*}
		Therefore $\big (\hat{\rho}, \hat{C}, \hat{\beta} \big) \in \mathcal{L}^{-}_{T}(x, 0)$.\par
		Now, setting
		\begin{equation*}
			\hat{\xi}_{n}: = \inf \bigg \{ t > 0;  |\hat{Z}^{\psi}_{t} | > n \bigg \} \wedge \inf \bigg \{ t > 0; \hat{X}(t) + \hat{\omega}_{1}(t) < \dfrac{1}{n} \bigg \} \wedge \inf \bigg \{ t > 0; \hat{C}(t) < \dfrac{1}{n}, \hat{X}(t) + \dfrac{\hat{\beta}(t)}{a(t)} < \dfrac{1}{n} \bigg \},
		\end{equation*}
		we have
		\begin{align*}
			\hat{F}_{\delta} \big (0, x, \hat{\omega}_{1}(0), \hat{\omega}^{z}_{2,\delta}(0) \big) = & \mathbb{E} \bigg [ \int_{0}^{T \wedge \hat{\xi}_n } e^{- \int_{0}^{t} \big ( \theta (s) + \mu(s)\big) ds } \bigg [ V_{\delta} \big (\hat{C}(t) \big) + \mu (t) V_{\delta} \bigg (\hat{X}(t) + \dfrac{\hat{\beta} (t)}{a(t)} \bigg ) \bigg ] dt \\
			& + e^{-\int_{0}^{T\wedge \hat{\xi}_n } \big (\theta(t) + \mu(t)\big)dt } \hat{F}_{\delta} \bigg (T\wedge \hat{\xi}_n , \hat{X} \big (T\wedge \hat{\xi}_n \big), \hat{\omega}_{1} \big (T\wedge \hat{\xi}_n \big), \hat{\omega}^{ \hat{Z}^{\psi}_{T\wedge \hat{\xi}_n}  }_{2,\delta} \big (T\wedge \hat{\xi}_n \big)  \bigg) \bigg ]. 
		\end{align*}
		Using Fatou's lemma, we obtain
		\begin{equation*}
			\mathcal{V} \big (0,  x, z; \hat{\rho} , \hat{C}, \hat{\beta}\big) \geqslant \hat{F}  \big (0, x, \hat{\omega}_{1}(0), \hat{\omega}^{z}_{2}(0) \big) .
		\end{equation*}
		$\square$
	\end{proof}
	\section{Conclusion} 
	The present study investigated an optimal investment-consumption-life insurance problem for a wage earner with partial information, where the Kalman filter was nonlinear and the prices of risky assets were correlated to the factor process. The combination of the Hamilton-Jacobi-Bellman (HJB) equation and two backward stochastic differential equations (BSDE) was derived using the dynamic programming principle. Additionally, the nonlinear filter was derived via the Zakai equation, and the verification theorem was demonstrated, leading to the construction of an optimal strategy. Future research can consider further correlation by making the assumption that the process $\big (\alpha(t) \big){t\in [0, \xi \wedge T]}$ in equations \eqref{4} and \eqref{5} depended on the factor process; that is, $\big (\alpha(Z_t,t) \big)_{t\in [0, \xi \wedge T]}$.
\section{Appendix }\label{A} The Proof of proposition \ref{prop1}.\\
\begin{proof}
	Equation \eqref{20} can be written as:
	\begin{equation*}
		-\dfrac{\partial F}{\partial t} + \sup_{(\rho, C, \beta) \in \mathbb{R}^k \times \mathbb{R}^{+} \times \mathbb{R}} \mathcal{W}\big (t, x, z; \rho, C, \beta \big) = 0,
	\end{equation*} 
	where
	\begin{equation*}
		\mathcal{W}\big (t, x, z; \rho, C, \beta \big) = \mathcal{D}^{\rho, C, \beta} F\big (t, x, z; \omega_{1}(t), \omega_{2}(t) \big) + V(C) + \mu (t) V \bigg (x + \dfrac{\beta}{a_t} \bigg).
	\end{equation*}
	Next,
	\begin{align*}
		& - F_{t}  + \mathcal{W}\big (t, x, z; \rho, C, \beta \big)\\  & = \big (\theta(t) + \mu(t) \big)F - F_{t} + F_{x} \bigg [x r_{t} + \rho^{T} \big (U z + u - r_{1} \mathds{1} \big) - C - \beta + R_t \bigg] \\
		& - F_{\omega_1} f_1 - F_{\omega_2}f_{2} + \dfrac{1}{2}F_{xx}\rho^{T} \big (\alpha \mathbb{C}_{B}\alpha^T + \Sigma \mathbb{C}_{W}\Sigma^T \big)\rho + F_{x\omega_1}\rho^T \big (\alpha \mathbb{C}_{B}\alpha^T + \Sigma \mathbb{C}_{W}\Sigma^T \big) \lambda_{1} \\
		& + F_{x \omega_2} \rho^T \big (\alpha \mathbb{C}_{B}\alpha^T + \Sigma \mathbb{C}_{W}\Sigma^T \big) \lambda_{2} + \dfrac{1}{2} F_{\omega_1 \omega_1}\lambda^{T}_{1} \big (\alpha \mathbb{C}_{B}\alpha^T + \Sigma \mathbb{C}_{W}\Sigma^T \big)\lambda_1  + F_{\omega_1 \omega_2}\lambda^{T}_{1} \big (\alpha \mathbb{C}_{B}\alpha^T \\
		& + \Sigma \mathbb{C}_{W}\Sigma^T \big) \lambda_{2}  + \dfrac{1}{2}F_{\omega_2 \omega_2}\lambda^{T}_{2}\big (\alpha \mathbb{C}_{B}\alpha^T + \Sigma \mathbb{C}_{W}\Sigma^T \big)\lambda_{2} + \dfrac{C^{\delta}}{\delta} + \mu(t) \dfrac{\big (x + \frac{\beta}{a(t)} \big)^{\delta} }{\delta} = 0.\\
	\end{align*}
	Then the supremum in \eqref{20} can be derived as follows:
				\begin{eqnarray*}
				&& \dfrac{\partial \mathcal{W}}{\partial \rho} \\
				&& = F_{x} \big (Uz + u - r_{t} \mathbf{1} \big) + \big (\alpha \mathbb{C}_{B}\alpha^T + \Sigma \mathbb{C}_{W}\Sigma^T \big) \rho F_{xx} + \big (\alpha \mathbb{C}_{B}\alpha^T + \Sigma \mathbb{C}_{W}\Sigma^T \big)\lambda_{1} F_{x \omega_1}\\
				&& + \big (\alpha \mathbb{C}_{B}\alpha^T + \Sigma \mathbb{C}_{W}\Sigma^T \big) \lambda_{2} F_{x \omega_2}\\
				&& \dfrac{\partial \mathcal{W}}{\partial \rho}  = 0 \Leftrightarrow \hat{\rho} = - \big (\alpha \mathbb{C}_{B}\alpha^T + \Sigma \mathbb{C}_{W}\Sigma^T \big)^{-1} \big (U z + u - r_{t} \mathbf{1} \big) \dfrac{F_x}{F_{xx}} - \lambda_{1} \dfrac{F_{x \omega_1}}{F_{xx}} - \lambda_{2} \dfrac{F_{x \omega_2}}{F_{xx}}\\
				&& \dfrac{\partial \mathcal{W}}{\partial C}  = F_{x} - C^{\delta - 1}\\
				&& \dfrac{\partial \mathcal{W}}{\partial C}  = 0 \Leftrightarrow C^{\delta - 1} = F_{x} \Leftrightarrow \hat{C} = (F_x)^{-\frac{1}{1 - \delta}}\\
				&& \dfrac{\partial \mathcal{W}}{\partial \beta}  =-  F_{x} + \dfrac{\mu(t)}{a(t)} \bigg (x + \dfrac{\beta}{a(t)} \bigg)^{\delta -1} \Leftrightarrow  \hat{\beta} = a(t) \bigg [\bigg (\dfrac{a(t)}{\mu(t)} F_{x} \bigg)^{- \frac{1}{1 - \delta}} - x\bigg].
			\end{eqnarray*}
	Hence equation \eqref{20} becomes:
				\begin{eqnarray*}
				&& \big (\theta(t) + \mu(t)\big)F- F_{t} + F_{x} \bigg \{x r_t + \bigg [ - \big (\alpha \mathbb{C}_{B}\alpha^T + \Sigma \mathbb{C}_{W}\Sigma^T \big)^{-1}\big (Uz + u - r_{t}\mathbf{1} \big)\dfrac{F_x}{F_{xx}} - \lambda_{1} \dfrac{F_{x \omega_1}}{F_{xx}}\\
				&& - \lambda_{2} \dfrac{F_{x  \omega_2}}{F_{xx}} \bigg ]^{T} \big ( Uz +  u - r_{t} \mathbf{1} \big )+ R(t) - (F_{x})^{-\frac{1}{1-\delta}} - a(t) \bigg [\bigg (\dfrac{a(t)}{\mu(t)}F_x \bigg)^{- \frac{1}{1 - \delta}} - x\bigg]\bigg \} - F_{\omega_{1}}f_{1} - F_{\omega_2} f_{2}\\
				&&  + \dfrac{1}{2} F_{xx} \bigg \{ - \big (\alpha \mathbb{C}_{B}\alpha^T + \Sigma \mathbb{C}_{W}\Sigma^T \big)^{-1} \big (Uz + u - r_{t}\mathbf{1} \big)\dfrac{F_x}{F_{xx}} - \lambda_{1} \dfrac{F_{x \omega_1}}{F_{xx}} -\lambda_{2} \dfrac{F_{x \omega_2}}{F_{xx}}\bigg \}^{T} \big (\alpha \mathbb{C}_{B}\alpha^T \\
				&& + \Sigma \mathbb{C}_{W}\Sigma^T \big)\bigg \{ - \big (\alpha \mathbb{C}_{B}\alpha^T  + \Sigma \mathbb{C}_{W}\Sigma^T \big)^{-1} \big (Uz + u - r_{t} \mathbf{1} \big) \dfrac{F_{x}}{F_{xx}}- \lambda_{1}\dfrac{F_{x \omega_1}}{F_{xx}} - \lambda_{2}\dfrac{F_{x \omega_2}}{F_{xx}} \bigg \} - \\
				&& \bigg \{ \big (\alpha \mathbb{C}_{B}\alpha^T + \Sigma \mathbb{C}_{W}\Sigma^T \big)^{-1}\big (Uz + u - r_{t} \mathbf{1} \big)\dfrac{F_{x}}{F_{xx}} +\lambda_{1}\dfrac{F_{x \omega_1}}{F_{xx}} + \lambda_{2} \dfrac{F_{x \omega_2}}{F_{xx}} \bigg \}^{T} \big (\alpha \mathbb{C}_{B}\alpha^T + \Sigma \mathbb{C}_{W}\Sigma^T \big)\lambda_{1}F_{x \omega_{1}}\\
				&& + \bigg \{ - \big (\alpha \mathbb{C}_{B}\alpha^T + \Sigma \mathbb{C}_{W}\Sigma^T \big)^{-1}\big (Uz + u - r_{t} \mathbf{1} \big)\dfrac{F_x}{F_{xx}} - \lambda_{1} \dfrac{F_{x \omega_1}}{F_{xx}}- \lambda_{2} \dfrac{ F_{x \omega_2}}{F_{xx}} \bigg \}^{T}\big (\alpha \mathbb{C}_{B}\alpha^T \\
				&& + \Sigma \mathbb{C}_{W}\Sigma^T \big) \lambda_{2} F_{x \omega_2}+ \dfrac{1}{2} \lambda^{T}_{1} \big (\alpha \mathbb{C}_{B}\alpha^T + \Sigma \mathbb{C}_{W}\Sigma^T \big)\lambda_{1} F_{\omega_{1}\omega_{1}} + \lambda^{T}_{1} \big (\alpha \mathbb{C}_{B}\alpha^T + \Sigma \mathbb{C}_{W}\Sigma^T \big)\lambda_{2}F_{\omega_1 \omega_2} \\
				&& + \dfrac{1}{2} \lambda^{T}_{2} \big (\alpha \mathbb{C}_{B}\alpha^T + \Sigma \mathbb{C}_{W}\Sigma^T \big) \lambda_2 F_{\omega_2 \omega_2}+ \dfrac{1}{\delta} (F_x)^{-\frac{\delta}{1-\delta}} + \dfrac{\mu(t)}{\delta} \bigg [x + \bigg (\frac{a(t)}{\mu(t)} F_x \bigg)^{-\frac{1}{1-\delta}} - x\bigg]^{\delta} = 0.\\
				&& \Rightarrow   \big (\theta(t) + \mu(t)\big)F - F_{t} + F_{x} \bigg \{ x r_{t} - \bigg [  \big (Uz + u - r_t \mathbf{1} \big)^{T}\big (\alpha \mathbb{C}_{B}\alpha^T + \Sigma \mathbb{C}_{W}\Sigma^T \big)^{-1} \dfrac{F_x}{F_{xx}}\\
				&& + \dfrac{F_{x \omega_1}}{F_{xx}}\lambda_{1}^{T}  +  \dfrac{F_{x \omega_2} }{F_{xx}}\lambda_{2}^{T} \bigg ] \big ( Uz + u - r_{t}\mathbf{1} \big )+ R(t) - (F_{x})^{- \frac{1}{1-\delta}} - a(t) \bigg [\bigg (\dfrac{a(t)}{\mu(t)} F_x \bigg)^{-\frac{1}{1-\delta}} - x\bigg] \bigg \} \\
				&& - F_{\omega_1}f_1 - F_{\omega_2}f_{2} + \dfrac{1}{2}F_{xx} \bigg \{  \dfrac{F_x}{F_{xx}} \big (Uz + u - r_{t} \mathbf{1} \big)^{T}\big (\alpha \mathbb{C}_{B}\alpha^T + \Sigma \mathbb{C}_{W}\Sigma^T \big)^{-1}\\
				&& + \dfrac{F_{x\omega_1}}{F_{xx}}\lambda_{1}^{T}  + \dfrac{F_{x\omega_2}}{F_{xx}}\lambda_{2}^{T} \bigg \}\big (\alpha \mathbb{C}_{B}\alpha^T + \Sigma \mathbb{C}_{W}\Sigma^T \big)\bigg \{ \big (\alpha \mathbb{C}_{B}\alpha^T + \Sigma \mathbb{C}_{W}\Sigma^T \big)^{-1}\big (Uz + u - r_{t} \mathbf{1} \big) \dfrac{F_x}{F_{xx}} \\
				&& + \lambda_{1} \dfrac{F_{x \omega_1} }{F_{xx}} + \lambda_{2} \dfrac{F_{x\omega_2}}{F_{xx}} \bigg \} - F_{x\omega_1}\bigg \{\dfrac{F_x}{F_{xx}} \big (Uz + u - r_{t} \mathbf{1} \big)^{T}\big (\alpha \mathbb{C}_{B}\alpha^T + \Sigma \mathbb{C}_{W}\Sigma^T \big)^{-1}  + \dfrac{F_{x\omega_1}}{F_{xx}}\lambda_{1}^{T} \\
				&& + \dfrac{F_{x\omega_2}}{F_{xx}}\lambda_{2}^{T} \bigg \}\big (\alpha \mathbb{C}_{B}\alpha^T+ \Sigma \mathbb{C}_{W}\Sigma^T \big)\lambda_{1}  - F_{x \omega_2}\bigg \{  \dfrac{F_x}{F_{xx}} \big (Uz +u  - r_{t} \mathbf{1}\big )^{T} \big ( \alpha \mathbb{C}_{B}\alpha^{T} + \Sigma \mathbb{C}_{W}\Sigma^T \big)^{-1}  \\
				&& +  \dfrac{F_{x \omega_1}}{F_{xx}}\lambda_{1}^{T} +  \dfrac{F_{x\omega_2}}{F_{xx}}\lambda_{2}^{T} \bigg \} \big (\alpha \mathbb{C}_{B}\alpha^T + \Sigma \mathbb{C}_{W}\Sigma^T \big)\lambda_{2} + \dfrac{1}{2} F_{\omega_1 \omega_1}\lambda^{T}_{1}\big (\alpha \mathbb{C}_{B}\alpha^T + \Sigma \mathbb{C}_{W}\Sigma^T \big)\lambda_{1}
			\end{eqnarray*}
	\begin{eqnarray*}
	&&  + F_{\omega_1 \omega_2} \lambda^{T}_{1} \big (\alpha \mathbb{C}_{B}\alpha^T + \Sigma \mathbb{C}_{W}\Sigma^T \big)\lambda_{2} + \dfrac{1}{2} F_{\omega_2 \omega_2} \lambda^{T}_{2} \big (\alpha \mathbb{C}_{B}\alpha^T + \Sigma \mathbb{C}_{W}\Sigma^T \big) \lambda_{2} + \dfrac{1}{\delta} \big (F_x \big)^{-\frac{\delta}{1-\delta}} \\
	&& + \dfrac{\mu(t)}{\delta} \bigg (\frac{a(t)}{\mu(t)}F_x \bigg)^{-\frac{\delta}{1-\delta}} = 0.\\
	&& \Rightarrow \big (\theta(t) + \mu(t)\big)F - F_{t} + F_{x} \bigg [ x \big ( r_{t} + a(t) \big ) + R(t) \bigg ] - F_{x} \bigg [\dfrac{F_x}{F_{xx}} \big ( Uz + u - r_{t}\mathbf{1}\big )^{T} \big (\alpha \mathbb{C}_{B}\alpha^T \\
	&& + \Sigma \mathbb{C}_{W}\Sigma^T \big)^{-1} + \dfrac{F_{x\omega_1}}{F_{xx}}\lambda_{1}^{T} +  \dfrac{F_{x\omega_2}}{F_{xx}}\lambda_{2}^{T} \bigg ]\big ( Uz + u  - r_{t}\mathbf{1} \big )- F_{\omega_1} f_{1} - F_{\omega_2}f_{2} \\
	&& + \dfrac{1}{2} F_{xx} \bigg \{  \dfrac{F_x}{F_{xx}} \big ( Uz + u - r_{t} \mathbf{1}\big )^{T}\big (\alpha \mathbb{C}_{B}\alpha^T + \Sigma \mathbb{C}_{W}\Sigma^T \big)^{-1} + \dfrac{F_{x\omega_1}}{F_{xx}} \lambda_{1}^{T} +  \dfrac{F_{x\omega_2}}{F_{xx}}\lambda_{2}^{T} \bigg \}\big (\alpha \mathbb{C}_{B}\alpha^T\\
	&& + \Sigma \mathbb{C}_{W}\Sigma^T \big)\bigg \{ \big (\alpha \mathbb{C}_{B}\alpha^T + \Sigma \mathbb{C}_{W}\Sigma^T \big)^{-1}\big (Uz + u - r_{t} \mathbf{1} \big ) \dfrac{F_x}{F_{xx}} + \lambda_{1} \dfrac{F_{x \omega_1}}{F_{xx}} + \lambda_{2} \dfrac{F_{x\omega_2}}{F_{xx}} \bigg \} \\
	&& -F_{x\omega_1} \bigg \{\big ( Uz + u - r_{t}\mathbf{1}\big )^{T} \big (\alpha \mathbb{C}_{B}\alpha^T + \Sigma \mathbb{C}_{W}\Sigma^T \big)^{-1} \dfrac{F_x}{F_{xx}} +  \dfrac{F_{x\omega_1} }{F_{xx}}\lambda_{1}^T +  \dfrac{F_{x \omega_2}}{F_{xx}} \lambda_{2}^T \bigg \}\big (\alpha \mathbb{C}_{B}\alpha^T \\
	&& + \Sigma \mathbb{C}_{W}\Sigma^T \big)\lambda_{1} - F_{x \omega_2}\bigg \{ \big (Uz + u - r_{t}\mathbf{1}\big )^{T}  \big (\alpha \mathbb{C}_{B}\alpha^T+ \Sigma \mathbb{C}_{W}\Sigma^T \big)^{-1} \dfrac{F_x}{F_{xx}} + \dfrac{F_{x\omega_1}}{F_{xx}}\lambda_{1}^{T} \\
	&& + \dfrac{F_{x \omega_2}}{F_{xx}}\lambda_{2}^T \bigg \}^{T}\big (\alpha \mathbb{C}_{B}\alpha^T + \Sigma \mathbb{C}_{W}\Sigma^T \big)\lambda_{2} + \dfrac{1}{2} F_{\omega_1 \omega_1} \lambda^{T}_{1} \big (\alpha \mathbb{C}_{B}\alpha^T + \Sigma \mathbb{C}_{W}\Sigma^T \big) \lambda_{1} +\\
	&& F_{\omega_1 \omega_2}\lambda^{T}_{1} \big (\alpha \mathbb{C}_{B}\alpha^T + \Sigma \mathbb{C}_{W}\Sigma^T \big)\lambda_{2} + \dfrac{1}{2}F_{\omega_2 \omega_2} \lambda^{T}_{2} \big (\alpha \mathbb{C}_{B}\alpha^T + \Sigma \mathbb{C}_{W}\Sigma^T \big) \lambda_{2}\\
	&& + \bigg (\dfrac{1-\delta}{\delta} \bigg) \bigg [1 + \dfrac{\big (a(t) \big)^{-\frac{\delta}{1-\delta}} }{\big (\mu(t) \big)^{-\frac{1}{1-\delta}} } \bigg]\big (F_x \big)^{-\frac{\delta}{1-\delta}} = 0.\\
	&&  \Rightarrow -F_{t} + \big (\theta(t) + \mu (t) \big ) F + F_{x} \big [\big (r_{t} + a(t) \big)x + R(t) \big ] - F_{\omega_{1}} f_{1} - F_{\omega_2} f_{2} + \bigg \{ \dfrac{F_{x}}{F_{xx}} \big (\alpha \mathbb{C}_{B} \alpha^{T} \\
	&& + \Sigma \mathbb{C}_{W}\Sigma^{T} \big)^{-1} \big ( Uz + u - r_{t} \mathbf{1}\big) + \dfrac{F_{x\omega_1}}{F_{xx}} \lambda_{1} + \dfrac{F_{x \omega_2}}{F_{xx}}\lambda_{2} \bigg \}^{T} \bigg [- F_{x} \big (Uz + u - r_{t}\mathbf{1} \big) + \\
	&&  \dfrac{1}{2} F_{xx} \big (\alpha \mathbb{C}_{B}\alpha^{T} + \Sigma \mathbb{C}_{W} \Sigma^{T} \big) \bigg \{\big (\alpha \mathbb{C}_{B}\alpha^{T} + \Sigma \mathbb{C}_{W} \Sigma^{T} \big)^{-1} \big (Uz + u - r_{t} \mathbf{1} \big) \dfrac{F_x}{F_{xx}} + \lambda_{1} \dfrac{F_{x\omega_1}}{F_{xx}} \\
	&& + \lambda_{2} \dfrac{F_{x \omega_2}}{F_{xx}}\bigg \}  - F_{x\omega_1}  \big (\alpha \mathbb{C}_{B} \alpha^{T} + \Sigma \mathbb{C}_{W}\Sigma^{T} \big)\lambda_1 - F_{x\omega_2}  \big (\alpha \mathbb{C}_{B} \alpha^{T} + \Sigma \mathbb{C}_{W}\Sigma^{T} \big)\lambda_2\bigg]\\
	&& + \dfrac{1}{2} F_{\omega_1 \omega_1} \lambda_{1}^{T}  \big (\alpha \mathbb{C}_{B} \alpha^{T} + \Sigma \mathbb{C}_{W}\Sigma^{T} \big) \lambda_{1} + F_{\omega_1 \omega_2} \lambda_{1}^{T}  \big (\alpha \mathbb{C}_{B}\alpha^{T} + \Sigma \mathbb{C}_{W} \Sigma^{T} \big) \lambda_{2} \\
	&& + \dfrac{1}{2} F_{\omega_2 \omega_2} \lambda_{2}^{T}  \big (\alpha \mathbb{C}_{B}\alpha^{T} + \Sigma \mathbb{C}_{W} \Sigma^{T} \big) \lambda_{2} + \bigg ( \dfrac{1 - \delta}{\delta} \bigg ) \bigg [1 + \dfrac{ \big (a(t) \big)^{-\frac{\delta}{1-\delta} } }{\big (\mu(t)\big)^{-\frac{1}{1-\delta}} } \bigg] \big (F_{x} \big)^{-\frac{\delta}{1-\delta}} = 0.\\
	&& \Rightarrow- F_{t} + \big (\theta (t) + \mu (t)\big ) F + F_{x} \big [\big (r_{t} + a_{t}\big) x + R(t) \big ] - F_{\omega_1}f_{1}  - F_{\omega_2} f_{2} \\
	&& - \dfrac{1}{2} F_{xx} \bigg \{ \dfrac{F_x}{F_{xx}}  \big (\alpha \mathbb{C}_{B}\alpha^{T} + \Sigma \mathbb{C}_{W} \Sigma^{T} \big)^{-1} \big (Uz + u -r_{t} \mathbf{1} \big) + \dfrac{F_{x \omega_1} }{F_{xx}} \lambda_{1} +  \dfrac{F_{x \omega_2} }{F_{xx}} \lambda_{2} \bigg \}^{T}   \big (\alpha \mathbb{C}_{B}\alpha^{T} \\
	&& + \Sigma \mathbb{C}_{W} \Sigma^{T} \big) \bigg \{ \dfrac{2 F_{x}}{F_{xx}} \big (\alpha \mathbb{C}_{B}\alpha^{T} + \Sigma \mathbb{C}_{W} \Sigma^{T} \big)^{-1} \big (Uz + u - r_{t} \mathbf{1} \big ) - \bigg \{ \big (\alpha \mathbb{C}_{B}\alpha^{T} + 
		 \Sigma \mathbb{C}_{W} \Sigma^{T} \big)^{-1} \big (Uz \\
		 && + u - r_{t} \mathbf{1}\big) \dfrac{F_x}{F_{xx}} + \lambda_{1} \dfrac{F_{x \omega_1}}{F_{xx}} + \lambda_{2} \dfrac{F_{x \omega_2}}{F_{xx}} \bigg \} + 2 \dfrac{F_{x\omega_1}}{F_{xx}} \lambda_{1} + 2 \dfrac{F_{x\omega_2}}{F_{xx}} \lambda_{2} \bigg \} + \dfrac{1}{2} F_{\omega_1 \omega_1} \lambda_{1}^{T}\big (\alpha \mathbb{C}_{B}\alpha^{T} \\
		 && + \Sigma \mathbb{C}_{W} \Sigma^{T} \big) \lambda_{1}  + F_{\omega_{1} \omega_{2}} \lambda_{1}^{T} \big (\alpha \mathbb{C}_{B}\alpha^{T} + \Sigma \mathbb{C}_{W} \Sigma^{T} \big)\lambda_{2} + \dfrac{1}{2} F_{\omega_2 \omega_2} \lambda_{2}^{T} \big (\alpha \mathbb{C}_{B}\alpha^{T} + \Sigma \mathbb{C}_{W} \Sigma^{T} \big)\lambda_{2} \\
		 && +  \bigg ( \dfrac{1 - \delta}{\delta} \bigg ) \bigg [1 + \dfrac{ \big (a(t) \big)^{-\frac{\delta}{1-\delta} } }{\big (\mu(t)\big)^{-\frac{1}{1-\delta}} } \bigg] \big (F_{x} \big)^{-\frac{\delta}{1-\delta}} = 0.
	\end{eqnarray*}
	\begin{eqnarray}\nonumber\label{22} 
&&  \Rightarrow \big (\theta(t) + \mu(t)\big)F - F_{t} + F_{x} \bigg [x \big (r_t + a(t) \big) + R(t)\bigg] - F_{\omega_1} f_1 - F_{\omega_2} f_2 -  \dfrac{1}{2} F_{xx} \bigg \{ \big (\alpha \mathbb{C}_{B}\alpha^T \\ \nonumber
&& + \Sigma \mathbb{C}_{W}\Sigma^T \big)^{-1}\big ( Uz + u - r_{t}\mathbf{1} \big ) \dfrac{F_x}{F_{xx}} +  \dfrac{F_{x \omega_1}}{F_{xx}}\lambda_{1} +  \dfrac{F_{x \omega_2}}{F_{xx}}\lambda_{2} \bigg \}^{T} \big (\alpha \mathbb{C}_{B}\alpha^{T} + \Sigma \mathbb{C}_{W} \Sigma^{T} \big) \bigg \{ \dfrac{F_x}{F_{xx}}\big (\alpha \mathbb{C}_{B}\alpha^T  \\\nonumber
&& + \Sigma \mathbb{C}_{W}\Sigma^T \big)^{-1} \big (Uz + u - r_{t} \mathbf{1} \big ) +  \dfrac{F_{x \omega_1}}{F_{xx}}\lambda_{1} +  \dfrac{F_{x \omega_2}}{F_{xx}}\lambda_{2} \bigg \}   +\dfrac{1}{2} F_{\omega_1 \omega_1} \lambda_{1}^{T}\big (\alpha \mathbb{C}_{B}\alpha^{T} + \Sigma \mathbb{C}_{W} \Sigma^{T} \big) \lambda_{1}  \\\nonumber
&&  +  F_{\omega_{1} \omega_{2}} \lambda_{1}^{T} \big (\alpha \mathbb{C}_{B}\alpha^{T} + \Sigma \mathbb{C}_{W} \Sigma^{T} \big)\lambda_{2} + \dfrac{1}{2} F_{\omega_2 \omega_2} \lambda_{2}^{T} \big (\alpha \mathbb{C}_{B}\alpha^{T} + \Sigma \mathbb{C}_{W} \Sigma^{T} \big)\lambda_{2}\\
&& + \bigg (\dfrac{1-\delta}{\delta} \bigg) \bigg [1 + \dfrac{\big (a(t) \big)^{-\frac{\delta}{1-\delta}} }{\big (\mu(t) \big)^{-\frac{1}{1-\delta}} } \bigg]\big (F_x \big)^{-\frac{\delta}{1-\delta}} = 0.
	\end{eqnarray}
	Let us assume that $F$ is of the form
	\begin{eqnarray}\label{23}
		F\big (t, x, \omega_{1}(t), \omega_{2}(t) \big) = \dfrac{1}{\delta} \big (x + \omega_{1}(t) \big)^{\delta} \big (\omega_{2}(t) \big)^{1-\delta}
	\end{eqnarray}
	Thus, we substitute equation \eqref{23} into equation \eqref{22} to have:
	\begin{eqnarray*}
F_{t} && = \dfrac{\partial F}{\partial t} = \big (x + \omega_{1}(t) \big)^{\delta - 1} \big (\omega_{2}(t) \big)^{1 - \delta}  \omega^{\prime}_{1}(t) + \dfrac{1-\delta}{\delta} \big (x + \omega_{1}(t) \big)^{\delta} \big (\omega_{2}(t) \big)^{-\delta}  \omega^{\prime}_{2}(t)\\
&& = - \bigg (\dfrac{\omega_{2}(t)}{x + \omega_{1}(t)} \bigg)^{1-\delta}f_{1} -  \bigg (\dfrac{1-\delta}{\delta} \bigg)\bigg (\dfrac{x + \omega_{1}(t)}{ \omega_{2}(t)} \bigg)^{\delta}f_{2} \\
F_{x} && = \dfrac{\partial F}{\partial x} = \big (x + \omega_{1}(t) \big)^{\delta - 1} \big (\omega_{2}(t) \big)^{1-\delta} = \bigg (\dfrac{\omega_{2}(t)}{x + \omega_{1}(t) } \bigg)^{1-\delta}\\
F_{\omega_1} && = \dfrac{\partial F}{\partial \omega_1} \Rightarrow \dfrac{\partial F}{\partial t}  = \dfrac{\partial F}{\partial \omega_{1}} \times \dfrac{\partial \omega_1}{\partial t} \\
&& = - \dfrac{1}{f_1} \bigg [- \bigg (\dfrac{\omega_{2}(t)}{x + \omega_{1}(t)} \bigg)^{1-\delta} f_{1} -  \bigg (\dfrac{1-\delta}{\delta} \bigg )\bigg (\dfrac{x + \omega_{1}(t) }{\omega_{2}(t)} \bigg)^{\delta} f_{2}  \bigg] \\
&& = \bigg (\dfrac{\omega_{2}(t) }{x + \omega_1(t)} \bigg)^{1-\delta} + \bigg (\dfrac{1-\delta}{\delta} \bigg) \bigg (\dfrac{x + \omega_{1}(t)}{\omega_{2}(t)} \bigg) \dfrac{f_2}{f_1}\\
F_{\omega_2} && = \dfrac{\partial F}{\partial \omega_2} \Rightarrow \dfrac{\partial F}{\partial t} \\
&& = \dfrac{\partial F}{\partial \omega_2} \times \dfrac{\partial \omega_2}{\partial t}\\
&& = - \dfrac{1}{f_2} \bigg [- \bigg (\dfrac{\omega_{2}(t)}{x + \omega_{1}(t)} \bigg)^{1-\delta} f_{1} - \bigg (\dfrac{1-\delta}{\delta} \bigg) \bigg (\dfrac{x + \omega_{1}(t)}{\omega_{2}(t)} \bigg)^{\delta} f_2\bigg]\\
&&  = \bigg (\dfrac{\omega_{2}(t)}{x + \omega_{1}(t)}\bigg)^{1-\delta} \dfrac{f_1}{f_2} + \bigg (\dfrac{1-\delta}{\delta} \bigg) \bigg (\dfrac{x + \omega_{1}(t)}{\omega_{2}(t) } \bigg)^{\delta}\\
F_{xx} && = \dfrac{\partial^2 F}{\partial x^2} \\
 && = - (1 - \delta)\big (x + \omega_{1}(t) \big)^{\delta -2} \big (\omega_{2}(t) \big)^{1-\delta}\\
 && = -\dfrac{ (1 - \delta) \big (\omega_{2}(t) \big)^{1-\delta}}{\big (x + \omega_{1}(t) \big)^{2-\delta}}\\
F_{x \omega_{1}} && = \dfrac{\partial^2 F}{\partial x \partial \omega_1}= \dfrac{\partial F_x / \partial t}{\partial \omega_1 / \partial t} \\ 
&& = - (1-\delta) \big (x + \omega_{1}(t) \big)^{\delta - 2}\big (\omega_{2}(t) \big)^{1-\delta} + \dfrac{ (1-\delta) \big (\omega_{2}(t) \big)^{-\delta} }{\big (x + \omega_{1}(t) \big)^{1-\delta} } \dfrac{f_2}{f_1}
	\end{eqnarray*}
	\begin{eqnarray*}
	F_{x \omega_{2}} && = \dfrac{\partial^2 F}{\partial x \partial \omega_2} = \dfrac{\partial F_x / \partial t}{\partial \omega_2 / \partial t} \\ 
	&& = - (1-\delta) \big (x + \omega_{1}(t) \big)^{\delta - 2}\big (\omega_{2}(t) \big)^{1-\delta} \dfrac{f_1}{f_2} + \dfrac{ (1-\delta) \big (\omega_{2}(t) \big)^{-\delta} }{\big (x + \omega_{1}(t) \big)^{1-\delta} }\\
	F_{\omega_1 \omega_1} && =  \dfrac{\partial^2 F}{\partial \omega^{2}_{1}} = \dfrac{\partial F_{\omega_1} / \partial t}{\partial \omega_1 / \partial t}\\
	\dfrac{\partial F_{\omega_1} }{\partial t} && = (1 - \delta) \big (x + \omega_{1}(t) \big)^{\delta - 2} \big (\omega_{2}(t) \big)^{1-\delta} f_1 - 2(1 - \delta) \big (x + \omega_{1}(t) \big)^{\delta - 1} \big (\omega_{2}(t) \big)^{-\delta} f_{2} \\
	&&  + (1 -\delta) \big (x + \omega_{1}(t) \big)^{\delta} \big (\omega_{2}(t) \big)^{-\delta - 1}\dfrac{f_{2}^2 }{f_1}\\
	F_{\omega_1 \omega_1} && = - (1 - \delta) \big (x + \omega_{1}(t) \big)^{\delta - 2} \big (\omega_{2}(t) \big)^{1-\delta}  + 2(1 - \delta) \big (x + \omega_{1}(t) \big)^{\delta - 1} \big (\omega_{2}(t) \big)^{-\delta} \dfrac{f_2}{f_1} \\
	&& - (1 - \delta)\big (x + \omega_{1}(t) \big)^{\delta } \big (\omega_{2}(t) \big)^{-\delta-1} \dfrac{f^2_2}{f^2_1}\\
	F_{\omega_2 \omega_2} && =  \dfrac{\partial^2 F}{\partial \omega^{2}_{2}} = \dfrac{\partial F_{\omega_2} / \partial t}{\partial \omega_2 / \partial t}\\ && = - (1 - \delta) \big (x + \omega_{1}(t) \big)^{\delta - 2} \big (\omega_{2}(t) \big)^{1-\delta} \dfrac{f^2_1}{f^2_2} + 2(1 - \delta) \big (x + \omega_{1}(t) \big)^{\delta - 1} \big (\omega_{2}(t) \big)^{-\delta} \dfrac{f_1}{f_2}\\
	&& - (1 - \delta)\big (x + \omega_{1}(t) \big)^{\delta } \big (\omega_{2}(t) \big)^{-\delta-1} \\
	F_{\omega_1 \omega_2} && = \dfrac{\partial^2 F}{\partial \omega_1 \partial \omega_2} = \dfrac{\partial F_{\omega_1} }{\omega_2}\\
	 && = \dfrac{\partial F_{\omega_1}/\partial t }{\partial \omega_2 / \partial t}\\
	 && = -(1-\delta) \big (x + \omega_{1}(t) \big)^{\delta - 2} \big (\omega_{2}(t) \big)^{1-\delta} \dfrac{f_1}{f_2} + 2 \dfrac{(1 - \delta)}{\big (x + \omega_{1}(t) \big)^{1- \delta} \big (\omega_{2}(t) \big)^{\delta}} \\
	&& - (1-\delta) \big (x + \omega_{1}(t) \big)^{\delta} \big (\omega_{2}(t) \big)^{-\delta - 1} \dfrac{f_2}{f_1}.
	\end{eqnarray*}
	It follows that 
	\begin{eqnarray*}
		\dfrac{F_{x}}{F_{xx}} && = - \dfrac{\big (x + \omega_{1}(t) \big)}{1-\delta}\\
		\dfrac{F_{x\omega_1}}{F_{xx}} && = 1 - \dfrac{\big (x + \omega_{1}(t) \big)}{\omega_{2}(t) } \dfrac{f_2}{f_1} \\
		\dfrac{F_{x \omega_2}}{F_{xx}} && = \dfrac{f_1}{f_2} - \dfrac{\big (x + \omega_{1}(t) \big)}{\omega_{2}(t)}.
	\end{eqnarray*}
	Hence
	\begin{eqnarray*}
	&& \eqref{22} \Rightarrow \bigg (\dfrac{\omega_{2}(t)}{x + \omega_{1}(t)} \bigg)^{1-\delta} f_{1} + \bigg (\dfrac{1 - \delta}{\delta}\bigg) \bigg ( \dfrac{x + \omega_{1}(t)}{\omega_{2}(t)} \bigg)^{\delta} f_{2}+ \dfrac{\big (\theta(t) + \mu(t) \big)}{\delta} \bigg (\dfrac{x + \omega_{1}(t)}{\omega_{2}(t)} \bigg)^{\delta} \omega_{2}(t) \\
	&& + \bigg (\dfrac{\omega_{2}(t)}{x + \omega_{1}(t)}\bigg)^{1-\delta} \bigg [ \big (r_t +a(t) \big ) x + R(t) \bigg ] - \bigg (\dfrac{\omega_{2}(t)}{x + \omega_{1}(t)} \bigg)^{1-\delta} f_{1} - \dfrac{\big (1-\delta \big)}{\delta} \bigg (\dfrac{x + \omega_{1}(t)}{\omega_{2}(t)} \bigg)^{\delta}f_{2}\\
	&& - \bigg (\dfrac{\omega_{2}(t)}{x + \omega_{1}(t)} \bigg)^{1-\delta}f_{1} - \dfrac{\big (1 - \delta \big)}{\delta} \bigg (\dfrac{x + \omega_{1}(t) }{\omega_{2}(t)} \bigg)^{\delta}f_{2} - \dfrac{1}{2} \bigg \{ \bigg (\dfrac{\omega_{2}(t)}{x + \omega_{1}(t)} \bigg)^{1-\delta} \big (Uz + u - r_{t} \mathbf{1}\big)^{T} \big (\alpha \mathbb{C}_{B} \alpha^{T}\\
	&& +  \Sigma \mathbb{C}_{W} \Sigma^{T} \big)^{-1} +  \big (1-\delta \big) \bigg (\dfrac{x + \omega_{1}(t)}{\omega_{2}} \bigg)^{\delta}  \bigg [ \dfrac{- \omega_{2}(t)}{\big (x + \omega_{1}(t) \big)^2} + \dfrac{1}{\big (x + \omega_{1}(t)\big)} \dfrac{f_2}{f_1} \bigg ] \lambda_{1}^{T} + 
	\end{eqnarray*}
	\begin{eqnarray*}
	&&  \big (1 - \delta \big ) \bigg (\dfrac{x + \omega_{1}(t) }{\omega_{2}(t)} \bigg)^{\delta} \bigg [\dfrac{- \omega_{2}(t) }{\big (x + \omega_{1}(t)  \big)^2 } \dfrac{f_1}{f_2} + \dfrac{1}{\big (x + \omega_{1}(t) \big)} \bigg] \lambda_{2}^{T} \bigg \} \bigg \{ - \dfrac{1}{\big (1 - \delta \big)}\big (x + \omega_{1}(t) \big)\big (Uz + u - r_{t} \mathbf{1} \big)\\
&&  \bigg (1 - \dfrac{\big (x + \omega_{1}(t) \big)}{\omega_{2}(t) } \dfrac{f_2}{f_1} \bigg) \big (\alpha \mathbb{C}_{B}\alpha^T + \Sigma \mathbb{C}_{W} \Sigma^T \big) \lambda_1 +  \bigg (\dfrac{f_1}{f_2} - \dfrac{\big (x + \omega_{1}(t) \big)}{\omega_{2}(t) }  \bigg) \big (\alpha \mathbb{C}_{B}\alpha^T + \Sigma \mathbb{C}_{W} \Sigma^T \big) \lambda_2 \bigg \} \\
&& - \dfrac{1}{2} \big (1 - \delta \big) \bigg (\dfrac{x + \omega_{1}(t) }{\omega_{2}(t) } \bigg)^{\delta}\bigg [ \dfrac{\omega_{2}(t)}{\big (x + \omega_{1}(t) \big)^2 } - 2 \dfrac{f_2}{f_1} \dfrac{1}{\big (x + \omega_{1}(t) \big)} + \dfrac{1}{\omega_{2}(t)} \dfrac{f^2_2}{f^2_1} \bigg ] \lambda^{T}_{1} \big (\alpha \mathbb{C}_{B}\alpha^T + \Sigma \mathbb{C}_{W}\Sigma^T \big)\lambda_{1}\\
&& - \big (1 - \delta \big) \bigg (\dfrac{x + \omega_{1}(t)}{\omega_{2}(t)} \bigg)^{\delta} \bigg [\dfrac{\omega_{2}(t)}{\big (x + \omega_{1}(t) \big)^{2} } \dfrac{f_1}{f_2}  - \dfrac{2}{\big (x + \omega_{1}(t)\big)} + \dfrac{1}{\omega_{2}(t)} \dfrac{f_2}{f_1}\bigg] \lambda^{T}_{1} \big (\alpha \mathbb{C}_{B} \alpha^{T} +  \Sigma \mathbb{C}_{W} \Sigma^{T} \big)\lambda_{2} \\
&& - \dfrac{1}{2} \big (1 - \delta \big) \bigg (\dfrac{x + \omega_{1}(t)}{\omega_{2}(t)} \bigg)^{\delta}\bigg [ \dfrac{\omega_{2}(t)}{\big (x + \omega_{1}(t)\big)^2 } \dfrac{f_{1}^{2}}{f_{2}^{2}} - \dfrac{2}{\big (x + \omega_{1}(t) \big) } \dfrac{f_1}{f_2} + \dfrac{1}{\omega_{2}(t)} \bigg ] \lambda^{T}_{2} \big (\alpha \mathbb{C}_{B}\alpha^T + \Sigma \mathbb{C}_{W}\Sigma^T \big)\lambda_{2} \\
&&  +  \big (\dfrac{1-\delta}{\delta} \big) \bigg [1 + \dfrac{ (a (t))^{- \frac{\delta}{1-\delta}} }{(\mu(t))^{- \frac{1}{1-\delta}} } \bigg] \bigg (\dfrac{\omega_{2}(t)}{x + \omega_{1} (t)} \bigg)^{-\delta} = 0.\\
&& \Rightarrow - \bigg (\dfrac{\omega_{2}(t) }{x + \omega_{1}(t)} \bigg)^{1-\delta} f_{1} - \dfrac{\big (1 - \delta \big)}{\delta} \bigg (\dfrac{x + \omega_{1}(t)}{\omega_{2}(t)} \bigg)^{\delta} f_{2} + \dfrac{\big (\theta(t) + \mu(t)\big)}{\delta} \bigg (\dfrac{x + \omega_{1}(t)}{\omega_{2}(t)} \bigg)^{\delta} \omega_{2}(t) +\\
&& \bigg (\dfrac{\omega_{2}(t) }{x + \omega_{1}(t)} \bigg)^{1-\delta} \bigg [\big (r_{t} + a(t) \big)x + R(t) \bigg] +  \dfrac{1}{2 \big (1 - \delta \big)} \dfrac{\big (\omega_{2}(t)\big)^{1-\delta} }{\big (x + \omega_{1}(t)\big)^{-\delta}} \big (Uz + u - r_{t} \mathbf{1} \big)^{T} \big (\alpha \mathbb{C}_{B} \alpha^{T} +  \\
&& \Sigma \mathbb{C}_{W} \Sigma^{T} \big)^{-1}  \big (Uz + u - r_{t} \mathbf{1} \big)  - \dfrac{1}{2} \bigg (\dfrac{\omega_{2}(t)}{x + \omega_{1}(t)}\bigg)^{1-\delta} \bigg (1 -\dfrac{\big (x + \omega_{1}(t) \big)}{\omega_{2}(t) } \dfrac{f_2}{f_1} \bigg)   \big (Uz + u - r_{t} \mathbf{1} \big)^{T}\lambda_{1} \\
&& - \dfrac{1}{2} \bigg (\dfrac{\omega_{2}(t)}{x + \omega_{1}(t)} \bigg)^{1-\delta} \bigg (\dfrac{f_1}{f_2} - \dfrac{\big (x + \omega_{1}(t) \big)}{\omega_{2}(t)} \bigg)\bigg (Uz + u - r_t \mathbf{1} \bigg)^{T}\lambda_{2} \\
&& + \dfrac{\big (x + \omega_{1}(t) \big)^{\delta + 1}  }{2 \big (\omega_{2}(t) \big)^{\delta}} \bigg [ - \dfrac{\omega_{2}(t)}{\big (x + \omega_{1}(t) \big)^2} + \dfrac{1}{\big (x + \omega_{1}(t) \big)} \dfrac{f_2}{f_1} \bigg ] \lambda^{T}_{1}\bigg (Uz + u - r_t \mathbf{1} \bigg) \\
&& - \dfrac{1}{2} \big (1 - \delta \big) \bigg (\dfrac{x + \omega_{1}(t)}{\omega_{2}(t)} \bigg)^{\delta} \bigg ( - \dfrac{\omega_{2}(t)}{\big (x + \omega_{1}(t) \big)^2 } + \dfrac{1}{\big (x + \omega_{1}(t) \big)} \dfrac{f_2}{f_1} \bigg ) \bigg (1 -\dfrac{\big (x + \omega_{1}(t)\big)}{\omega_{2}(t)} \dfrac{f_2}{f_1}\bigg )\lambda^{T}_{1}\big (\alpha \mathbb{C}_{B} \alpha^{T} \\
&& +  \Sigma \mathbb{C}_{W} \Sigma^{T} \big)\lambda_{1} - \dfrac{1}{2} \big (1 - \delta \big) \bigg (\dfrac{x + \omega_{1}(t)}{\omega_{2}(t)} \bigg)^{\delta} \bigg [- \dfrac{\omega_{2}(t)}{\big (x + \omega_{1}(t) \big)^2 } + \dfrac{1}{\big (x + \omega_{1}(t)\big)} \dfrac{f_2}{f_1} \bigg] \bigg (\dfrac{f_1}{f_2}\\
&&- \dfrac{\big (x + \omega_{1}(t) \big)}{\omega_{2}(t)} \bigg) \lambda^{T}_{1}\big (\alpha \mathbb{C}_{B} \alpha^{T} +  \Sigma \mathbb{C}_{W} \Sigma^{T} \big)\lambda_{2} + \dfrac{1}{2} \bigg (\dfrac{x + \omega_{1}(t)}{\omega_{2}(t)} \bigg)^{\delta} \bigg [- \dfrac{\omega_{2}(t)}{\big (x + \omega_{1}(t) \big)} \dfrac{f_1}{f_2} + 1 \bigg] \lambda^{T}_{2} \big (Uz + u - \\
&& r_{1} \mathbf{1} \big)- \dfrac{1}{2} \big (1 - \delta \big) \bigg (\dfrac{x + \omega_{1}(t)}{\omega_{2}(t)} \bigg)^{\delta} \bigg [- \dfrac{\omega_{2}(t) }{\big (x + \omega_{1}(t) \big)^2} \dfrac{f_1}{f_2} + \dfrac{1}{\big (x + \omega_{1}(t) \big)} \bigg] \bigg (1 - \dfrac{\big (x + \omega_{1}(t) \big )}{\omega_{2} (t)} \dfrac{f_2}{f_1} \bigg)\lambda^{T}_{2} \big (\alpha \mathbb{C}_{B} \alpha^{T} \\
&& +  \Sigma \mathbb{C}_{W} \Sigma^{T} \big)\lambda_{1} - \dfrac{1}{2} \big (1 - \delta \big) \bigg (\dfrac{x + \omega_{1}(t)}{\omega_{2}(t)} \bigg)^{\delta} \bigg [- \dfrac{\omega_{2}(t) }{\big (x + \omega_{1}(t) \big)^2} \dfrac{f_1}{f_2} + \dfrac{1}{\big (x + \omega_{1}(t) \big)} \bigg] \bigg (\dfrac{f_1}{f_2} - \\
&& \dfrac{\big (x + \omega_{1}(t) \big )}{\omega_{2} (t)} \bigg)\lambda^{T}_{2} \big (\alpha \mathbb{C}_{B} \alpha^{T}  +  \Sigma \mathbb{C}_{W} \Sigma^{T} \big)\lambda_{2}- \dfrac{1}{2} \big (1 - \delta \big) \bigg (\dfrac{x + \omega_{1}(t)}{\omega_{2}(t)} \bigg)^{\delta} \bigg [ \dfrac{\omega_{2}(t) }{\big (x + \omega_{1}(t) \big)^2} -  2 \dfrac{f_2}{f_1}\dfrac{1}{\big (x + \omega_{1}(t) \big)}  \\
&& + \dfrac{1}{\omega_{2}(t)} \dfrac{f^{2}_2}{f^{2}_1}\bigg]\lambda^{T}_{1} \big (\alpha \mathbb{C}_{B} \alpha^{T} +  \Sigma \mathbb{C}_{W} \Sigma^{T} \big)\lambda_{1} -  \big (1 - \delta \big) \bigg (\dfrac{x + \omega_{1}(t)}{\omega_{2}(t)} \bigg)^{\delta} \bigg [ \dfrac{\omega_{2}(t) }{\big (x + \omega_{1}(t) \big)^2} \dfrac{f_1}{f_2} -  2 \dfrac{1}{\big (x + \omega_{1}(t) \big)}\\
&& + \dfrac{1}{\omega_{2}(t)} \dfrac{f_2}{f_1}\bigg]\lambda^{T}_{1} \big (\alpha \mathbb{C}_{B} \alpha^{T} +  \Sigma \mathbb{C}_{W} \Sigma^{T} \big)\lambda_{2}  - \dfrac{1}{2} \big (1 - \delta \big) \bigg (\dfrac{x + \omega_{1}(t)}{\omega_{2}(t)} \bigg)^{\delta} \bigg [ \dfrac{\omega_{2}(t) }{\big (x + \omega_{1}(t) \big)^2} \dfrac{f^2_1}{f^2_2} -  2\dfrac{1}{\big (x + \omega_{1}(t) \big)} \dfrac{f_1}{f_2}  \\
&& + \dfrac{1}{\omega_{2}(t)} \bigg]\lambda^{T}_{2} \big (\alpha \mathbb{C}_{B} \alpha^{T} +  \Sigma \mathbb{C}_{W} \Sigma^{T} \big)\lambda_{2}+ \bigg (\dfrac{1-\delta}{\delta} \bigg) \bigg [ 1 + \dfrac{(a(t))^{-\frac{\delta}{1-\delta}} }{(\mu (t))^{-\frac{1}{1-\delta}} } \bigg] \bigg (\dfrac{x + \omega_{1}(t)}{\omega_{2}(t)} \bigg)^{\delta} = 0.
	\end{eqnarray*}
	\begin{eqnarray*}
	&& \Rightarrow - \bigg (\dfrac{\omega_{2}(t) }{x + \omega_{1}(t)} \bigg)^{1-\delta} f_{1} - \dfrac{\big (1 - \delta \big)}{\delta} \bigg (\dfrac{x + \omega_{1}(t)}{\omega_{2}(t)} \bigg)^{\delta} f_{2} + \dfrac{\big (\theta(t) + \mu(t)\big)}{\delta} \bigg (\dfrac{x + \omega_{1}(t)}{\omega_{2}(t)} \bigg)^{\delta} \omega_{2}(t) + \\
	&& \bigg (\dfrac{\omega_{2}(t) }{x + \omega_{1}(t)} \bigg)^{1-\delta} \bigg [\big (r_{t} + a(t) \big)x + R(t) \bigg] +  \dfrac{1}{2 \big (1 - \delta \big)} \bigg (\dfrac{x + \omega_{1}(t)}{\omega_{2}(t)} \bigg)^{\delta} \omega_{2}(t) \big (Uz + u - r_{t} \mathbf{1} \big)^{T} \big (\alpha \mathbb{C}_{B} \alpha^{T} \\
	&& +  \Sigma \mathbb{C}_{W} \Sigma^{T} \big)^{-1}  \big (Uz + u - r_{t} \mathbf{1} \big)- \dfrac{1}{2}\bigg (\dfrac{\omega_{2}(t)}{x + \omega_{1}(t)} \bigg)^{1-\delta} \big (Uz + u - r_{t}\mathbf{1} \big)^{T}\lambda_{1} + \dfrac{1}{2}\bigg (\dfrac{x + \omega_{1}(t)}{\omega_{2}(t)} \bigg)^{\delta} \dfrac{f_2}{f_1} \big (Uz \\
	&& + u - r_t \mathbf{1} \big)^{T}\lambda_{1} - \dfrac{1}{2}\bigg (\dfrac{\omega_{2}(t)}{x + \omega_{1}(t)} \bigg)^{1-\delta} \dfrac{f_1}{f_2} \big (Uz + u - r_t \mathbf{1} \big)^{T}\lambda_{2} +  \dfrac{1}{2}\bigg (\dfrac{x + \omega_{1}(t)}{\omega_{2}(t)} \bigg)^{\delta} \big (Uz + u - r_t \mathbf{1} \big)^{T}\lambda_{2} \\
	&& - \dfrac{1}{2}\bigg (\dfrac{\omega_{2}(t)}{x + \omega_{1}(t)} \bigg)^{1-\delta} \lambda^{T}_{1} \big (Uz + u - r_{t}\mathbf{1} \big) +  \dfrac{1}{2}\bigg (\dfrac{x + \omega_{1}(t)}{\omega_{2}(t)} \bigg)^{\delta} \dfrac{f_2}{f_1} \lambda^{T}_{1} \big (Uz + u- r_t \mathbf{1} \big)\\
	&&- \dfrac{1}{2} \big (1 - \delta \big) \bigg (\dfrac{x + \omega_{1}(t)}{\omega_{2}(t)} \bigg)^{\delta} \bigg [ - \dfrac{\omega_{2}(t) }{\big (x + \omega_{1}(t) \big)^2} + \dfrac{1}{\big (x + \omega_{1}(t) \big)} \dfrac{f_2}{f_1} + \dfrac{1}{\big (x + \omega_{1}(t) \big)} \dfrac{f_2}{f_1} - \\
	&& \dfrac{1}{\omega_{2}(t)} \dfrac{f^2_2}{f^2_1} \bigg ] \lambda^{T}_{1}  \big (\alpha \mathbb{C}_{B} \alpha^{T} +  \Sigma \mathbb{C}_{W} \Sigma^{T} \big) \lambda_{1} - \dfrac{1}{2} \big (1 - \delta \big) \bigg (\dfrac{x + \omega_{1}(t)}{\omega_{2}(t)} \bigg)^{\delta} \bigg [ - \dfrac{\omega_{2}(t) }{\big (x + \omega_{1}(t) \big)^2 } \dfrac{f_1}{f_2} + \dfrac{2}{\big (x + \omega_{1}(t) \big)} \\
	&& - \dfrac{1}{\omega_{2}(t)} \dfrac{f_2}{f_1} \bigg ] \lambda^{T}_{1}  \big (\alpha \mathbb{C}_{B} \alpha^{T} +  \Sigma \mathbb{C}_{W} \Sigma^{T} \big) \lambda_{2}- \dfrac{1}{2} \bigg (\dfrac{\omega_{2}(t) }{x + \omega_{1}(t) } \bigg)^{1-\delta} \dfrac{f_1}{f_2} \lambda^{T}_{2} \big (Uz + u - r_{t} \mathbf{1} \big) \\
	&& + \dfrac{1}{2} \bigg ( \dfrac{x + \omega_{1}(t) }{\omega_{2}(t)} \bigg)^{\delta} \lambda^{T}_{2} \big (Uz + u - r_1 \mathbf{1} \big) - \\
	&& \dfrac{1}{2}\big (1 -\delta \big) \bigg (\dfrac{x + \omega_{1}(t) }{\omega_{2}(t)} \bigg)^{\delta} \bigg [ -\dfrac{\omega_{2}(t) }{\big (x + \omega_{1}(t) \big)^2} \dfrac{f_1}{f_2} + \dfrac{2}{\big (x + \omega_{1}(t) \big)} - \dfrac{1}{\omega_{2}(t)} \dfrac{f_2}{f_1} \bigg ] \lambda^{T}_{2}  \big (\alpha \mathbb{C}_{B} \alpha^{T} +  \Sigma \mathbb{C}_{W} \Sigma^{T} \big) \lambda_{1}\\
	&& - \dfrac{1}{2}\big (1 -\delta \big) \bigg (\dfrac{x + \omega_{1}(t) }{\omega_{2}(t)} \bigg)^{\delta} \bigg [ -\dfrac{\omega_{2}(t) }{\big (x + \omega_{1}(t) \big)^2} \dfrac{f^2_1}{f^2_2} +\dfrac{2}{\big (x + \omega_{1}(t) \big)}\dfrac{f_1}{f_2} - \dfrac{1}{\omega_{2}(t)} \bigg ] \lambda^{T}_{2}  \big (\alpha \mathbb{C}_{B} \alpha^{T} +  \Sigma \mathbb{C}_{W} \Sigma^{T} \big) \lambda_{2} \\
	&& - \dfrac{1}{2}\big (1 -\delta \big) \bigg (\dfrac{x + \omega_{1}(t) }{\omega_{2}(t)} \bigg)^{\delta} \bigg [\dfrac{\omega_{2}(t) }{\big (x + \omega_{1}(t) \big)^2}  - \dfrac{2}{\big (x + \omega_{1}(t) \big)}\dfrac{f_2}{f_1} +\dfrac{1}{\omega_{2}(t)} \dfrac{f^2_2}{f^2_1} \bigg ] \lambda^{T}_{1}  \big (\alpha \mathbb{C}_{B} \alpha^{T} +  \Sigma \mathbb{C}_{W} \Sigma^{T} \big) \lambda_{1}\\
	&& - \big (1 -\delta \big) \bigg (\dfrac{x + \omega_{1}(t) }{\omega_{2}(t)} \bigg)^{\delta} \bigg [\dfrac{\omega_{2}(t) }{\big (x + \omega_{1}(t) \big)^2} \dfrac{f_1}{f_2}  - \dfrac{2}{\big (x + \omega_{1}(t) \big)} + \dfrac{1}{\omega_{2}(t)} \dfrac{f_2}{f_1} \bigg ] \lambda^{T}_{1}  \big (\alpha \mathbb{C}_{B} \alpha^{T} + \Sigma \mathbb{C}_{W} \Sigma^{T} \big) \lambda_{2} \\
	&&  - \dfrac{1}{2} \big (1 -\delta \big) \bigg (\dfrac{x + \omega_{1}(t) }{\omega_{2}(t)} \bigg)^{\delta} \bigg [\dfrac{\omega_{2}(t) }{\big (x + \omega_{1}(t) \big)^2} \dfrac{f^2_1}{f^2_2}  - \dfrac{2}{\big (x + \omega_{1}(t) \big)} \dfrac{f_1}{f_2} + \dfrac{1}{\omega_{2}(t)} \bigg ] \lambda^{T}_{2}  \big (\alpha \mathbb{C}_{B} \alpha^{T} +  \Sigma \mathbb{C}_{W} \Sigma^{T} \big) \lambda_{2}\\
	&& +  \bigg (\dfrac{1-\delta}{\delta} \bigg) \bigg [ 1 + \dfrac{(a(t))^{-\frac{\delta}{1-\delta}} }{(\mu (t))^{-\frac{1}{1-\delta}} } \bigg] \bigg (\dfrac{x + \omega_{1}(t)}{\omega_{2}(t)} \bigg)^{\delta} = 0.\\
	&& \Rightarrow - \bigg (\dfrac{\omega_{2}(t) }{x + \omega_{1}(t)} \bigg)^{1-\delta} f_{1} - \dfrac{\big (1 - \delta \big)}{\delta} \bigg (\dfrac{x + \omega_{1}(t)}{\omega_{2}(t)} \bigg)^{\delta} f_{2} + \dfrac{\big (\theta(t) + \mu(t)\big)}{\delta} \bigg (\dfrac{x + \omega_{1}(t)}{\omega_{2}(t)} \bigg)^{\delta} \omega_{2}(t) \\
	&& +  \bigg (\dfrac{\omega_{2}(t) }{x + \omega_{1}(t)} \bigg)^{1-\delta} \bigg [\big (r_{t} + a(t) \big)x+ R(t) \bigg] + \\
	&&  \dfrac{1}{2 \big (1 - \delta \big)} \bigg (\dfrac{x + \omega_{1}(t)}{\omega_{2}(t)} \bigg)^{\delta} \omega_{2}(t) \big (Uz + u - r_{t} \mathbf{1} \big)^{T} \big (\alpha \mathbb{C}_{B} \alpha^{T} +  \Sigma \mathbb{C}_{W} \Sigma^{T} \big)^{-1}  \big (Uz + u - r_{t} \mathbf{1} \big) \\
	&& - \bigg (\dfrac{\omega_{2}(t) }{x + \omega_{1}(t)}\bigg)^{1-\delta} \lambda^{T}_{1} \big (Uz + u - r_{t} \mathbf{1} \big) + \bigg (\dfrac{x + \omega_{1}(t)}{\omega_{2}(t)} \bigg)^{\delta}\lambda^{T}_{2} \big (Uz + u - r_t \mathbf{1} \big)\\
	&& +  \bigg (\dfrac{x + \omega_{1}(t)}{\omega_{2}(t)} \bigg)^{\delta}\dfrac{f_2}{f_1} \lambda^{T}_{1} \big (Uz + u - r_t \mathbf{1} \big)-   \bigg (\dfrac{\omega_{2}(t)}{x + \omega_{1}(t)} \bigg)^{1-\delta}\dfrac{f_1}{f_2} \lambda^{T}_{2} \big (Uz + u - r_t \mathbf{1} \big) \\
	&& +  \bigg (\dfrac{1-\delta}{\delta} \bigg) \bigg [ 1 + \dfrac{(a(t))^{-\frac{\delta}{1-\delta}} }{(\mu (t))^{-\frac{1}{1-\delta}} } \bigg] \bigg (\dfrac{x + \omega_{1}(t)}{\omega_{2}(t)} \bigg)^{\delta} = 0.
	\end{eqnarray*}
	\begin{eqnarray*}
	&& \Rightarrow - \bigg (\dfrac{\omega_{2}(t) }{x + \omega_{1}(t)} \bigg)^{1-\delta} f_{1} - \dfrac{\big (1 - \delta \big)}{\delta} \bigg (\dfrac{x + \omega_{1}(t)}{\omega_{2}(t)} \bigg)^{\delta} f_{2} + \dfrac{\big (\theta(t) + \mu(t)\big)}{\delta} \bigg (\dfrac{x + \omega_{1}(t)}{\omega_{2}(t)} \bigg)^{\delta} \omega_{2}(t) \\
	&& +  \bigg (\dfrac{\omega_{2}(t) }{x + \omega_{1}(t)} \bigg)^{1-\delta} \bigg [\big (r_{t} + a(t) \big)x+ \omega_{1}(t) \big (r_t + a(t) \big) - \omega_{1}(t) \big (r_t + a(t) \big) + R(t) \bigg] \\
	&& +  \dfrac{1}{2 \big (1 - \delta \big)} \bigg (\dfrac{x + \omega_{1}(t)}{\omega_{2}(t)} \bigg)^{\delta} \omega_{2}(t) \big (Uz + u - r_{t} \mathbf{1} \big)^{T} \big (\alpha \mathbb{C}_{B} \alpha^{T} +  \Sigma \mathbb{C}_{W} \Sigma^{T} \big)^{-1}  \big (Uz + u - r_{t} \mathbf{1} \big)\\
	&& - \bigg (\dfrac{\omega_{2}(t) }{x + \omega_{1}(t)}\bigg)^{1-\delta} \lambda^{T}_{1} \big ( Uz + u - r_{t} \mathbf{1} \big) + \bigg (\dfrac{x + \omega_{1}(t)}{\omega_{2}(t)} \bigg)^{\delta}\lambda^{T}_{2} \big (Uz + u - r_t \mathbf{1} \big)+  \bigg (\dfrac{x + \omega_{1}(t)}{\omega_{2}(t)} \bigg)^{\delta}\dfrac{f_2}{f_1} \lambda^{T}_{1} \big (Uz \\
	&& + u - r_t \mathbf{1} \big) -   \bigg (\dfrac{\omega_{2}(t)}{x + \omega_{1}(t)} \bigg)^{1-\delta}\dfrac{f_1}{f_2} \lambda^{T}_{2} \big (Uz + u - r_t \mathbf{1} \big) +  \bigg (\dfrac{1-\delta}{\delta} \bigg) \bigg [ 1 + \dfrac{(a(t))^{-\frac{\delta}{1-\delta}} }{(\mu (t))^{-\frac{1}{1-\delta}} } \bigg] \bigg (\dfrac{x + \omega_{1}(t)}{\omega_{2}(t)} \bigg)^{\delta} = 0.\\
	&& \Rightarrow - \bigg (\dfrac{\omega_{2}(t) }{x + \omega_{1}(t)} \bigg)^{1-\delta} f_{1} - \dfrac{\big (1 - \delta \big)}{\delta} \bigg (\dfrac{x + \omega_{1}(t)}{\omega_{2}(t)} \bigg)^{\delta} f_{2} + \dfrac{\big (\theta(t) + \mu(t)\big)}{\delta} \bigg (\dfrac{x + \omega_{1}(t)}{\omega_{2}(t)} \bigg)^{\delta} \omega_{2}(t) \\
	&& + \bigg (\dfrac{x +  \omega_{1}(t)}{\omega_{2}(t) } \bigg)^{\delta} \omega_{2}(t) \bigg [ r_{t} + a(t) \bigg] + \bigg (\dfrac{\omega_{2}(t)}{ x + \omega_{1}(t)} \bigg)^{1-\delta} \bigg [- \omega_{1}(t) \big (r_t + a(t) \big) + R(t) \bigg] \\
	&& +  \dfrac{1}{2 \big (1 - \delta \big)} \bigg (\dfrac{x + \omega_{1}(t)}{\omega_{2}(t)} \bigg)^{\delta} \omega_{2}(t) \big (Uz + u - r_{t} \mathbf{1} \big)^{T} \big (\alpha \mathbb{C}_{B} \alpha^{T} +  \Sigma \mathbb{C}_{W} \Sigma^{T} \big)^{-1}  \big (Uz + u - r_{t} \mathbf{1} \big)\\
	&& - \bigg (\dfrac{\omega_{2}(t) }{x + \omega_{1}(t)}\bigg)^{1-\delta} \lambda^{T}_{1} \big ( Uz + u - r_{t} \mathbf{1} \big) + \bigg (\dfrac{x + \omega_{1}(t)}{\omega_{2}(t)} \bigg)^{\delta}\lambda^{T}_{2} \big (Uz + u - r_t \mathbf{1} \big) \\
	&& +  \bigg (\dfrac{x + \omega_{1}(t)}{\omega_{2}(t)} \bigg)^{\delta}\dfrac{f_2}{f_1} \lambda^{T}_{1} \big (Uz + u - r_t \mathbf{1} \big) -   \bigg (\dfrac{\omega_{2}(t)}{x + \omega_{1}(t)} \bigg)^{1-\delta}\dfrac{f_1}{f_2} \lambda^{T}_{2} \big (Uz + u - r_t \mathbf{1} \big) \\
	&& +  \bigg (\dfrac{1-\delta}{\delta} \bigg) \bigg [ 1 + \dfrac{(a(t))^{-\frac{\delta}{1-\delta}} }{(\mu (t))^{-\frac{1}{1-\delta}} } \bigg] \bigg (\dfrac{x + \omega_{1}(t)}{\omega_{2}(t)} \bigg)^{\delta} = 0.\\
	&&  \Rightarrow \bigg (\dfrac{\omega_{2}(t)}{x + \omega_{1}(t)} \bigg)^{1-\delta}\bigg \{ - f_{1} - \big (r_{t} + a(t)\big) \omega_{1}(t) + R(t) - \lambda^{T}_{1} \big (Uz + u - r_t \mathbf{1} \big) - \dfrac{f_1}{f_2} \lambda^{T}_{2} \big (Uz + u - r_t \mathbf{1} \big) \bigg \} \\
	&& +  \bigg (\dfrac{1-\delta}{\delta} \bigg)\bigg (\dfrac{x + \omega_{1}(t) }{\omega_{2}(t)} \bigg)^{\delta} \bigg \{ - f_{2} + \bigg [1 + \dfrac{(a(t))^{-\frac{\delta}{1-\delta}} }{(\mu(t))^{- \frac{1}{1-\delta}} } \bigg] + \bigg (\dfrac{\delta}{1-\delta} \bigg) \lambda^{T}_{2} \big (Uz + u - r_t \mathbf{1}  \big) \\
	&& +  \bigg (\dfrac{\delta}{1-\delta} \bigg) \dfrac{f_2}{f_1} \lambda^{T}_{1} \big (Uz + u - r_t \mathbf{1}  \big) +  \bigg [ \dfrac{1}{1-\delta} \big (\theta(t) +  \mu(t)\big) + \bigg (\dfrac{\delta}{1-\delta} \bigg) \bigg (r_t  + a(t) \bigg) \\
	&& +   \dfrac{\delta}{2 \big (1 - \delta \big)^2}  \big (Uz + u - r_{t} \mathbf{1} \big)^{T} \big (\alpha \mathbb{C}_{B} \alpha^{T} +  \Sigma \mathbb{C}_{W} \Sigma^{T} \big)^{-1}  \big ( Uz +  u - r_{t} \mathbf{1} \big)\bigg ] \omega_{2}(t) \bigg \} = 0.
	\end{eqnarray*}
	It follows that
	\begin{equation}\label{39}
		f_{1} \bigg (1 + \dfrac{1}{f_2} \lambda^{T}_{2} \big (Uz + u - r_t \mathbf{1} \big) \bigg) = - \big (r_t + a(t) \big) \omega_{1}(t) + R(t) - \lambda^{T}_{1} \big (Uz + u - r_t \mathbf{1} \big)
	\end{equation}
	and 
	\begin{equation}\label{40}
		f_{2} \bigg (1 - \big (\dfrac{\delta}{1-\delta} \big) \dfrac{1}{f_1} \lambda^{T}_{1} \big (Uz + u - r_1 \mathbf{1} \big) \bigg) = H(t) + \bigg (\dfrac{\delta}{1-\delta} \bigg) \lambda^{T}_{2} \big (Uz + u - r_t \mathbf{1} \big) + K(t,z) \omega_{2}(t),
	\end{equation}
	where
	\begin{align*}
		H(t) & = 1 + \dfrac{(a(t))^{-\frac{\delta}{1-\delta}} }{(\mu (t))^{-\frac{1}{1-\delta}} }, \\
		K(t,z) & =  \dfrac{\delta}{2 \big (1 - \delta \big)^2}  \big (Uz + u - r_{t} \mathbf{1} \big)^{T} \big (\alpha \mathbb{C}_{B} \alpha^{T} +  \Sigma \mathbb{C}_{W} \Sigma^{T} \big)^{-1}  \big ( Uz +  u - r_{t} \mathbf{1} \big) +  \dfrac{\big (\theta(t) + \mu(t)\big)}{1-\delta} \\
		& + \bigg (\dfrac{\delta}{1-\delta} \bigg) \bigg (r_t + a(t) \bigg).
	\end{align*}
	\begin{equation}\label{41}
		\eqref{39} \Rightarrow f_{1} = \dfrac{ - \big (r_t + a(t) \big)\omega_{1}(t)  + R(t) - \lambda^{T}_{1}\big (Uz + u - r_t \mathbf{1} \big)  }{ 1 + \dfrac{1}{f_2}\lambda^T_2 \big (Uz + u - r_t \mathbf{1} \big) }.
	\end{equation}
	Then \eqref{41} in \eqref{40}
	\begin{eqnarray*}
	&& \Rightarrow f_{2} \bigg [1 - \dfrac{ \big (\frac{\delta}{1-\delta} \big)\lambda^T_1 \big (Uz + u - r_t \mathbf{1} \big)\big [1 + \frac{1}{f_2} \lambda^T_2 \big (Uz + u - r_{t} \mathbf{1} \big) \big]}{- \big (r_t + a(t) \big)\omega_{1}(t)  + R(t) - \lambda^T_1 \big (Uz + u - r_t \mathbf{1} \big)} \bigg] \\
	&& = H(t)  + \big (\dfrac{\delta}{1-\delta}\big) \lambda^{T}_2 \big (Uz + u - r_t \mathbf{1} \big) + K(t,z) \omega_{2}(t)\\
	&& \\
	&&  f_{2} \bigg [1 - \dfrac{\big (\frac{\delta}{1-\delta} \big) \lambda^T_1 \big (Uz + u - r_t \mathbf{1}\big)}{- \big (r_t + a(t) \big)\omega_{1}(t) + R(t) - \lambda^T_1 \big (Uz + u - r_t \mathbf{1} \big)} \bigg]  \\
	&& \\
	&& = \dfrac{\big (\frac{\delta}{1-\delta} \big) \lambda^T_1 \big (Uz + u - r_t \mathbf{1} \big)  \lambda^T_2 \big (Uz + u - r_t \mathbf{1} \big)}{- \big (r_t + a(t) \big)\omega_{1}(t) + R(t) - \lambda^T_1 \big (Uz + u - r_t \mathbf{1} \big)}	+ H(t) + \big (\dfrac{\delta}{1-\delta} \big) \lambda^T_2 \big (Uz + u - r_t \mathbf{1}\big) + K(t,z)\omega_{2}(t).
	\end{eqnarray*}
	Hence
	\begin{align}\nonumber\label{42}
		f_{2} = & \dfrac{\big (\frac{\delta}{1-\delta}\big) \lambda^T_2 \big (Uz + u - r_t \mathbf{1} \big)\big (- \big(r_t + a(t) \big)\omega_{1}(t) + R(t) \big)  }{ - \big (r_t + a(t) \big)\omega_{1}(t) + R(t) - \big (\frac{1}{1-\delta} \big)\lambda^T_1 \big (Uz + u - r_t \mathbf{1} \big) }  \\ \nonumber
		& \\ \nonumber
		& +  \dfrac{\big (H(t) + K(t,z) \omega_{2}(t)\big) \big (- \big (r_t + a(t) \big)\omega_{1}(t) + R(t) - \lambda^T_1 \big (Uz + u - r_t \mathbf{1} \big) \big)}{- \big (r_t + a(t) \big)\omega_{1}(t) + R(t) - \big (\frac{1}{1-\delta} \big)\lambda^T_1 \big (Uz + u - r_t \mathbf{1} \big) }\\
		&  \approx \bigg (\dfrac{\delta}{1-\delta} \bigg) \lambda^T_2 \big (Uz + u - r_t \mathbf{1} \big) + H(t) + K(t,z) \omega_{2}(t).
	\end{align}
	\eqref{42}  \text{in}  \eqref{41} allow us to get 
	\begin{align}\nonumber\label{43}
		f_{1} & = \dfrac{\big (\frac{\delta}{1-\delta} \big)\lambda^T_2 \big (Uz + u - r_t \mathbf{1} \big) + H(t) + K(t,z) \omega_{2}(t) }{\big (\frac{1}{1-\delta} \big)\lambda^T_2 \big (Uz + u - r_t \mathbf{1} \big) + H(t) + K(t,z) \omega_{2}(t) } \big (- \big (r_t + a(t) \big)\omega_{1}(t) + R(t) \big) \\
		& - \lambda^T_1 \big (Uz + u - r_t \mathbf{1} \big) \dfrac{H(t) + K(t,z) \omega_{2}(t) }{H(t) + K(t,z) \omega_{2}(t) + \big (\frac{1}{1-\delta} \big) \lambda^T_2 \big (Uz + u - r_t \mathbf{1} \big)}\\\nonumber
		&\approx - \big (r_t + a(t) \big) \omega_{1}(t) + R(t) - \lambda^{T}_{1} \big (Uz + u - r_t \mathbf{1} \big).
	\end{align}
	By assuming \textbf{H1)} and the fact that 
	\begin{equation}\label{24}
		f_{1} \big (t, z, \omega_{1}(t), \lambda_{1}(t) \big) = - \big (r_t + a(t) \big) \omega_{1}(t) + R(t) - \lambda^{T}_{1} \big (Uz + u - r_t \mathbf{1} \big)
	\end{equation}
	and 
	\begin{equation} \label{25}
		f_{2} \big (t, z, \omega_{2}(t), \lambda_{2}(t) \big) =  \bigg (\dfrac{\delta}{1-\delta} \bigg) \lambda^T_2 \big (Uz + u - r_t \mathbf{1} \big) + H(t) + K(t,z) \omega_{2}(t),
	\end{equation}
$\square$
\end{proof}
	\nocite{*}
	\bibliographystyle{apa}
	
\end{document}